\documentclass[10pt]{amsart}

\usepackage{amssymb,latexsym, amscd, amsthm, amsmath, upgreek, bm}
\input xypic
\xyoption{all}

\newtheorem{theorem}{Theorem}[section]
\newtheorem{definition}[theorem]{Definition}
\newtheorem{proposition}[theorem]{Proposition}
\newtheorem{lemma}[theorem]{Lemma}

\newtheorem{corollary}[theorem]{Corollary}
\newtheorem{conjecture}[theorem]{Conjecture}

\newtheorem{example}[theorem]{Example}
\newtheorem{remark}[theorem]{Remark}
\newcommand{\bmu}{\boldsymbol{\mu}}
\newcommand{\co}{\mathcal O}
\newcommand{\ck}{\mathcal K}
\newcommand{\cl}{\mathcal L}
\newcommand{\cs}{\mathcal S}
\newcommand{\ct}{\mathcal T}
\newcommand{\fa}{\mathfrak A}
\newcommand{\ca}{\mathcal A}
\newcommand{\cg}{\mathcal G}

\newcommand{\cx}{\mathfrak X}
\newcommand{\cxs}{\mathfrak X_{\cs}}
\newcommand{\z}{\mathbb Z}
\newcommand{\C}{\mathbb C}
\newcommand{\q}{\mathbb Q}
\newcommand{\zp}{\mathbb Z_p}
\newcommand{\qp}{\mathbb Q_p}
\newcommand{\cp}{\mathbb C_p}
\newcommand{\cm}{\mathcal M}

\newcommand{\mk}{\mathcal M^{\ck}_{\cs, \ct}}
\newcommand{\mkp}{\mathcal M^{\ck'}_{\cs', \ct'}}
\newcommand{\kstm}{\ck_{\cs, \ct}^{(m)}}
\newcommand{\kpstm}{\ck_{\cs', \ct'}^{'(m)}}
\newcommand{\kt}{\ck_\ct^\times}
\newcommand{\ktm}{\ck_\ct^{\times m}}
\newcommand{\kpt}{{\ck'}_{\ct'}^\times}
\newcommand{\kptm}{{\ck}_{\ct'}^{'\times m}}
\newcommand{\fp}{\mathbb F_p}
\newcommand{\fit}{\rm{Fit}}
\newcommand{\vu}{\varUpsilon}

\begin{document}

\mathchardef\smallp="3170

\title[An Equivariant Main Conjecture] {{An Equivariant Main Conjecture in Iwasawa Theory and Applications}}
\author[C. Greither and C. D. Popescu]{Cornelius Greither and Cristian D. Popescu}

\address{Institut f\"ur Theoretische Informatik und Mathematik,
Universit\"at der Bundeswehr, M\"unchen,
85577 Neubiberg, Germany}
\email{cornelius.greither@unibw.de}

\address{Department of Mathematics, University of California, San Diego, La Jolla, CA 92093-0112, USA}
\email{cpopescu@math.ucsd.edu}

\keywords{Global and $p$--adic $L$-functions; Iwasawa theory;
Galois module structure; Fitting ideals; Ideal-class groups; \'Etale cohomology; Quillen $K$--theory.}

\subjclass[2010]{11R23, 11R34, 11R33, 11R42, 11R70, 14G40}

\thanks{The second author was partially supported by NSF Grants DMS-0600905 and DMS-0901447}

\begin{abstract} We construct a new class of Iwasawa modules, which are the number field
analogues of the $p$--adic realizations of the Picard $1$--motives constructed by Deligne in \cite{Deligne-HodgeIII}
and studied extensively from a Galois module structure point of view  in our previous work \cite{GP} and \cite{GP2}. We prove that the new
Iwasawa modules are of projective dimension $1$ over the appropriate profinite group rings. In the abelian case, we prove an
Equivariant Main Conjecture, identifying the first Fitting ideal of the Iwasawa module in question over the appropriate
profinite group ring with the principal ideal generated by a certain equivariant $p$--adic $L$--function. This is an integral,
equivariant refinement of the classical Main Conjecture over totally real number fields proved by Wiles in \cite{Wiles}. Finally,
we use these results and Iwasawa co-descent to prove refinements of the (imprimitive) Brumer-Stark Conjecture and the Coates-Sinnott Conjecture, away from their $2$--primary components,
in the most general number field setting. All of the above is achieved under the assumption that the relevant prime $p$ is odd and that the appropriate classical Iwasawa $\mu$--invariants vanish (as conjectured
by Iwasawa.)
\end{abstract}
\maketitle

\section{\bf Introduction}

Let us consider the data $(\ck, \cs, \ct, p)$, where $p$ is an odd prime, $\ck$ is the cyclotomic $\zp$--extension of a number field $K$ (i.e. $\ck$ is a $\zp$--field, in the terminology of \cite{Iwasawa-RH}), and $\cs$ and
$\ct$ are two finite sets of finite primes in $\ck$. We assume that $\ct\cap(\cs\cup\cs_p)=\emptyset$, where $\cs_p$ is the set of $p$--adic primes in $\ck$. Also, let us assume that Iwasawa's $\mu$--invariant
conjecture holds for $K$ and $p$, i.e. we have $\mu_{K,p}=0$, where $\mu_{K,p}$ is the classical Iwasawa $\mu$--invariant associated to $K$ and $p$.

The main goals of this paper are threefold:

{\it Firstly,} for any data $(\ck,\cs, \ct, p)$ as above, we
construct a new class of Iwasawa modules $T_p(\mk)$ which are
$\zp$--free of finite rank (see \S3 below.) We study the $\zp[[\cg]]$--module structure of these modules, under the assumption that
$\ck$ is of CM-type, $\ck$ is a Galois extension
of a totally real number field $k$ of Galois group $\cg:={\rm Gal}(\ck/k)$, the sets $\cs$ and $\ct$ are
$\cg$--invariant, $\ct\ne\emptyset$ and $\cs$ contains the finite primes which ramify in $\ck/k$. In this context, we prove that
\begin{equation}\label{pd-intro}{\rm pd}_{\zp[[\cg]]^-}T_p(\mk)^-=1,\end{equation}
whenever $T_p(\mk)^-\ne 0$. (See Theorem \ref{projective} below.) If $\cg$ is abelian, then the equality above implies that the first Fitting ideal
${\rm Fit}_{\zp[[\cg]]^-} T_p(\mk)^-$ is principal (see Proposition \ref{fitting-calculation}(1) below.)  
This brings us naturally to our next goal, namely
the construction of a canonical generator of this ideal.

{\it Secondly,} working under the additional assumption that $\cg$ is abelian, we construct an equivariant $p$--adic $L$--function
$\Theta_{S,T}^{(\infty)}\in\zp[[\cg]]^-$ (see \S\ref{equivariant-L-functions} below) and prove the following Equivariant Main Conjecture - type equality.
\begin{equation}\label{emc-intro} {\rm Fit}_{\zp[[\cg]]^-} T_p(\mk)^-=\Theta_{S,T}^{(\infty)}\cdot\zp[[\cg]]^-.\end{equation}
(See Theorem \ref{emc} below.) This refines the classical Main Conjecture proved by Wiles in \cite{Wiles} in the following precise sense.

For simplicity,
let us assume that $\ck$ contains the group $\bmu_{p^\infty}$ of $p$--power roots of unity. Then, under our working assumption that $\mu_{K,p}=0$, Wiles's
main theorem in \cite{Wiles} can be restated as follows.
\begin{equation}\label{Wiles-intro}{\rm Fit}_{\qp(\zp[[\cg]]^-)}(\qp\otimes_{\zp} \cxs^+(-1)^\ast)= \mathfrak G_S^{(\infty)}\cdot\qp(\zp[[\cg]]^-),\end{equation}
where $\qp(\zp[[\cg]]^-):=\qp\otimes_{\zp}\zp[[\cg]]^-$, $\cxs$ is the classical Iwasawa module given by the Galois group of the maximal pro-$p$ abelian extension of $\ck$, unramified away from $\cs\cup\cs_p$,
and $\mathfrak G_S^{(\infty)}$ is a canonical element\footnote{The reader can easily check that $\mathfrak G_S^{(\infty)}=(\iota\circ t_1)(G_S)$, with notations as in \S\ref{equivariant-L-functions}.} in $\qp(\zp[[\cg]]^-)$ constructed out of the $S$--imprimitive $p$--adic $L$--functions
associated to $\ck/k$.

A na\"ive integral refinement of Wiles's result would aim for a calculation of ${\rm Fit}_{\zp[[\cg]]^-} (\cxs^+(-1)^\ast)$ in terms of $p$--adic $L$--functions.
Unfortunately, the $\zp[[\cg]]^-$--module $\cxs^+(-1)^\ast$ is only very rarely of finite projective dimension. Consequently, its Fitting ideal
is not principal and difficult to compute, in general.
However, we prove that there is an exact sequence of $\zp[[\cg]]^-$--modules
\begin{equation}\label{classical-intro}\xymatrix{0\ar[r] &T_p(\Delta_{\ck, \ct})^-/\zp(1)\ar[r] &T_p(\mk)^-\ar[r] &\cxs^+(-1)^\ast
\ar[r] &0}.\end{equation}
(See \S\ref{classical} below for the proof and notations.) So, the new modules $T_p(\mk)^-$ are certain extensions of the classical modules $\cxs^+(-1)^\ast$ by $T_p(\Delta_{\ck, \ct})^-/\zp(1)$. These modules
happen to be of projective dimension $1$ over $\zp[[\cg]]^-$.  Also, a consequence of the definition of the equivariant $p$--adic $L$--function $\Theta_{S,T}^{(\infty)}$ is that
\begin{equation}\label{theta-intro}\Theta_{S,T}^{(\infty)}=\mathfrak G_S^{(\infty)}\cdot\mathfrak D_T^{(\infty)},\end{equation}
where
$\mathfrak D_T^{(\infty)}$ is a non zero--divisor distinguished generator\,\footnote{The reader can easily check that $\mathfrak D_T^{(\infty)}=\delta_T^{(\infty)}/(\iota\circ t_1)(H_S)$, with notations as in \S\ref{equivariant-L-functions}.}
of the Fitting ideal ${\rm Fit}_{\qp(\zp[[\cg]]^-)}(T_p(\Delta_{\ck, \ct})^-/\zp(1)\otimes_{\zp}\qp)$.  Since Fitting ideals over $\qp(\zp[[\cg]]^-)$ are multiplicative in short exact sequences, by way of \eqref{classical-intro}--\eqref{theta-intro} our main theorem \eqref{emc-intro} ``tensored with $\qp$'' (i.e. base-changed from the integral group ring $\zp[[\cg]]^-$ to the rational one
$\qp(\zp[[\cg]]^-)$)
is equivalent to Wiles's main theorem \eqref{Wiles-intro}.

{\it Finally}, we use our main theorems (1) and (2) and Iwasawa co-descent to prove refinements of the (imprimitive) Brumer-Stark Conjecture (see Theorem \ref{refined-BS} below)
and the Coates-Sinnott Conjecture (see Theorem \ref{cs-theorem} and Corollary \ref{cs-corollary} below), away from their $2$--primary component, in the most general number field setting.

The main ideas and motivation behind the construction of the new
Iwasawa modules $T_p(\mk)$ are rooted in Deligne's theory of
$1$--motives (see \cite{Deligne-HodgeIII}) and in our previous
work \cite{GP} on the Galois module structure of $p$--adic
realizations of Picard $1$--motives. Also, we were strongly
motivated and inspired by the Deligne-Tate proof of the
Brumer-Stark Conjecture in function fields  via $p$--adic
realizations of Picard $1$--motives. (See Chapter V of
\cite{Tate-Stark}.)

In \cite{GP}, we consider the Picard $1$--motive $\cm_{\cs, \ct}^X$ associated to a smooth, projective curve $X$ defined
over an arbitrary algebraically closed field $\kappa$ and two finite, disjoint sets of closed points $\cs$ and $\ct$ on $X$. We study the Galois module structure of its $p$--adic realizations $T_p(\cm_{\cs, \ct}^X)$ in finite $G$--Galois
covers $X\to Y$ of smooth, projective curves over $\kappa$, in the case where $\cs$ contains the ramification locus of the cover and $\cs$ and $\ct$ are $G$--invariant. In particular, if $\kappa=\overline{\Bbb F_q}$, for some finite field $\Bbb F_q$, and the data $(X\to Y, \cs, \ct)$ is defined over $\Bbb F_q$, with $Y=Y_0\times_{\Bbb F_q}\overline{\Bbb F_q}$, for a smooth, projective curve $Y_0$ over $\Bbb F_q$, then we prove that
\begin{equation}\label{ff-intro}{\rm pd}_{\zp[[\cg]]} T_p(\cm_{\cs,\ct}^X)=1, \qquad   {\rm Fit}_{\zp[[\cg]]} T_p(\cm_{\cs,\ct}^X)=\zp[[\cg]]\cdot\Theta_{S,T}(\gamma^{-1}),\end{equation}
for all primes $p$, where $\cg:={\rm Gal}(\overline{\Bbb F_q}(X)/\Bbb F_q(Y_0))$, $\gamma$ is the $q$--power arithmetic Frobenius (viewed in $\cg$) and $\Theta_{S,T}(\gamma^{-1})$ is an equivariant $L$--function which is the exact function field analogue of $\Theta_{S,T}^{(\infty)}$. Theorem \ref{projective} (equality \eqref{pd-intro} above) is the number field analogue of the first equality in \eqref{ff-intro}, while Theorem \ref{emc} (equality \eqref{emc-intro} above) is the number field
analogue of the second equality in \eqref{ff-intro}.
Wiles's Theorem (equality \eqref{Wiles-intro} above) is equivalent to the number field analogue of Th\'eor\`eme 2.5 in Ch.V of \cite{Tate-Stark}. The second equality in \eqref{ff-intro} is a natural
integral refinement of Th\'eor\`eme 2.5 in loc.cit. Further, in \cite{GP}, we use the second equality in \eqref{ff-intro} and Galois co-descent with respect to ${\rm Gal}(\overline{\Bbb F_q}/\Bbb F_q)$ to prove refinements
of the Brumer-Stark and Coates-Sinnott Conjectures in function fields. In the number field case, these tasks are achieved with almost identical techniques, in Theorems \ref{refined-BS} and \ref{cs-theorem} below.

In \S\ref{abstract-one-motives} below, we define a category of purely algebraic objects which we call ``abstract $1$--motives''. This category is endowed with covariant functors $T_p$ to the category of free $\zp$--modules
of finite rank, called $p$--adic realizations, for all primes $p$. Two particular examples of abstract $1$--motives are Deligne's Picard $1$--motives $\cm_{\cs, \ct}^X$ and their number field analogues
$\mk$, constructed in \S\ref{the-abstract-one-motive}. The Iwasawa module $T_p(\mk)$ of interest in this paper is the $p$--adic realization of $\mk$, and it is the number field analogue of $T_p(\cm_{\cs, \ct}^X)$.

The paper is organized as follows. In \S\ref{algebra-section}, we define the category of abstract $1$--motives and lay down the number theoretic groundwork necessary for constructing the abstract $1$--motive $\mk$.
In \S\ref{Iwasawa-modules-section}, we construct the abstract $1$--motive $\mk$, show that its $p$--adic realizations $T_p(\mk)$ have a natural Iwasawa module structure and link these to the classical Iwasawa modules $\cxs^+(-1)^\ast$, under the appropriate hypotheses. In \S\ref{ct-section}, we prove the projective dimension result, \eqref{pd-intro} above. In \S\ref{EMC}, we prove the Equivariant Main Conjecture, \eqref{emc-intro} above.
In \S\ref{applications-section}, we use the results obtained in the previous sections to prove refined version of the (imprimitive) Brumer-Stark Conjecture and the Coates-Sinnott Conjecture, away from their $2$--primary components, in the general number field setting. The paper has an Appendix (\S7), where we gather various homological algebra results needed throughout.

In upcoming work, we will give further applications of the results
obtained in this paper and its function field companion \cite{GP}
in the following directions: {\bf I.} An explicit construction of
$\ell$--adic models for Tate sequences in the general case of
global fields; {\bf II.} A proof of the Equivariant Tamagawa
Number Conjecture for Dirichlet motives, in full generality for
function fields, and on the ``minus side'' and away from its
$2$--primary part for CM--extensions of totally real number
fields. A proof of the Gross-Rubin-Stark Conjecture (a vast
refined generalization of the Brumer-Stark Conjecture, see
Conjecture 5.3.4 in \cite{Popescu-PCMI}) will ensue in the cases
listed above; {\bf III.} Finally, we will consider non-abelian
versions of the Equivariant Main Conjectures proved in this paper
and \cite{GP}.

\noindent {\bf Acknowledgement.} The authors would like to thank
their home universities for making mutual visits possible. These visits were funded by the DFG, the NSF, and U. C. San Diego.

\section{\bf Algebraic and number theoretic preliminaries}\label{algebra-section}

\subsection{Abstract $1$--motives}\label{abstract-one-motives}
Assume that $J$ is an arbitrary abelian group, $m$ is a strictly positive integer and $p$ is a prime.
We denote by $J[m]$ the maximal $m$--torsion subgroup of $J$ and let $J[p^\infty]:=\cup_mJ[p^m]$. As usual, $T_p(J)$ will denote the $p$--adic Tate module of $J$. By definition, $T_p(J)$ is the $\Bbb Z_p$--module given by
$$T_p(J):=\underset{n}{\underset{\longleftarrow}\lim}\, J[p^n]\,,$$
where the projective limit is taken with respect to the multiplication--by--$p$ maps. There is an obvious canonical isomorphism
of $\Bbb Z_p$--modules
$$T_p(J)\simeq{\rm Hom}_{\Bbb Z_p}(\Bbb Q_p/\Bbb Z_p, J)={\rm Hom}_{\Bbb Z_p}(\Bbb Q_p/\Bbb Z_p, J[p^\infty])\,.$$

Now, let us assume that $J$ is an abelian, divisible group.  $J$ is said to be {\it of finite local corank} if there exists a positive integer $r_p(J)$ and a $\Bbb Z_p$--module isomorphism
$$J[p^\infty]\simeq\left(\Bbb Q_p/\Bbb Z_p\right)^{r_p(J)},$$
for any prime $p$.
The integer $r_p(J)$ is called the $p$--corank of $J$.  Obviously, in this case, we have $T_p(J)\simeq\Bbb Z_p^{r_p(J)}$, for all
primes $p$. Further, one can use the ``Hom''--description of $T_p(J)$ given above combined with the injectivity of the $\Bbb Z_p$--module $J[p^\infty]$ to establish
canonical $\Bbb Z_p$--module isomorphisms
$$T_p(J)/p^n T_p(J)\simeq J[p^n],$$
for all primes $p$ and all $n\in\Bbb Z_{\geq 1}$. (See Remark 3.8 in \cite{GP} for these isomorphisms.)
\begin{definition}\label{define-abstract-one-motive} An abstract $1$--motive $\mathcal M:=[L\overset{\delta}\longrightarrow J]$ consists of the following.
\begin{itemize}
\item A free $\Bbb Z$--module $L$ of finite rank (also called a lattice in what follows);
\item An abelian, divisible group $J$ of finite local corank;
\item A group morphism $\delta: L\longrightarrow J$.
\end{itemize}
\end{definition}
In many ways, an abstract $1$--motive $\mathcal M$ enjoys the same properties as an abelian, divisible group of finite local corank.
Namely, for any $n\in\Bbb Z_{\geq 1}$ and any prime $p$, one can construct an $n$--torsion group $\mathcal M[n]$,
a $p$--divisible $p$--torsion group of finite corank $\mathcal M[p^\infty]$, and a $p$--adic Tate module (or $p$--adic realization)
$T_p(\mathcal M)$, which is $\zp$--free of finite rank. This is done as follows.

For every $n\in\mathbb Z_{\geq 1}$, we take the fiber-product of groups
$J\times^n_J L$, with respect to the map $L\overset{\delta}\longrightarrow J$ and the
multiplication by $n$ map $J\overset
n\longrightarrow J.$ An element in this fiber
product consists of a pair $(j, \lambda)\in J\times L,$
such that $nj=\delta(\lambda)$.
Since $J$ is a divisible group, the map $J\overset
n\longrightarrow J$ is surjective. Consequently, we have a
commutative diagram (in the category of abelian groups) whose rows
are exact.
$$\xymatrix {
0\ar[r] & J[n]\ar[r]\ar[d]^{=} &J\times^n_{J}L\ar[r]\ar[d] & L\ar[r]\ar[d]^{\delta} &0\\
0\ar[r] & J[n]\ar[r] & J\ar[r]^n & J\ar[r] &0}
$$
\begin{definition}\label{define-n-torsion} The group $\mathcal M[n]$ of $n$--torsion
points of $\mathcal M$ is defined by
$$\mathcal M[n]:=(J\times^n_{J}L)\otimes\mathbb Z/n\mathbb Z\,.$$
\end{definition}

\noindent Since $L$ is a free $\mathbb
Z$--module, the rows of the above diagram stay exact when tensored with $\mathbb Z/n\mathbb Z$ and $\z/m\z$, respectively. Consequently,
we have commutative diagrams with exact rows
\begin{equation}\label{motive-n-torsion}\nonumber \xymatrix {
0\ar[r] & J[m]\ar[r]\ar@{>>}[d]^{m/n} &\mathcal M[m]\ar[r]\ar@{>>}[d] &L\otimes\mathbb Z/m\mathbb Z\ar[r]\ar@{>>}[d] &0\\
0\ar[r] & J[n]\ar[r] &\mathcal M[n]\ar[r] &L \otimes\mathbb Z/n\mathbb Z\ar[r] &0\,,
}\end{equation}
for all $n, m\in\mathbb N$, with $n\mid m$, where the left vertical map is multiplication by $m/n$, the right vertical map is the canonical surjection and the middle vertical
map is the unique morphism which makes the diagram commute. This middle vertical morphism maps $(j, \lambda)\otimes\widehat 1$ to $(m/n\cdot j, \lambda)\otimes\widehat 1$ and, in what follows, we
refer to it as the multiplication--by--$m/n$ map $\mathcal M[m]\overset{m/n}\longrightarrow\mathcal M[n]$.

\begin{definition} For a prime $p$, the $p$--adic Tate module $T_p(\mathcal M)$ of $\mathcal M$
is given by
$$T_p(\mathcal M)=\underset{\underset n\longleftarrow}\lim\, \mathcal M[p^n]\,,$$
where the projective limit is taken with respect to the surjective multiplication--by--$p$ maps described above.
\end{definition}
\noindent This way, for every prime $p$, we obtain exact sequences of free $\mathbb Z_p$--modules
$$0\longrightarrow T_{p}(J)\longrightarrow T_{p}(\mathcal M)\longrightarrow L\otimes\mathbb Z_p\longrightarrow 0\,.$$
Clearly, we have an isomorphism of $\Bbb Z_p$--modules $T_p(\mathcal M)\simeq \mathbb Z_p^{(r_p(J)+{\rm rk}_{\mathbb Z}L)}$, where
$r_p(J)$ is the $p$--corank of $J$ and ${\rm rk}_{\mathbb Z}L$ is the $\mathbb Z$--rank of $L$.
\medskip

For $n, m\in\Bbb Z_{\geq 1}$ with $n\mid m$, one also has commutative diagrams with exact rows
\begin{equation}\nonumber \xymatrix {
0\ar[r] & J[m]\ar[r] &\mathcal M[m]\ar[r] &L\otimes\mathbb Z/m\mathbb Z\ar[r] &0\\
0\ar[r] & J[n]\ar[r]\ar@{_{(}->}[u] &\mathcal M[n]\ar[r]\ar@{_{(}->}[u] &L \otimes\mathbb Z/n\mathbb Z\ar@{_{(}->}[u]\ar[r] &0\,,
}\end{equation}
with injective vertical morphisms described as follows: the left--most morphism is the usual inclusion; the right-most morphism
sends $\lambda\otimes\widehat x$ to $\lambda\otimes\widehat{m/n\cdot x}$; the middle morphism sends $(j, \lambda)\otimes\widehat 1$
to $(j, m/n\cdot\lambda)\otimes\widehat 1$. Now, if $p$ is a prime number, we define
$$\mathcal M[p^\infty]:=\underset{m}{\underset{\longrightarrow}\lim}\,\mathcal M[p^m]\,,$$
where the injective limit is taken with respect to the injective maps described above. Obviously, $\mathcal M[p^\infty]$ is a
torsion, $p$--divisible group which sits  in the middle of an exact sequence
$$0\longrightarrow J[p^\infty]\longrightarrow\mathcal M[p^\infty]\longrightarrow L\otimes\mathbb Q_p/\mathbb Z_p\longrightarrow 0\,,$$
and whose $p$--corank equals $(r_p(J)+{\rm rk}_{\mathbb Z}L)$.
It is not difficult to see that there are canonical $\zp$--module isomorphisms
\begin{equation}\label{Tate-torsion} T_p(\mathcal M)\simeq T_p(\mathcal M[p^\infty]),\qquad  T_p(\mathcal M)/p^n T_p(\mathcal M)\simeq \mathcal M[p^n],
\end{equation}
for all $n\in\z_{\geq 1}$.
\begin{remark}\label{abstract-p-adic-1-motives} For a given prime $p$, one can define an {\rm abstract $p$--adic $1$--motive}
$\cm:=[L\overset\delta\longrightarrow J]$ to consist of a $\zp$--module morphism $\delta$ between a free $\zp$--module of finite rank $L$ and a $p$--divisible
$\zp$--module of finite corank $J$. Obviously, the constructions above lead to a $p$--divisible group $\cm[p^\infty]$ and a $p$--adic Tate module
$T_p(\cm)\simeq T_p(\cm[p^\infty])$, for any abstract $p$--adic $1$--motive $\cm$.
Any abstract $1$--motive $\cm:=[L\overset\delta\longrightarrow J]$ gives an abstract $p$--adic $1$--motive
$$\cm_p:=[L\otimes\zp\overset{\delta\otimes\mathbf 1}\longrightarrow J\otimes\zp],$$
for every prime $p$. It is easy to show that there are canonical isomorphisms
$$\cm[p^n]\simeq\cm_p[p^n],\qquad T_p(\cm)\simeq T_p(\cm_p),$$
for all primes $p$ and all $n\in\z_{\geq 1}$.
\end{remark}

\begin{definition}\label{morphism} A morphism $f: \cm'\to\cm$ between two abstract (respectively, abstract $p$--adic) $1$--motives
$\cm':=[L'\overset{\delta'}\longrightarrow J']$ and $\cm:=[L\overset{\delta}\longrightarrow J]$ is a pair
$f:=(\lambda, \iota)$ of group (respectively, $\zp$--module) morphisms
$\lambda: L'\to L$ and $\iota: J'\to J$, such that $\iota\circ\delta'=\delta\circ\lambda$.
\end{definition}

\noindent Now, it is clear what it is meant by composition of morphisms and the category of abstract (respectively, abstract $p$--adic) $1$--motives.
It is easy to see that any morphism $f=(\lambda, \iota)$ as in the definition above induces
canonical $\z/m\z$--module and $\zp$--module morphisms
$$f[m]: \cm'[m]\to \cm[m], \qquad T_p(f): T_p(\cm')\to T_p(\cm)\,,$$
respectively, for all the appropriate $m$ and $p$. Moreover, the associations $f\to f[m]$ and $f\to T_p(f)$
commute with composition of morphisms, so we obtain functors from the category of abstract ($p$--adic) $1$-motives to that of $\z/m\z$--modules
and $\zp$--modules, respectively, for all $m$ and $p$ as above. A consequence is the following.

\begin{remark}\label{equivariance}If $R$ is a ring, $L$ and $J$ are (say, left) $R$--modules and the map $\delta$ is
$R$--linear, then the groups $\cm[m]$, $T_p(\cm)$, $\cm[p^\infty]$ associated to the abstract ($p$--adic) $1$--motive $\mathcal M:=[L\overset{\delta}\longrightarrow J]$ are naturally endowed with $R$--module
structures. Also, all the above exact sequences can be viewed as exact sequences in the category of $R$--modules. This observation will be of particular interest to us
in the case where $R$ is either a group algebra or a profinite group algebra with coefficients in $\z$ or $\zp$.
\end{remark}

\begin{example}\label{Picard} In the case where $J:=\mathcal A(\kappa)$ is the group of $\kappa$--rational points of a semi-abelian variety $\mathcal A$ defined over an algebraically
closed field $\kappa$ and $L$ is an arbitrary lattice then, for any group morphism $\delta: L\to \ca(\kappa)$, the abstract $1$--motive $\mathcal M:=[L\overset{\delta}\longrightarrow J]$
is precisely what Deligne calls a $1$--motive in \cite{Deligne-HodgeIII}.

In particular, let $X$ be a smooth, connected, projective curve defined over an algebraically closed field $\kappa$, and let $\cs$ and $\ct$ be two finite, disjoint
sets of closed points on $X$. Let $J_{\ct}$ be the generalized Jacobian associated to $(X, \ct)$. We remind the reader that $J_{\ct}$ is a semi-abelian variety (an extension of the Jacobian $J_X$ of $X$
by a torus $\tau_\ct$) defined over $\kappa$,
whose group of $\kappa$--rational points is given by
$$J_{\ct}(\kappa)=\frac{{\rm Div}^0(X\setminus\mathcal T)}{\{{\rm div}(f)\mid
f\in\kappa(X)^\times,\quad f(v)=1, \forall\, v\in \ct\}}.$$
Here, ${\rm Div}^0(X\setminus\ct)$ denotes the set of $X$--divisors of degree $0$ supported away from $\ct$ and ${\rm div}(f)$ is the divisor associated to
any non--zero element $f$ in the field $\kappa(X)$ of $\kappa$--rational functions on $X$. The group $J_{\ct}(\kappa)$ is divisible, as an extension of the divisible
groups $J_X(\kappa)$ and $\tau_{\ct}(\kappa)$. Its local coranks satisfy
$$\text{\rm $p$--corank}\,(J_{\ct}(\kappa))=\left\{
                                      \begin{array}{ll}
                                        2g_X+\mid\ct\mid -1, & \hbox{$p\ne{\rm char}(\kappa)$;} \\
                                        \gamma_X, & \hbox{$p={\rm char}(\kappa)$,}
                                      \end{array}
                                    \right.$$
where $g_X$ and $\gamma_X$ are the genus and Hasse--Witt invariant of $X$, respectively. (See \S3 in \cite{GP} for the definitions and the above equality.)

If ${\rm Div}^0(\cs)$ is the lattice of divisors of degree $0$ on the curve $X$
supported on $\cs$ and $\delta: {\rm Div}^0(\cs)\to J_{\ct}(\kappa)$ is the usual divisor--class map, then the abstract $1$--motive
$$\mathcal M^X_{\cs, \ct}:=[{\rm Div}^0(\cs)\overset\delta\longrightarrow J_{\ct}(\kappa)]$$
is called the Picard $1$--motive associated to $(X, \cs, \ct, \kappa)$. Its $\ell$--adic realizations $T_{\ell}(\mathcal M^X_{\cs, \ct})$, for all primes $\ell$,
were extensively studied in \cite{GP}.
\end{example}

In this paper, we will construct a class of abstract $1$--motives which arises naturally in the context of Iwasawa theory of number fields. These should be
viewed as the number field analogues of the Picard $1$--motives $\mathcal M^X_{\cs, \ct}$ described in the example above. As we will see in the next few sections, their $\ell$--adic realizations
satisfy similar properties to those proved to be satisfied by $T_{\ell}(\mathcal M^X_{\cs, \ct})$ in \cite{GP}.

\subsection{Generalized ideal--class groups.}\label{ideal-class-groups}  Let $p$ be a prime number. In what follows, we borrow Iwasawa's terminology (see \cite{Iwasawa-RH})
and call a field $\ck$ a {\it $\zp$--field} if it is the cyclotomic $\zp$--extension of a number field. If a $\zp$--field $\ck$ contains the group
${\bmu}_{p^\infty}$ of $p$--power roots of unity, then it is called a {\it cyclotomic $\zp$--field.} A $\zp$--field $\ck$ is said to be of CM--type if it is
the cyclotomic $\zp$--extension of a CM--number field $K$. This is equivalent to the existence of a (necessarily unique) involution $j$ of $\ck$ (the so--called complex conjugation automorphism of $\ck$),
such that its fixed subfield $\ck^+:=\ck^{j=\mathbf 1}$ is the cyclotomic $\zp$--extension of a totally real number field. The uniqueness of $j$ with these properties makes it commute with any field automorphism
of $\ck$.
\smallskip

For the moment, $\mathcal K$ will denote either a number field or a $\zp$--field, for some fixed prime $p$. As usual, a (finite) prime in $\ck$ is an equivalence class of valuations
on $\ck$. If $\ck$ is a number field, all these valuations are discrete of rank one (with value group isomorphic to $(\,\z,\, +)$ with the usual order.)  If $\ck$ is a $\zp$--field, then its valuations are
either discrete of rank one, if they extend an $\ell$--adic valuation of $\q$, for some prime $\ell\ne p$, or they have value group isomorphic to $(\,\z[1/p],\, +)$ with the usual order, if they
extend the $p$--adic valuation on $\q$. For every $\ck$ and $v$ as above, we pick a valuation  ${\rm ord}_v$ in the class $v$ as follows.

$\bullet$ If $\ck$ is a number field, then ${\rm ord}_v$ is the unique valuation in the class $v$ whose strict value group $\Gamma_v:={\rm ord}_v(\ck^\times)$ satisfies $\Gamma_v=\z$. Note that
if $\ck/\ck'$ is an extension of number fields and $v$ and $v'$ are finite primes in $\ck$ and $\ck'$, with $v$ dividing $v'$, then we have a commutative diagram
\begin{equation}\label{extending-valuations}\xymatrix{\ck^\times\ar[r]^{{\rm ord}_v} &\Gamma_{v}\\
{\ck'}^\times\ar[r]^{\quad{\rm ord}_{v'}}\ar[u]^{\subseteq} &\Gamma_{v'}\ar[u]_{\times\, e(v/v')},}\end{equation}
where $e(v/v'):=[\Gamma_v\,:\,{\rm ord}_v(\ck'^\times)]$ is the usual ramification index.

$\bullet$ If $\ck$ is a $\zp$--field, then we denote by $v_K$ the prime sitting below $v$ in any number field $K$ contained in $\ck$. We let $\Gamma_v:= {\underrightarrow{\lim}}_K \Gamma_{v_K}$,
where the limit is viewed in the category of ordered groups and is taken with respect to all number fields $K$ contained in $\ck$ and the (ordered group morphisms given by) multiplication
by the ramification index maps in the diagram above. Then, since $\ck^\times={\underrightarrow{\lim}}_K\, K^\times$, where the limit is taken with respect to the inclusion morphisms, we can
define
$${\rm ord}_v: \ck^\times \longrightarrow \Gamma_v, \qquad {\rm ord}_v:={\underrightarrow{\lim}}_K {\rm ord}_{{v_K}}\,.$$
It is easy to check that ${\rm ord}_v$ is a valuation on $\ck$ in the class $v$ and that $\Gamma_v={\rm ord}_v(\ck^\times)$. As a consequence of the existence of a number
field $K$, with $K\subseteq\ck$, such that $\ck/K$ is a $\zp$--extension totally ramified at all $p$--adic primes $v$ in $K$ and unramified everywhere else, there are non-canonical ordered group isomorphisms
$$\Gamma_v\simeq\z[1/p],\qquad \Gamma_v\simeq\z,$$
respectively, depending on whether $v$ sits above the $p$--adic prime in $\q$ or not. Also, for any (necessarily finite) extension $\ck/\ck'$ of $\zp$--fields, we have a commutative diagram
identical to (\ref{extending-valuations}) above.
\smallskip

For any field $\ck$ as above, its (additive) divisor group is given by
$${\mathcal Div}_{\ck}:=\bigoplus_v \Gamma_v\cdot v\,,$$
where the direct sum is taken with respect to all the finite primes in $\ck$. Let $\ct$ be a finite (possibly empty) set of finite primes of $\ck$, which is disjoint from the set $\cs_p$ consisting
of all the primes in $\ck$ extending the $p$--adic valuation of $\q$. We let
$${\mathcal Div}_{\ck, \ct}:=\bigoplus_{v\not\in\ct} \Gamma_v\cdot v\,,\qquad \ck_{\ct}^\times:=\{x\in\ck^\times\mid {\rm ord}_v(x-1)>0\,, \forall\, v\in\ct\}\,.$$
The usual divisor map associated to $\ck$ induces a group morphism
$${div}_{\ck}: \ck_{\ct}^\times\longrightarrow {\mathcal Div}_{\ck, \ct}, \quad {div}_{\ck}(x)=\sum_{v}{\rm ord}_v(x)\cdot v=\sum_{v\not\in\ct}{\rm ord}_v(x)\cdot v\,,$$
whose kernel is the following subgroup of the group $U_{\ck}$ of units in $\ck$.
$$U_{\ck, \ct}:=\{u\in U_{\ck}\mid {\rm ord}_v(u-1)>0\,, \text{ for all } v\in\ct \}.$$
To $\ck$ and $\ct$ as above, we associate the following generalized
ideal--class group
$$C_{\ck, \ct}:= \frac{{\mathcal Div}_{\ck, \ct}}{{div}_{\ck}(\ck_{\ct}^\times)}\,.$$

If $\ck/\ck'$ is a finite extension of fields as above and $\ct'$ is the set of primes in $\ck'$ sitting below primes in $\ct$, then we have an injective group morphism
$${\mathcal Div}_{\ck', \ct'}\longrightarrow {\mathcal Div}_{\ck, \ct},\qquad v'\to \sum_{v\mid v'}e(v/v')\cdot v,$$
where the sum is taken over all the primes $v$ of $\ck$ sitting above the prime $v'$ in $\ck'$.  This morphism is compatible with the divisor maps (see diagram (\ref{extending-valuations}))
and it induces a (not necessarily injective)
group morphism  $C_{\ck', \ct'}\to C_{\ck, \ct}$ at the level of generalized ideal class--groups. It is easy to check that
$$\underset K{\underrightarrow{\lim}}\,{\mathcal Div}_{K, T_K}\simeq {\mathcal Div}_{\ck, \ct},\qquad \underset K{\underrightarrow{\lim}}\,C_{K, T_K}\simeq C_{\ck, T_K}\,,$$
where the limits are taken with respect to all number fields $K\subseteq\ck$ and $T_K$ is the set of primes in $K$ sitting below primes in $\ct$.
In particular, let us assume that $\ck$ is the $\zp$--cyclotomic extension of a number field $K$ and write it as a union of number fields $\ck=\cup_n K_n$, where
$K_n$ is the unique intermediate field $K\subseteq K_n\subseteq \ck$ with $[K_n:K]=p^n$. Then, since the set $\{K_n\mid n\geq 1\}$ is cofinal (with respect to inclusion) in the set of all number fields contained in $\ck$, we have group isomorphisms
\begin{equation}\label{limits}\underset n{\underset \longrightarrow{\lim}}\,{\mathcal Div}_{K_n, T_n}\simeq {\mathcal Div}_{\ck, \ct},\qquad \underset n{\underset \longrightarrow{\lim}}\,C_{K_n, T_n}\simeq C_{\ck, \ct}\,,
\end{equation}
where $T_n$ is the set of primes in $K_n$ sitting below those in $\ct$ and the injective limits are taken with respect to the morphisms described above.
\medskip

If $\ct=\emptyset$, then we drop it from the notation arriving this way at the classical group of units $U_{\ck}$ and ideal--class group $C_{\ck}$ associated to $\ck$. Let
$$\Delta_{\ck, \ct}:=\bigoplus_{v\in\ct}\kappa(v)^\times\,,$$
where $\kappa(v)$ denotes the residue field associated to the prime $v$. It is easy to see that for any $\ct$ as above we have an exact sequence of groups
\begin{equation}\label{t-sequence}\xymatrix{
0\ar[r] &U_{\ck}/U_{\ck, \ct}\ar[r] &\Delta_{\ck, \ct}\ar[r] &C_{\ck, \ct}\ar[r] &C_{\ck}\ar[r] &0.}
\end{equation}
If $\ck$ is a number field, the exact sequence above is described in detail in \cite{Rubin-Stark} or \cite{Popescu-PCMI}, for example. If $\ck$ is the cyclotomic $\zp$--extension
of a number field $K$, then the exact sequence above is obtained by taking the injective limit of the corresponding sequences at the
finite levels $K_n$ with respect to the obvious transition maps. If $\ck$ is a number field, (\ref{t-sequence}) shows that $C_{\ck, \ct}$ is finite.
In that case, the Artin reciprocity map associated to $\ck$ establishes an isomorphism between $C_{\ck, \ct}$ and the Galois group of the maximal abelian extension
of $\ck$ which is unramified away from $\ct$ and at most tamely ramified at primes in $\ct$ (see \cite{Rubin-Stark}.) In the case of a $\zp$--field $\ck$, the second isomorphism
in (\ref{limits}) shows that $C_{\ck, \ct}$ is a torsion group.
\smallskip

Throughout the rest of this section, $\ck$ is a $\zp$--field, for some prime $p$. We let
$$\ca_{\ck}:=C_{\ck}\otimes\zp,\qquad \ca_{\ck, \ct}:=C_{\ck, \ct}\otimes\zp\,,$$
for any $\ct$ as above.
A classical theorem of Iwasawa (see \cite{Iwasawa-RH}, pp. 272-273 and the references therein) shows that there is an isomorphism of
groups
$$\ca_{\ck}\simeq(\qp/\zp)^{\lambda_{\ck}}\oplus\ca',$$
where $\lambda_{\ck}$ is a positive integer which depends only on
$\ck$, called the $\lambda$--invariant of $\ck$ and $\ca'$ is a
torsion $\zp$--module of finite exponent. Iwasawa conjectured that
$\ca'$ is trivial (see loc. cit.). This conjecture is equivalent
to the vanishing of the classical Iwasawa $\mu$--invariants
$\mu_{\ck/K}:=\mu_{K,p}$ associated to $p$ and all number fields
$K$ such that $\ck$ is the cyclotomic $\zp$--extension of $K$.
For a given $\zp$--field $\ck$, although
$\mu_{\ck/K}$ depends on the chosen $K$ with the above
properties, its vanishing is independent of that choice.
The vanishing is known to hold if $\ck$ is the cyclotomic
$\zp$--extension of an abelian number field (see
\cite{Ferrero-Washington}.) In what follows, if $\ck$ is a
$\zp$--field, we write $\mu_{\ck}=0$ to mean that $\mu_{\ck/K}=0$,
for all (one) $K$ as above.
\smallskip

Now, assume that $\ck$ is of CM--type and $j$ is its complex conjugation automorphism. Assume that the prime $p$ is odd. For any $\z$--module $M$ endowed with a $j$--action
(called a $j$--module in what follows) and on which multiplication by $2$ is invertible, we let
$M^\pm:=\frac{1}{2}(1\pm j)\cdot M$ denote the $\pm$--eigenspaces of $j$ on $M$. We have a direct sum decomposition
$M =M^-\oplus M^+\,$ and
the two functors $M\to M^\pm$ are obviously exact.
In particular, if we take $M:=\ca_\ck$, we obtain a direct sum decomposition and $\zp$--module isomorphisms
$$\ca_\ck=\ca_{\ck}^-\oplus\ca_{\ck}^+, \quad \ca_\ck^-\simeq(\qp/\zp)^{\lambda_\ck^-}\oplus\ca^{'-}, \quad \ca_\ck^+\simeq(\qp/\zp)^{\lambda_\ck^+}\oplus\ca^{'+},$$
where $\lambda_\ck^{\pm}$ are positive integers such that $\lambda_\ck=\lambda_\ck^-+\lambda_{\ck}^+$, and $\ca^\pm$ are torsion $\zp$--modules of finite exponent, such
that $\ca'=\ca^{'+}\oplus\ca^{'-}$. Also, if the set of primes $\ct$ above is $j$--invariant, then (\ref{t-sequence}) is as an exact sequence in the category
of $j$--modules. Consequently, we obtain two exact sequences of $\zp$--modules:
\begin{equation}\label{t-sequence-pm}\xymatrix{
0\ar[r] &(U_{\ck}/U_{\ck, \ct}\otimes\zp)^\pm\ar[r] &(\Delta_{\ck, \ct}\otimes\zp)^\pm\ar[r] &\ca_{\ck, \ct}^\pm\ar[r] &\ca^\pm_\ck\ar[r] &0.}
\end{equation}

\begin{lemma}\label{class-group-coranks} Let $\ck$ be a $\zp$--field and $\ct$ a finite, non--empty set of finite primes  in $\ck$, disjoint from the set of $p$--adic primes $\cs_p$. Then the following hold.
\begin{enumerate}\item The module $\Delta_{\ck, \ct}\otimes\zp$ is $p$--torsion, divisible, of finite corank denoted $\delta_{\ck, \ct}$.
\item If $\mu_{\ck}=0$, then the module $\ca_{\ck, \ct}$ is $p$--torsion, divisible, of corank at most $\lambda_{\ck}+\delta_{\ck, \ct}$.
\item If $\ck$ is of CM--type, $p$ is odd, $\ct$ is $j$--invariant, and $\mu_{\ck}=0$, then the module $\ca_{\ck, \ct}^{-}$ is $p$-torsion, divisible, of corank
$$\lambda_{\ck}^-+\delta_{\ck, \ct}^--\delta_{\ck},$$
where $\delta_{\ck, \ct}^\pm:={\rm corank}\, (\Delta_{\ck, \ct}\otimes\zp)^\pm$ and $\delta_{\ck}=1$, if $\ck$ is a cyclotomic $\zp$--field and $\delta_\ck=0$, otherwise.
\end{enumerate}
\end{lemma}
\begin{proof} (1) Assume that $\ck$ is the cyclotomic $\zp$--extension of a number field $K$, let $v\in\ct$ and let $v_0$ be the prime in $K$ sitting below $v$.
Since $v\not\in\cs_p$, the prime $v_0$ splits completely up to a finite level $K_n$ in $\ck/K$ and remains inert in $\ck/K_n$. As a consequence, the residue field $\kappa(v)$ associated to $v$ is
a $\zp$--extension of the finite field $\kappa(v_0)$. Consequently, the $p$--primary part $\kappa(v)^\times\otimes\zp$ of the torsion group $\kappa(v)^\times$ either equals the group $\bmu_{p^\infty}$ of all the $p$--power roots of unity or it is trivial, according to whether $\kappa(v_0)^\times$ contains $\bmu_p$ or not. Consequently,
$$\Delta_{\ck, \ct}\otimes\zp\simeq (\qp/\zp)^{\delta_{\ck, \ct}},$$
where $\delta_{\ck, \ct}$ is the number of primes $v\in\ct$, such that $\bmu_p\subseteq\kappa(v)^\times.$ Of course, if $\ck$ is a cyclotomic $\zp$--field, then $\delta_{\ck, \ct}=\mid\ct\mid$. This is because no non--trivial $p$--power root of unity
is congruent to $1$ modulo a non $p$--adic prime $v$, so the reduction mod $v$ maps $\bmu_{p^\infty}\to \kappa(v)$ are injective, for all $v\not\in\cs_p$, in particular for $v\in\ct$.

Part (2) is a consequence of part (1), exact sequence (\ref{t-sequence}), and the fact that the category of $p$--torsion, divisible,
groups of finite corank is closed under quotients and extensions.

Under the hypotheses of part (3), we obviously have
\begin{equation}\label{minus-units}
(U_\ck\otimes\zp)^{-}=\bmu_{p^\infty}\cap \ck^\times\simeq(\qp/\zp)^{\delta_{\ck}}.
\end{equation}
Since $\ct\cap\cs_p=\emptyset$, this implies that $(U_{\ck, \ct}\otimes\zp)^-$ is trivial.
In light of these facts, part (3) is a direct consequence of exact sequence (\ref{t-sequence-pm}).
\end{proof}

Now, assume that $\ck$ is the cyclotomic $\zp$--extension of the number field $K$ and that $\ck$ is of CM--type. Then,
every intermediate field $K_n$ is a CM--number field whose complex conjugation automorphism is the restriction of the complex conjugation $j$
of $\ck$ to $K_n$ and will be denoted by $j$ as well in what follows. For a finite set $\ct$ of finite primes in $\ck$, disjoint from $\cs_p$, we denote by
$T_n$ the set of primes in $K_n$ sitting below primes in $\ct$, for all $n\geq 0$. We let
$$A_{K_n}:=C_{K_n}\otimes \zp, \qquad A_{K_n, T_n}:=C_{K_n, T_n}\otimes\zp\,.$$
If $p$ is odd and the set $\ct$ is $j$--invariant, then we can split these $\zp$--modules
$$A_{K_n}=A_{K_n}^+\oplus A_{K_n}^-, \qquad  A_{K_n, T_n}:=A_{K_n, T_n}^+\oplus A_{K_n, T_n}^- ,$$
into their $\pm$--eigenspaces with respect to the $j$--action.

\begin{lemma}\label{no-capitulation} Under the above hypotheses, the natural maps
$$A_{K_n, T_n}^-\longrightarrow A_{K_{n+1}, T_{n+1}}^-$$
are injective, for all $n\geq 0$.
\end{lemma}
\begin{proof} If $\ct=\emptyset$, this is Prop. 13.26 in \cite{Washington}. The general case follows from the case $\ct=\emptyset$  by a snake lemma argument applied to the ``minus''
exact sequence (\ref{t-sequence-pm}).
\end{proof}
\medskip

\section{The relevant Iwasawa modules}\label{Iwasawa-modules-section}

\subsection{An Iwasawa theoretic abstract $1$--motive}\label{the-abstract-one-motive} In this section, we fix a prime $p$, a $\zp$--field $\ck$ and two finite sets
$\cs$ and $\ct$ of finite primes in $\ck$, with the property that $\ct\cap(\cs\cup\cs_p)=\emptyset.$ {\it From this point on, we assume that $\mu_{\ck}=0$.}
To the data $(\ck, \cs, \ct)$ we associate the abstract $1$--motive
$$\cm_{\cs, \ct}^\ck:=[\,{\mathcal Div}_\ck(\cs\setminus\cs_p)\overset{\delta}\longrightarrow\ca_{\ck, \ct}\,]\,,$$
where ${\mathcal Div}_\ck(\cs\setminus \cs_p)$ is the group of divisors of $\ck$ supported on $\cs\setminus\cs_p$ and $\delta$ is the usual divisor--class map
sending the divisor $D$ into $\widehat D\otimes 1$ in $\ca_{\ck, \ct}=C_{\ck, \ct}\otimes\zp$, where $\widehat D$ denotes the class of $D$ in $C_{\ck, \ct}$. Note that, under our current hypotheses, $\rm{Div}_\ck(\cs\setminus\cs_p)$
is a free $\z$--module of rank $d_{\ck, \cs}:=\mid\cs\setminus\cs_p\mid$ and $\ca_{\ck, \ct}$ is torsion, divisible, of finite local corank. This local corank equals  $0$ at primes $\ell\ne p$ and
equals at most $(\lambda_{\ck}+\delta_{\ck, \ct})$ at $p$, as Lemma \ref{class-group-coranks}(2) shows. Consequently, all the requirements in Definition \ref{define-abstract-one-motive} are met.
The $p$--adic realization $T_p(\cm_{\cs, \ct}^\ck)$ of $\cm_{\cs, \ct}^\ck$ is a free $\zp$--module of
rank at most $(\lambda_\ck + \delta_{\ck, \ct}+ d_{\ck, \cs})$, sitting in an exact sequence
$$\xymatrix{ 0\ar[r] &T_p(\ca_{\ck, \ct})\ar[r] &T_p(\cm_{\cs, \ct}^\ck)\ar[r] & {\mathcal Div}_{\ck}(\cs\setminus\cs_p)\otimes\zp\ar[r] &0\,. }$$
Next, we give a new interpretation of the $p^n$--torsion groups $\cm_{\cs, \ct}^\ck[p^n]$ of the abstract $1$--motive above, for all $n\geq 1$.  For that purpose, we need the following.
\begin{definition} For every $p$, $\ck$, $\cs$, $\ct$ as above and every $m:=p^n$, where $n\in\Bbb Z_{\geq 1}$,
we define the following subgroup of $\ck_{\ct}^\times$.
$$\ck_{\cs, \ct}^{(m)}:=\left\{f\in \ck_{\ct}^\times\mid {div}_{\ck}(f)=mD+y\right\},$$
where $D\in{\mathcal Div}_{\ck, \ct}$ and $y\in{\mathcal Div}_{\ck}(\cs\setminus\cs_p)$. (Note that since the group ${\mathcal Div}_{\ck}(\cs_p)$ is $p$--divisible, we could
write $y\in{\mathcal Div}_{\ck}(\cs)$ or $y\in{\mathcal Div}_{\ck}(\cs\cup\cs_p)$  instead of $y\in{\mathcal Div}_{\ck}(\cs\setminus\cs_p)$ and arrive at an equivalent definition for $\kstm$.)
\end{definition}
\noindent Note that we have an inclusion $\ck_\ct^{\times m}\cdot U_{\ck, \ct}\subseteq \ck_{\cs, \ct}^{(m)}$ of subgroups of $\ck_{\ct}^\times$.

\begin{proposition}\label{reinterpret} For every $p$, $\ck$, $\cs$, $\ct$ and $m$ as in the definition above, we have a canonical group isomorphism
$$\cm_{\cs, \ct}^\ck[m]\simeq \ck_{\cs, \ct}^{(m)}/(\ck_\ct^{\times m}\cdot U_{\ck, \ct})\,.$$
\end{proposition}
\begin{proof} In what follows, in order to simplify notation, we view $\ca_{\ck, \ct}$ as the $p$--Sylow subgroup of the torsion group $C_{\ck, \ct}$. The general element in the fiber--product
$\ca_{\ck, \ct}\times^m_{\ca_{\ck, \ct}}{\mathcal Div}_{\ck}(\cs\setminus\cs_p)$ (see \S\ref{abstract-one-motives} for the definition) consists of a pair
$(\widehat D, x)$, where $D\in{\mathcal Div}_{\ck, \ct}$, such that its class $\widehat D$ in $C_{\ck, \ct}$ lies in $\ca_{\ck, \ct}$, and $x\in{\mathcal Div}_{\ck}(\cs\setminus\cs_p)$
with the property that $mD-x={div}_{\ck}(f)$, for some $f\in\ck_{\ct}^\times$. Note that, by definition, any such $f$ lies in $\ck_{\cs, \ct}^{(m)}$.
We define a map
$$\ca_{\ck, \ct}\times^m_{\ca_{\ck, \ct}}{\mathcal Div}_{\ck}(\cs\setminus\cs_p)\overset{\phi}\longrightarrow \ck_{\cs, \ct}^{(m)}/(\ck_\ct^{\times m}\cdot U_{\ck, \ct}), \qquad \phi(\widehat D, x)=\widehat f, $$
where $D$, $x$ and $f$ are as above, and $\widehat f$ is the class of $f$ in the quotient to the right. We claim that $\phi$ is a well-defined, surjective group morphism. We check this next.
\smallskip

{\bf Step 1. $\phi$ is a well defined group morphism.} Assume that $D, D'\in{\mathcal Div}_{\ck, \ct}$ and $x\in{\mathcal Div}_{\ck}(\cs\setminus\cs_p)$, such that $(\widehat D, x)$ and $(\widehat D', x)$ belong to and are equal in
the fiber product above. This means that
$${div}_{\ck}(f)=mD-x,\qquad {div}_{\ck}(f')=mD'-x, \qquad D-D'={div}_{\ck}(g),$$
for some $f, f', g\in \ck_{\ct}^\times$. Obviously, this implies that ${div}_{\ck, \ct}(f)={div}_{\ck, \ct}(f'g^m)$. Consequently, there exists $u\in U_{\ck, \ct}$, such that
$f=f'\cdot g^mu$. This implies that $\phi(\widehat D, x)=\widehat f=\widehat f'=\phi(\widehat D', x)$, showing that $\phi$ is well-defined as a function.
The fact that $\phi$ is a group morphism is obvious.
\smallskip

{\bf Step 2. $\phi$ is surjective.} Let $f\in\ck_{\cs, \ct}^{(m)}$ and $D\in{\mathcal Div}_{\ck, \ct}$ and $x\in{\mathcal Div}_{\ck}(\cs\setminus\cs_p)$, such that ${div}_{\ck}(f)=mD-x$. Since $C_{\ck, \ct}$ is a torsion group, there exists
a natural number $a$, with $\gcd(a, m)=1$, such that $a\widehat D\in\ca_{\ck, \ct}$. This shows that $(a\widehat D, -ax)\in \ca_{\ck, \ct}\times^m_{\ca_{\ck, \ct}}{\mathcal Div}(\cs\setminus\cs_p)$ and $\phi(a\widehat D, -ax)={\widehat f}^{\,\,a}$. Consequently, ${\widehat f}^{\,\,a}\in{\rm Im}(\phi)$. However, since ${\rm Im}(\phi)$ is a group of exponent dividing $m$ (as a subgroup of $\ck_{\cs, \ct}^{(m)}/\ck_\ct^{\times m}\cdot U_{\ck, \ct}$), it is
uniquely $a$--divisible. This shows that $\widehat f\in {\rm Im}(\phi)$, which concludes the proof of the surjectivity of $\phi$.
\smallskip

{\bf Step 3. The kernel of $\phi$.} Next, we prove that
$$\ker(\phi)=m(\ca_{\ck, \ct}\times^m_{\ca_{\ck, \ct}}{\mathcal Div}_{\ck}(\cs\setminus\cs_p)).$$
Let $(\widehat D, x)\in \ker(\phi)$. This means that there exists $g\in\ck_{\ct}^\times$, such that
$$mD-x={div}_{\ck}(g^m)=m\cdot {div}_{\ck}(g)\,.$$
This implies that $x=mx'$, for some $x'\in{\mathcal Div}_{\ck}(\cs\setminus\cs_p)$. However, since $\ca_{\ck, \ct}$ is $m$--divisible, there exists
$D'\in{\mathcal Div}_{\ck, \ct}$, such that $\widehat D'\in\ca_{\ck, \ct}$ and $\widehat D=m\widehat D'$. This means that
$D-mD'={div}_{\ck}(f')$, for some $f'\in\ck_{\ct}^\times$.
This shows that
$$(\widehat D, x)=m(\widehat D', x'),\qquad mD'-x'={div}_{\ck}(gf')\,.$$
Therefore, $\ker(\phi)\subseteq m(\ca_{\ck, \ct}\times^m_{\ca_{\ck, \ct}}{\mathcal Div}(\cs\setminus\cs_p))$. The opposite inclusion is obvious.
\smallskip

Now, we combine Steps 1--3 above with the definition of $\cm_{\cs, \ct}^\ck[m]$ (see \S\ref{abstract-one-motives}) to conclude that the map $\phi$ factors
through a group isomorphism $\widetilde\phi$
$$\xymatrix{\ca_{\ck, \ct}\times^m_{\ca_{\ck, \ct}}{\mathcal Div}_{\ck}(\cs\setminus\cs_p)\ar@{>>}[r]^{\quad\phi}\ar@{>>}[d] &\ck_{\cs, \ct}^{(m)}/(\ck_\ct^{\times m}\cdot U_{\ck, \ct})\\
\cm_{\cs, \ct}^\ck[m]\ar[ur]^{\widetilde\phi}_{\sim} &
}$$
This concludes the proof of the proposition.
\end{proof}

\begin{remark} If $G$ is a group of field automorphisms of $\ck$ and the sets $\cs$ and $\ct$ are $G$--equivariant, then $T_p(\cm_{\cs, \ct}^\ck)$ and $\cm_{\cs, \ct}^\ck[p^n]$ are endowed with natural
$\zp[G]$--module structures (see Remark \ref{equivariance}.) In this case, it is easily seen that
the canonical isomorphism in Proposition \ref{reinterpret} is $\zp[G]$--linear.

If, in addition, $\ck$ is of CM--type, $p$ is odd, and $\cs$ and $\ct$ are $j$--invariant, where $j$ is the complex conjugation automorphism of $\ck$,
then we can talk about the eigenspaces $T_p(\cm_{\cs, \ct}^\ck)^\pm$ and $\cm_{\cs, \ct}^\ck[p^n]^\pm$, for all $n\geq 1$. Since the actions of $G$ and $j$ commute, these eigenspaces have natural $\zp[G]$--module structures. Moreover, the isomorphism in Proposition \ref{reinterpret} induces two $\zp[G]$--linear isomorphisms
$$\cm_{\cs, \ct}^\ck[m]^\pm\simeq (\ck_{\cs, \ct}^{(m)}/\ck_\ct^{\times m}\cdot U_{\ck, \ct})^\pm\,,$$
for all $m$ which are powers of $p$. This follows immediately from the fact that the action of $j$ and $G$ on $\ck$ commute, for any $G$ as above.
\end{remark}

\begin{corollary}\label{minus-reinterpret} If $\ck$ is of CM--type, $p$ is odd and $\cs$ and $\ct$ are $j$--invariant, then the isomorphism in Proposition \ref{reinterpret} induces isomorphisms
of $\Bbb Z_p$--modules
$$\cm_{\cs, \ct}^\ck[m]^-\simeq (\ck_{\cs, \ct}^{(m)}/\ck_\ct^{\times m}\cdot U_{\ck, \ct})^-\simeq\left(\ck_{\cs, \ct}^{(m)}/\ck_\ct^{\times m}\right)^-\,,$$
for all $m:=p^n$, where $n\geq 1$. If, in addition, $\cs$ and $\ct$ are $G$--invariant, for a group of automorphisms $G$ of $\ck$, then the above
isomorphisms are $\zp[G]$--linear.
\end{corollary}
\begin{proof} Let $m$ be as above. There is an obvious exact sequence of $\zp[\langle j\rangle ]$--modules
$$\xymatrix{
U_{\ck, \ct}/U_{\ck, \ct}^m\ar[r] & \ck_{\cs, \ct}^{(m)}/\ck_\ct^{\times m}\ar[r] & \ck_{\cs, \ct}^{(m)}/(\ck_{\ct}^{\times m}\cdot U_{\ck, \ct})\ar[r] &0.
}$$
Now, (\ref{minus-units}) above shows that $(U_{\ck, \ct}\otimes\zp)^-=\bmu_{p^\infty}\cap U_{\ck, \ct}$. If $\ct\ne\emptyset$, this intersection is trivial, whereas if $\ct=\emptyset$, this intersection
is either equal to $\mu_{p^\infty}$ or trivial, depending on whether $\ck$ is a cyclotomic $\zp$--field or not. In all these cases, $(U_{\ck, \ct}\otimes\zp)^-$ is $p$--divisible.
Consequently, $(U_{\ck, \ct}/U_{\ck, \ct}^m)^-\simeq (U_{\ck, \ct}\otimes\zp)^-\otimes\z/m\z=0$. Consequently, the exact sequence above leads to a group isomorphism
$$(\ck_{\cs, \ct}^{(m)}/\ck_\ct^{\times m})^-\simeq (\ck_{\cs, \ct}^{(m)}/\ck_{\ct}^{\times m}\cdot U_{\ck, \ct})^-.$$
Now, the first part of the corollary is a direct consequence of Proposition \ref{reinterpret}. The $\zp[G]$--linearity is a consequence of the previous Remark combined with
the obvious $\zp[G]$--linearity of the last displayed isomorphism.
\end{proof}

\begin{remark}\label{reinterpret-extensions} Assume that $\ck/\ck'$ is an extension of $\zp$--fields, and $\cs$ and $\ct$ are sets of primes in $\ck$ as above. Let $\cs'$, $\ct'$ and $\cs_p'$ denote
the sets of primes in $\ck'$ sitting below primes in $\cs$, $\ct$ and $\cs_p$, respectively. Then, the inclusion $\ck'\to\ck$ induces natural group morphisms ${\mathcal Div}_{\ck'}(\cs'\setminus\cs_p')\to
{\mathcal Div}_{\ck}(\cs\setminus\cs_p)$ and $\ca_{\ck', \ct'}\to\ca_{\ck, \ct}$. These lead to a morphism of abstract $1$--motives $\mkp\to\mk$. 
Since the isomorphisms constructed in Proposition \ref{reinterpret} are functorial, we get commutative diagrams
$$\xymatrix{\mk[m]\ar[r]^{\sim\qquad } &\kstm/(\ktm\cdot U_{\ck, \ct})\\
\mkp[m]\ar[u]\ar[r]^{\sim\qquad } &\kpstm/(\kptm\cdot U_{\ck', \ct'}).\ar[u]
}
$$
\end{remark}

\medskip

\subsection{The ensuing Iwasawa modules} Let us assume that $\ck$ is the cyclotomic $\zp$--extension of a number field $K$. As usual,
we let $\Gamma:={\rm Gal}(\ck/K)$. We pick two finite sets $\cs$ and $\ct$ of finite primes in $\ck$, such that $\ct\cap(\cs\cup\cs_p)=\emptyset$. We assume that
$\cs$ and $\ct$ are $\Gamma$--invariant.

For all $n\in\z_{\geq 0}$, we let $\Gamma_n:={\rm Gal}(\ck/K_n)$, where $K_n$ is the unique intermediate field $K\subseteq K_n\subseteq\ck$,
such that $[K_n:K]=p^n$. Also, we let
$$\Lambda=\zp[[\Gamma]]:=\underset{n}{\underleftarrow{\lim}}\,\zp[\Gamma/\Gamma_n]$$
denote the $p$--adic profinite group ring associated to $\Gamma$, where the projective limit is taken with
respect to the surjections $\Gamma/\Gamma_{n+1}\twoheadrightarrow\Gamma/\Gamma_n$ induced by Galois restriction.

As in \S\ref{ideal-class-groups}, we denote by $T_n$, $S_n$ and $S_{p,n}$ the sets of primes in $K_n$ which sit below primes in $\cs$, $\ct$, and $\cs_p$,
respectively, for all $n\geq 0$. Obviously, $S_n$, $T_n$ and $S_{p, n}$ are $\Gamma/\Gamma_n$--invariant. Consequently, $A_{K_n, T_n}$ and ${\mathcal Div}_{K_n}(S_n\setminus S_{p, n})\otimes\zp$ have canonical
$\zp[\Gamma/\Gamma_n]$--module structures. Via the ring morphisms $\Lambda\twoheadrightarrow\zp[\Gamma/\Gamma_n]$, they are endowed with canonical $\Lambda$--module structures.
The natural morphisms
$$A_{K_n, T_n}\longrightarrow A_{K_{n+1}, T_{n+1}}, \qquad {\mathcal Div}_{K_n}(S_n\setminus S_{p,n})\otimes\zp\longrightarrow {\mathcal Div}_{K_{n+1}}(S_{n+1}\setminus S_{p,n+1})\otimes\zp$$
are $\Lambda$--linear. This leads to canonical $\Lambda$--module structures for
$$\ca_{\ck, \ct}\simeq \underset{n}{\underrightarrow\lim}\, A_{K_n, T_n}\quad \text{ and }\quad {\mathcal Div}_{\ck}(\cs\setminus\cs_p)\otimes\zp\simeq
\underset{n}{\underrightarrow\lim}\,({\mathcal Div}_{K_n}(S_n\setminus S_{p,n})\otimes\zp),$$
as direct limits of $\Lambda$--modules with respect to $\Lambda$--linear transition maps. This way, $T_p(\ca_{\ck, \ct})$ inherits
a canonical $\Lambda$--module structure as well.

\begin{lemma}\label{lambda-module-lemma} With notations as above, the $\zp$--module $T_p(\mk)$ can be endowed with a natural $\Lambda$--module structure which makes the sequence
$$\xymatrix{ 0\ar[r] &T_p(\ca_{\ck, \ct})\ar[r] &T_p(\cm_{\cs, \ct}^\ck)\ar[r] & {\mathcal Div}_{\ck}(\cs\setminus\cs_p)\otimes\zp\ar[r] &0}$$
exact in the category of $\Lambda$--modules.
\end{lemma}

\begin{proof}

%\begin{remark} Assume that $\delta: A\to B$ is a morphism of abelian groups, where $A$ is a free $\z$--module and $B$ is $p$--torsion, $p$--divisible, for some prime number $p$.  Then, $B$ has a natural $\zp$--module structure and %we can extend $\delta$ by $\zp$--linearity to
%a $\zp$--module morphism $\delta\otimes 1: A\otimes\zp\to B$. If $m$ is a power of $p$, we can consider the usual fiber products
%\begin{eqnarray}
 %   \nonumber B\times^m_BA & = &\{(b, a)\mid b\in B, a\in A,\,\, \delta(a)=mb\}\\
  %  \nonumber B\times^m_B(A\otimes\zp)& = &\{(b, \alpha)\mid b\in B, \alpha\in A\otimes\zp,\,\, \delta\otimes 1(\alpha)=mb\}.
%\end{eqnarray}
%It is easy to see that the natural group morphism $B\times^m_BA \to B\times^m_B(A\otimes\zp)$ sending $(b, a)\to (b, a\otimes 1)$ becomes
%a $\zp$--module isomorphism
%$$(B\times^m_BA) \otimes \z/m\z \simeq (B\times^m_B(A\otimes\zp))\otimes \z/m\z$$ after extending scalars to $\z/m\z$. Indeed, we have a commutative diagram
%with exact rows in the category of $\zp$--modules
%$$\xymatrix{0\ar[r] &B[m]\ar[r]\ar[d]^{=} &(B\times^m_BA) \otimes \z/m\z\ar[r]\ar[d] & A\otimes\z/m\z\ar[r]\ar[d]^{\wr} &0\\
%0\ar[r] &B[m]\ar[r] & (B\times^m_B(A\otimes\zp))\otimes \z/m\z\ar[r] & (A\otimes\zp)\otimes\z/m\z\ar[r] &0
%}$$
%whose right and left vertical maps are isomorphisms. Therefore, the middle vertical map is an isomorphism as well.
%\end{remark}

Remark \ref{abstract-p-adic-1-motives} shows that in defining the torsion groups $\mk[m]$, with $m$ a power of $p$, and the $p$--adic Tate module $T_p(\mk)$, we may replace the
abstract $1$--motive $\mk$ with the associated abstract $p$--adic $1$--motive
$$(\mk)_p:=[{\mathcal Div}_{\ck}(\cs\setminus\cs_p)\otimes\zp\overset{\delta\otimes 1}\longrightarrow\ca_{\ck, \ct}],$$
where $\delta\otimes 1$ is the extension of $\delta$ by
$\zp$--linearity to ${\mathcal Div}_{\ck}(\cs\setminus\cs_p)\otimes\zp$. (Note that $\ca_{\ck, \ct}\simeq \ca_{\ck, \ct}\otimes\zp$.)  Now, we obviously have
$$[{\mathcal Div}_{\ck}(\cs\setminus\cs_p)\otimes\zp\overset{\delta\otimes 1}\longrightarrow\ca_{\ck, \ct}]=\underset{n}{\underrightarrow\lim}\,[{\mathcal Div}_{K_n}(S_n\setminus S_{p,n})\otimes\zp\overset{\delta_n\otimes 1}\longrightarrow A_{K_n, T_n}]\,$$
where $\delta_n$ is the usual divisor--class map
sending the divisor $D\in {\mathcal Div}_{K_n}(S_n\setminus S_{p,n})$ to $\widehat D\otimes 1$ in $C_{K_n, T_n}\otimes\zp=A_{K_n, T_n}$, with $\widehat D$ denoting the class of $D$ in $C_{K_n, T_n}$. Since the maps
$\delta_n\otimes 1$ are $\zp[\Gamma/\Gamma_n]$--linear, they are $\Lambda$--linear and therefore $\delta\otimes 1$ is $\Lambda$--linear. Remark \ref{equivariance} applied to ${(\mk)_p}$ and $R:=\Lambda$,
implies that $\mk[m]\simeq (\mk)_p[m]$ can be endowed with a natural
$\Lambda$--module structure which makes the sequence
$$\xymatrix{0\ar[r] &\ca_{\ck, \ct}[m]\ar[r] &\mk[m]\ar[r] &{\mathcal Div}_{\ck}(\cs\setminus\cs_p)\otimes\z/m\z\ar[r] & 0
}$$
exact in the category of $\Lambda$--modules, for all $m$ which are powers of $p$.  The statement in the Lemma is now obtained by taking the projective limit
of the exact sequences above for $m=p^a$, with $a\in\z_{\geq 1}$, with respect to the multiplication-by-$p$ maps, which are clearly $\Lambda$--linear.
\end{proof}

\begin{remark}\label{lambdag} Assume that $K/k$ is a Galois extension of number fields, of Galois group $G$. Let $\ck$ and  be the cyclotomic $\zp$--extensions of $K$,
for some prime $p$.  Then, the extension $\ck/k$ is Galois. Let $\cg:={\rm Gal}(\ck/k)$.
Now, let us assume that the sets $\cs$ and $\ct$ in $\ck$ considered above are $\cg$--invariant. Then, just as above, ${\mathcal Div}_{\ck}(\cs\setminus\cs_p)\otimes\zp$ and $\ca_{\ck, \ct}$ are endowed
with obvious canonical $\zp[[\cg]]$--module structures and the map $\delta\otimes 1$ above is $\zp[[\cg]]$--linear. Consequently, $\mk[p^a]$, for $a\geq 1$, and $T_p(\mk)$ are endowed with a canonical $\zp[[\cg]]$--module structure
and the sequence in Lemma \ref{lambda-module-lemma} is exact in the category of $\zp[[\cg]]$--modules.

If, in addition, $K$ is a CM--number field (which makes $\ck$ of CM--type), $\cs$ and $\ct$ are $j$--invariant
and $p$ is odd, then  $\mk[p^a]^\pm$, for $a\geq 1$, and $T_p(\mk)^\pm$ have natural $\zp[[\cg]]$--module structures.
The exact sequence of $\zp[[\cg]]$--modules in Lemma \ref{lambda-module-lemma} can be split into a direct sum of two exact sequences of $\zp[[\cg]]$--modules
\begin{equation}\label{motive-exact-sequence}
\xymatrix{ 0\ar[r] &T_p(\ca_{\ck, \ct})^\pm\ar[r] &T_p(\cm_{\cs, \ct}^\ck)^\pm\ar[r] & {\mathcal Div}_{\ck}(\cs\setminus\cs_p)^\pm\otimes\zp\ar[r] &0}.\end{equation}
\end{remark}

\begin{remark} For every $n\geq 0$, let $M_{S_n, T_n}^{K_n}:=[{\mathcal Div}_{K_n}(S_n\setminus S_{p,n})\overset{\delta_n}\longrightarrow A_{K_n, T_n}]$, where $\delta_n$ is the divisor--class map defined in the proof of the above Lemma. Of course, $M_{S_n, T_n}^{K_n}$ is not an abstract $1$--motive,
unless the finite group $A_{K_n, T_n}$ happens to be trivial. Nevertheless, one can define the $\zp[\Gamma/\Gamma_n]$--modules
$$M_{S_n, T_n}^{K_n}[m]:=(A_{K_n, T_n}\times^m_{A_{K_n, T_n}}{\mathcal Div}_{K_n}(S_n\setminus S_{p,n}))\otimes\z/m\z,$$
for every $m$ which is a power of $p$. For all $m$ and $n$ as above, we have obvious and canonical $\Lambda$--module morphisms
$$M_{S_n, T_n}^{K_n}[m]\to M_{S_{n+1}, T_{n+1}}^{K_{n+1}}[m]\to \mk[m].$$
It is easily proved that these lead to a $\Lambda$--module isomorphism
$$\underset{n}{\underrightarrow\lim}\, M^{K_n}_{S_n, T_n}[m] \simeq \mk[m].$$

\end{remark}

\subsection{Linking $T_p(\cm_{\cs, \ct})^-$ to the classical Iwasawa modules}\label{classical} In this section, we assume that $K/k$ is a Galois extension of number fields,
where $K$ is CM and $k$ is totally real. We fix a prime $p>2$ and
let $\ck$ denote the cyclotomic $\zp$--extension of $K$. We let $G:={\rm Gal}(K/k)$, $\cg:={\rm Gal}(\ck/k)$, $\Gamma:={\rm Gal}(\ck/K)$ and $\Lambda:=\zp[[\Gamma]]$.

Let $\cs$ be a $\cg$--equivariant finite set of finite primes in $\ck$. As in classical Iwasawa theory, we denote by $\cxs$ the Galois group  of the maximal abelian pro--$p$
extension of $\ck$ which is unramified away from $\cs\cup\cs_p$. Then, $\cxs$ is endowed with the usual canonical $\zp[[\cg]]$--module structure (with the $\cg$--action on $\cxs$ given
by lift--and--conjugation.) In particular, $\cxs$ has a canonical $\Lambda$--module structure. Recall Iwasawa's classical theorems stating that
$\cxs$ is a finitely generated $\Lambda$--module of rank $r_2(K)$ (the number of complex infinite primes in $K$) and that $\cxs^+$ is $\Lambda$--torsion and contains no non--trivial
finite $\Lambda$--submodules (see \cite{Iwasawa-zl}.)

Further, let us assume that $\bmu_p\subseteq K$ (i.e. $\bmu_{p^\infty}\subseteq\ck$.) In what follows, we adopt the notations and definitions of \S\ref{appendix-twisting} in the Appendix. In particular, note
that under our current assumptions the Teichm\"uller component $\omega_p$ of the $p$--adic cyclotomic character $c_p:\cg\to\zp^\times$ factors through $G$.

%We let $c_p: \cg\to {\rm Aut}(\mu_{p^\infty})\simeq\zp^\times$ denote the $p$--cyclotomic character of $\cg$.  As usual, we decompose $c_p$ as
%$c_p=\omega_p\cdot\kappa_p$, where $\omega_p$ is the Teichm\"uler character of $\cg$ (taking values in ${\rm Aut}(\bmu_p)\simeq \bmu_{p-1}$) and $\kappa_p:=c_p\cdot\omega_p^{-1}$
%(taking values in $1+p\zp$.) Under our hypotheses, $\omega_p$ factors through $G$. For every $n\in\Bbb Z$, there is a unique continuous isomorphism
%of topological $\zp$--algebras
%$${t_n}: \zp[[\cg]]\overset\sim\longrightarrow \zp[[\cg]],$$
%satisfying $t_n(g)=c_p(g)^n\cdot g$, for all $g\in G(\ck/K)$.
%For any $\zp[[\cg]]$--module $M$ and any $n\in\z$, we let $M(n)$ denote the usual Tate twist of $M$ by $c_p^n$. More precisely, $M(n):=M$ with a new $\zp[[\cg]]$--action given by
%$\lambda\ast x:=t_n(\lambda)\cdot x$, for all $\lambda\in\zp[[\cg]]$ and $x\in M$. Throughout, $\zp$ is viewed as a $\zp[[\cg]]$--module
%with the trivial $\cg$--action.  We let
%$$\iota: \zp[[\cg]]\overset\sim\longrightarrow \zp[[\cg]]^{\rm op}$$
%be the unique continuous isomorphism of topological $\zp$--algebras which satisfies $\iota(g)=g^{-1}$, for all $g\in \cg$.

By the definitions of $\cxs$ and $\ck_{\cs, \emptyset}$,  Kummer theory leads to perfect $\zp$--bilinear pairings of $\zp[[\cg]]$--modules
$$\langle \cdot\,,\,\cdot\rangle _m\,:\, \cxs/p^m\cxs \, \times\,  \ck_{\cs, \emptyset}^{(p^m)}/\ck^{\times p^m}\longrightarrow\, \bmu_{p^m}\,,$$
for all $m\in\z_{\geq 1}$, with the $\cg$--equivariance property
\begin{equation}\label{pairing-equivariance}
\langle {}^gx,\, {}^gy\rangle _m=g(\langle x,\,  y\rangle_m),\end{equation}
for all $x\in \cxs/p^m\cxs$,
$y\in \ck_{\cs, \emptyset}^{(p^m)}/\ck^{\times p^m}$ and $g\in \cg$. Consequently, for all $m$ as above,
we obtain perfect $\zp$--bilinear pairings
$$\cxs^+/p^m\cxs^+ \, \times\,  (\ck_{\cs, \emptyset}^{(p^m)}/\ck^{\times p^m})^-\longrightarrow\, \bmu_{p^m}\,,$$
with property (\ref{pairing-equivariance}). Now, we use Corollary \ref{minus-reinterpret} and  pass to a projective limit with respect to $m$ and the obvious transition maps in the pairings above  to obtain a perfect $\zp$--bilinear, continuous pairing of $\zp[[\cg]]$--modules
$$\langle \cdot\,,\,\cdot\rangle \,:\, \cxs^+\, \times\, T_p(\cm^\ck_{\cs, \emptyset})^-\longrightarrow \zp(1),$$
with property (\ref{pairing-equivariance}). Consequently, the following holds.
\begin{lemma}\label{link-classical} The last pairing induces an isomorphism of
$\zp[[\cg]]$--modules
$$T_p(\cm^\ck_{\cs, \emptyset})^-\simeq\, {\rm Hom}_{\zp}(\mathfrak X_S^+,\, \zp(1)), $$
where the right-hand side has the $\zp[[\cg]]$--module structure given by
$$(\lambda\ast f)(x):=f((\iota\circ t_1)(\lambda)\cdot x),$$
for all $f\in{\rm Hom}_{\zp}(\mathfrak X_S^+,\, \zp(1))$,
$x\in \mathfrak X_S^+$ and $\lambda\in \zp[[\cg]]$.
\end{lemma}
\begin{proof}
For all $x\in \mathfrak X_S^+$, $y\in T_p(\cm^\ck_{\cs, \emptyset})^-$ and $g\in \cg$, we have
$$\langle {}^gx,\, y\rangle =g(\langle x,\, {}^{g^{-1}}y\rangle )=c_p(g)\langle x,\, {}^{g^{-1}}y\rangle =\langle x,\,(\iota\circ t_1)(g)\cdot y\rangle .$$
This shows that the isomorphism above is at least $\zp[\cg]$--linear. However, since it is
continuous, it has to be  $\zp[[\cg]]$--linear, as desired. \end{proof}

\begin{remark}\label{remark-link-classical} With notations as above, let $\ct$ be a finite, nonempty, $\cg$--invariant set of finite primes
in $\ck$, which is disjoint from $\cs\cup\cs_p$. Then, the group morphism $\ca_{\ck, \ct}\twoheadrightarrow\ca_\ck$ induces
an obvious morphism of abstract $1$--motives $\mk\longrightarrow \cm^\ck_{\cs, \emptyset}$. This leads to a morphism  $T_p(\mk)\longrightarrow T_p(\cm^\ck_{\cs, \emptyset})$ of
$\zp[[\cg]]$--modules. As a consequence of (\ref{t-sequence-pm})
this morphism leads to an exact sequence of $\zp[[\cg]]$--modules
\begin{equation}\label{sequence-empty-T}\xymatrix{0\ar[r] &T_p(\Delta_{\ck, \ct})^-/\zp(1)\ar[r] &T_p(\mk)^-\ar[r] &T_p(\cm^\ck_{\cs, \emptyset})^-
\ar[r] &0}\end{equation}
Therefore, Lemma \ref{link-classical} above gives surjective morphisms of $\zp[[\cg]]$--modules
$$T_p(\mk)^-(n-1)\twoheadrightarrow T_p(\cm^\ck_{\cs, \emptyset})^-(n-1))\simeq (\cxs^+)^\ast(n),\qquad \forall\, n\in\z.$$
Above,  $(\cxs^+)^\ast:={\rm Hom}_{\zp}(\cxs^+, \zp)$, endowed with the contravariant $\cg$--action given by ${}^gf(x):=f(g^{-1}\cdot x)$, for
all $f\in(\cxs^+)^\ast$, $g\in \cg$ and $x\in \cxs^+$.
\end{remark}

\section{Cohomological triviality}\label{ct-section}

 Throughout this section, $\ck/\ck'$ will denote a Galois extension of $\zp$--fields, of Galois group $G$. We fix two finite sets $\cs$ and $\ct$
of finite primes in $\ck$, such that
$\ct\cap(\cs_p\cup\cs)=\emptyset$. From now on, we assume that $\ct\ne\emptyset$ and $\cs$
contains the finite ramification locus $\cs_{\rm ram}^{\rm fin}(\ck/\ck')$ of
$\ck/\ck'$. Also, we assume that $\cs$ and $\ct$ are $G$--invariant and
let $\cs'$ and $\ct'$ denote the sets consisting of all primes on
$\ck'$ sitting below primes in $\cs$ and $\ct$, respectively.

If $\ck$ is of CM--type, $j$ will denote, as usual, the complex conjugation automorphism of $\ck$. Since $j$ commutes with any element in $G$ (as automorphisms of $\ck$), it is easy to check that, in this case,
$\ck'$ is either totally real
or of CM--type, depending on whether $\ck'\subseteq\ck^+:=\ck^{j=1}$ or not. In the case where both $\ck$ and $\ck'$ are of CM--type, then
the complex conjugation automorphism of $\ck'$ is the restriction of $j$ to $\ck'$ and will be denoted by $j$ as well. If $\ck$ is of CM--type, we will assume further that $\cs$ and $\ct$ are $j$--invariant as well.

{\it Throughout this section, we assume the vanishing of the Iwasawa $\mu$--invariant of all $\zp$--fields involved.}

For simplicity, we let $\cm:=\mk$ and $\cm':=\mkp$. The natural abstract $1$--motive morphism $\cm'\to\cm$ (see Remark \ref{reinterpret-extensions}) induces $\zp$--module morphisms $\cm'[p^n]\to\cm[p^n]^G$ and
$T_p(\cm')\to T_p(\cm)^G$. The main goal of this section is to use these morphisms in order to study the $\zp[G]$--module structure of
$T_p(\cm)^-$, in the case where $\ck$ is of CM--type and $p$ is odd.
\medskip

\begin{proposition}\label{invariants} Assume that $m$ is a non--trivial power of
$p$. Then,
\begin{enumerate}
  \item The inclusion $\ck'^\times\subseteq\ck^\times$ induces a group isomorphism
  $$\kpstm/\kptm \simeq (\kstm/\ktm)^G\,.$$
  \item If $\ck$ and $\ck'$ are of CM--type and $p$ is odd, then the morphism $\cm'\to\cm$ of abstract $1$--motives
  induces canonical $\zp$--module isomorphisms
  $$\cm'[m]^-\simeq(\cm[m]^{-})^G, \qquad T_p(\cm')^- \simeq (T_p(\cm)^-)^G\,.$$
\end{enumerate}
\end{proposition}

\begin{proof} (1) Let us fix an $m$ as above. Since $\ck/\ck'$ is unramified at finite primes outside of $\cs'$, we have inclusions
$$\kpt=({\kt})^G\subseteq\kt, \qquad \kpstm\subseteq({\kstm})^G\subseteq\kstm.$$
Since $\ct$ is non-empty and disjoint from $\cs_p$, the group $\kt$ has no $m$--torsion. Therefore, the $m$--power--map induces
a $G$--invariant group isomorphism $\kt\simeq\ktm$. If one takes $G$--invariants in this isomorphism, one obtains
$$\kptm=(\ktm)^G=\ktm\cap\ck'^\times.$$
When combined with the displayed inclusions above, this leads to an injection
$$\kpstm/\kptm \hookrightarrow (\kstm/\ktm)^G\,.$$
In order to complete the proof of (1), we need to show that this injection is an isomorphism. For this purpose, we write the first four terms in the long $G$--cohomology sequence
associated to the short exact sequence of $\z[G]$--modules
$$\xymatrix{1\ar[r] &\ktm\ar[r]\ar[r] &\kstm\ar[r] &\kstm/\ktm\ar[r] &1}.$$
We obtain an exact sequence of multiplicative groups
$$\xymatrix{1\ar[r] &\kptm\ar[r] &(\kstm)^G\ar[r] &(\kstm/\ktm)^G\ar[r] & H^1(G,\, \ktm)\ar[r]&\cdots}. $$
Therefore, (1) would be a consequence of the following equalities.
\begin{equation}\label{reduction}
H^1(G,\, \ktm)=0,\qquad (\kstm)^G=\kpstm.
\end{equation}
Since $\kt\simeq\ktm$, the first equality above is equivalent to
$H^1(G,\, \kt)=0$. This is proved as follows. We let
$\ck_{(\ct)}^\times:=\{x\in\ck^\times\mid {\rm ord}_w(x)=0,
\forall w\in\ct\}$. We have short exact sequences of
$\z[G]$--modules
$$\xymatrix{1\ar[r]&\kt\ar[r] &\ck_{(\ct)}^\times\ar[r]^{{\rm res}_{\ct}} &\Delta_{\ck, \ct}\ar[r] & 1},$$
$$\xymatrix{1\ar[r]&\ck_{(\ct)}^\times\ar[r] &\ck^\times\ar[r]^{\quad{div}_{\ck, \ct}\qquad} &{\mathcal Div}_{\ck}(\ct)\ar[r] & 0.}$$
Here, we have ${\rm res}_{\ct}(x):=(x\mod w)_{w\in\ct}$, for all
$x\in\ck_{(\ct)}^\times$, and ${div}_{\ck,
\ct}(x):=\sum_{w\in\ct}{\rm ord}_w(x)\cdot w$, for all
$x\in\ck^\times$. The maps ${\rm res}_{\ct}$ and  ${div}_{\ck,
\ct}$ are surjective as a consequence of the weak approximation
theorem applied to the independent valuations of $\ck$
corresponding to the primes in $\ct$. Since the extension
$\ck/\ck'$ is unramified at primes in $\ct'$ and
$\ct\cap\cs_p=\emptyset$, we have $\z[G]$--module isomorphisms
\begin{equation}\label{Delta-T}
\Delta_{\ck, \ct}\simeq\bigoplus_{w'\in\ct'}(\kappa(w)^\times\otimes_{\z[G_w]}\z[G]), \qquad  {\mathcal Div}_{\ck}(\ct)\simeq \bigoplus_{w'\in\ct'}(\z w\otimes_{\z[G_w]}\z[G])\,,
\end{equation}
where $w$ is a prime in $\ct$ sitting above $w'$ and $G_w$ is the
decomposition group of $w$ in $\ck/\ck'$. Since $G_w\simeq
G(\kappa(w)/\kappa(w'))$ and $\kappa(w')$ is a $\zp$--extension of
a finite field, $G_w$ is cyclic of order coprime to $p$. The first
isomorphism in (\ref{Delta-T}) combined with Shapiro's Lemma,
Hilbert's Theorem 90 and Herbrandt quotient theory (view
$\kappa(w)$ as a union of Galois extensions of Galois group $G_w$
of finite subfields of $\kappa(w')$) give
$$\widehat H^i(G, \Delta_{\ck, \ct})\simeq \bigoplus_{w'\in\ct'}\widehat H^i(G_w,
\kappa(w)^\times)=0\,,\qquad \forall\,  i\in\z.$$ Consequently,
Tate cohomology applied to the first exact sequences above
gives
$$H^1(G, \kt)\simeq H^1(G, \ck_{(\ct)}^\times)\,.$$
The second isomorphism in (\ref{Delta-T}) combined with Shapiro's
Lemma gives
$${\mathcal Div}_{\ck'}(\ct')\simeq {\mathcal Div}_{\ck}(\ct)^G, \quad
H^1(G, {\mathcal Div}_{\ck}(\ct))\simeq
\bigoplus_{w'\in\ct'}H^1(G_w, \z)=0\,,$$ where the first
isomorphism above is induced by the canonical injection ${\mathcal
Div}_{\ck'}(\ct')\to {\mathcal Div}_{\ck}(\ct)$. Since
${div}_{\ck, \ct}\mid_{\ck'^\times}={div}_{\ck, \ct}$, Tate
cohomology applied to the second exact sequence above combined
with Hilbert's Theorem 90 gives
$$H^1(G,\ck_{(\ct)}^\times)\simeq H^1(G, \ck^\times)=0\,.$$
Consequently, we have
$$H^1(G, \kt)\simeq H^1(G, \ck_{(\ct)}^\times)\simeq H^1(G, \ck^\times)=0,$$
which concludes the proof of the first equality in (\ref{reduction}).

In order to prove the second equality in (\ref{reduction}), we consider the exact sequence
$$\xymatrix{1\ar[r] &\kstm\ar[r] &\kt\ar[r]^{\overline{div}_{\ck, \cs}\quad\qquad } &{\mathcal Div}_{\ck, \cs}\otimes\z/m\z}$$
in the category of $\z[G]$--modules, where ${\mathcal Div}_{\ck,
\cs}:=\oplus_{v\not\in\cs}\Gamma_v\cdot v$ is the group of
divisors in $\ck$ supported away from $\cs$ and $\overline{div}_{\ck,
\cs}(x):=(\sum_{v\not\in\cs}{\rm ord}_v(x)\cdot v)\otimes\widehat
1$, for all $x\in\kt$. Since $\ck/\ck'$ is unramified away from
$\cs$, the natural injective morphism ${\mathcal Div}_{\ck',\cs'}\to
{\mathcal Div}_{\ck, \cs}$ induces an isomorphism of groups ${\mathcal Div}_{\ck', \cs'}\simeq ({\mathcal Div}_{\ck, \cs})^G$ and the
restriction of $\overline{div}_{\ck, \cs}$ to $\kpt$ equals $\overline{div}_{\ck',
\cs'}$. Consequently, when we take $G$--invariants in the exact
sequence above, we obtain an exact sequence at the $\ck'$--level
$$\xymatrix{1\ar[r] &(\kstm)^G\ar[r] &\kpt\ar[r]^{\overline{div}_{\ck', \cs'}\quad\qquad } &{\mathcal Div}_{\ck', \cs'}\otimes\z/m\z},$$
which shows that $\kpstm=(\kstm)^G$, as desired.
\medskip

(2) The first isomorphism in part (2) of the Proposition is a direct consequence of part (1) combined with Remark \ref{reinterpret-extensions} and Corollary \ref{minus-reinterpret}.
The second isomorphism in part (2) is obtained from the first by taking the obvious projective limit.
\end{proof}
\bigskip

Let us assume that $\ck$ is of CM--type and $p$ is odd. Then, if we let $\fp$ denote the field with $p$ elements, we have an exact sequence of $\fp[G]$--modules
$$\xymatrix{1\ar[r] &\ca_{\ck, \ct}^-[p]\ar[r] &\cm[p]^-\ar[r] &({\mathcal Div}_{\ck}(\cs\setminus\cs_p)\otimes\z/p\z)^-\ar[r] &1. }
$$
According to Lemma \ref{class-group-coranks}, if we let $r_{\cm}^-:=\dim_{\fp}\,\cm[p]^-$, we have
\begin{equation}\label{rm-minus}
r_{\cm}^- =\lambda_{\ck}^- + \delta_{\ck, \ct}^- -\delta_{\ck} + d_{\ck, \cs}^-\,,
\end{equation}
where $\lambda_{\ck}^-$, $\delta_{\ck, \ct}^-$ and  $\delta_{\ck}$ are as in Lemma \ref{class-group-coranks} and
$$d_{\ck, \cs}^-:=\dim_{\fp}\,({\mathcal Div}_{\ck}(\cs\setminus\cs_p)\otimes\z/p\z)^-.$$
We need a concrete formula for $d_{\ck, \cs}^-$. For that purpose, we let $J:=G(\ck/\ck^+)$ and we view ${\mathcal Div}_{\ck}(\cs\setminus\cs_p)$ as a $\z[J]$--module.
Obviously, $J$ is cyclic of order $2$, generated by $j$. Let $\cs^+$ and $\cs_p^+$ denote the sets of primes in $\ck^+$ sitting below primes
in $\cs$ and $\cs_p$, respectively. Also, for every $v\in\cs^+\setminus\cs_p^+$, let $J_v$ be the decomposition group of $v$ in $\ck/\ck^+$. Then, we have the following
obvious $\z[J]$--module and $\fp$--vector space isomorphisms, respectively:
$${\mathcal Div}_{\ck}(\cs\setminus\cs_p)\,\,\simeq\bigoplus_{v\in\cs^+\setminus\cs_p^+}\z[J/J_v], \quad ({\mathcal Div}_{\ck}(\cs\setminus\cs_p)\otimes\z/p\z)^-\simeq \bigoplus_{v\in\cs^+\setminus\cs_p^+}\fp[J/J_v]^-.$$
Obviously, we have $\fp[J/J_v]^-\simeq\fp$, if $J_v$ is trivial (i.e. if $v$ splits in $\ck/\ck^+$) and $\fp[J/J_v]^-=0$, otherwise. Consequently, we have the following formula
for $d_{\ck, \cs}^-$:
$$d_{\ck, \cs}^-={\rm card}\, \{v\mid v\in \cs^+\setminus\cs_p^+, \text{ $v$ splits completely in } \ck/\ck^+\}\,.$$
The above formula permits us to reformulate the main result in \cite{Kida} as follows.
\begin{theorem}[Kida]\label{Kida} Assume that $p$ is odd, $G$ is a $p$--group and $\ck$ and $\ck'$ are of CM--type. Then,
we have the following equality.
$$(\lambda_\ck^- - \delta_\ck + d_{\ck, \cs}^-) = |G|\cdot (\lambda_{\ck'}^- - \delta_{\ck'} + d_{\ck', \cs'}^-)\,.$$
\end{theorem}

\noindent The above result is known is ``Kida's formula''. It was first proved by Kida in \cite{Kida} and later and with different methods by Iwasawa (see \cite{Iwasawa-RH}) and
Sinnott (see \cite{Sinnott-Kida}.) This should be viewed as a ``partial'' number field analogue of Hurwitz's genus formula
for finite covers of smooth, projective curves over an algebraically closed field. Here, $\ca_{\ck}^-$ plays the role
of the Jacobian and its corank (equal to $\lambda_{\ck}^-$) plays the role of twice the genus of the underlying curve. Of course, a ``full'' number field
analogue of the Hurwitz  genus formula should involve the ideal class--group $\ca_{\ck}$ and its corank $\lambda_{\ck}$. To our knowledge,
no such formula exists at the moment. The following consequence of the above theorem is of interest to us.

 \begin{corollary}\label{rm-minus-G} Assume that $p$ is odd, $G$ is a $p$--group and $\ck$ and $\ck'$ are of CM--type. Then,
we have the following equality.
$$r_{\cm}^- = |G|\cdot r_{\cm'}^-\,.$$
 \end{corollary}
 \begin{proof} The first isomorphism in (\ref{Delta-T}) induces an isomorphism of $\zp[G]$--modules
 $$(\Delta_{\ck, \ct}\otimes\zp)^-\simeq (\Delta_{\ck', \ct}\otimes\zp)^-\otimes_{\zp}\zp[G]\,.$$
Above, we took into account that since $G$ is a $p$--group, $G_w$
is trivial, for all $w\in\ct$.
 Now, by comparing $p$--coranks on both sides of the above isomorphism, we
 get
 $$\delta_{\ck, \ct}^-=\mid G\mid\cdot\, \delta_{\ck', \ct'}^-\,.$$
 The equality in the statement of the corollary is a direct consequence of the above equality combined
 with (\ref{rm-minus}) and Theorem \ref{Kida}.
 \end{proof}
\medskip

 Now, we are ready to state and prove the main results of this section. The reader will note that in what follows we are using a strategy similar to
that of \S3 in \cite{GP}.

\begin{theorem}\label{free} Assume that $G$ is a $p$--group, $\ck$ and $\ck'$ are of CM--type, and $p$ is odd. Then, the following hold.
\begin{enumerate}
  \item The $\fp[G]$--module $\cm[p]^-$ is free of rank
  $$r_{\cm'}^-:=\lambda_{\ck'}^-+\delta_{\ck', \ct'}^{-} + d_{\ck', \cs'}^- -\delta_{\ck'}\,.$$
  \item The $\zp[G]$--module $T_p(\cm)^-$ is free of rank $r_{\cm'}^-$.
\end{enumerate}
\end{theorem}
\begin{proof} For the proof of (1), we need the following (see \cite{Nakajima}, Proposition 2, \S4.)
\begin{lemma}[Nakajima]
Assume that $k$ is a field of characteristic $\ell\ne 0$, $G$ is a finite $\ell$--group and
$M$ is a finitely generated $k[G]$--module, such that
$$\dim_k M\geq |G|\cdot\dim_k M^G.$$
Then, $M$ is a free $k[G]$--module of rank equal to $\dim_k M^G$.
\end{lemma}
\noindent Now, if we combine the first isomorphism in Proposition \ref{invariants}(2) for $m:=p$, with Corollary \ref{rm-minus-G} above, we obtain the following.
$$\dim_{\fp} \cm[p]^- = \mid G\mid\cdot\dim_{\fp}(\cm[p]^-)^G\,.$$
The above Lemma applied to $\ell:=p$, $k:=\fp$ and $M:=\cm[p]^-$ implies that, indeed, $\cm[p]^-$ is a free $\fp[G]$--module whose rank satisfies
$${\rm rank}_{\fp[G]}\cm[p]^-=\dim_{\fp}(\cm[p]^-)^G=\dim_{\fp}\cm'[p]^-=r_{\cm'}^-\,.$$
This concludes the proof of part (1).
\smallskip

In order to prove part (2), we consider the exact sequence of $\zp[G]$--modules
$$\xymatrix{0\ar[r] &T_p(\cm)^-\ar[r]^{\times p} &T_p(\cm)^-\ar[r] &\cm[p]^-\ar[r] &0}
$$
where the injective morphism is the multiplication--by--$p$ map (see (\ref{Tate-torsion}).) Since $\cm[p]^-$ is a free $\fp[G]$--module, it is $G$--induced and therefore
$G$--cohomologically trivial. Consequently, when we apply Tate $G$--cohomology to the short exact sequence above, we obtain group--isomorphisms
$$\xymatrix{\widehat H^i(G, T_p(\cm)^-)\ar[r]^{\times p}_{\sim} &\widehat H^i(G, T_p(\cm)^-)
}$$
given by the multiplication--by--$p$ map, for all $i\in\z$. Since the cohomology groups above are $p$--groups, this implies that $\widehat H^i(G, T_p(\cm)^-)=0$.
Since $G$ is a $p$--group, this implies that $T_p(\cm)^-$ is cohomologically trivial (see Theorem 8, in \cite{Serre-Local}, Ch. IX, \S5.) Therefore, $T_p(\cm)^-$ is a projective $\zp[G]$--module, as it is
$\zp$--free (see loc.cit.) However, since $G$ is a $p$--group, the ring $\zp[G]$ is local and therefore $T_p(\cm)^-$ is a free $\zp[G]$--module (see \cite{Milnor}, Lemma 1.2.)
The $\fp[G]$--module isomorphism
$$T_p(\cm)^-=\fp[G]\otimes_{\zp[G]}T_p(\cm)^-$$
(a consequence of the short exact sequence above) combined with part (1) leads to
$${\rm rank}_{\zp[G]}T_p(\cm)^-={\rm rank}_{\fp[G]}\cm[p]^-=r_{\cm'}^-.$$
This concludes the proof of the Theorem.
\end{proof}

\begin{theorem}\label{projective} Assume that $\ck$ is of CM--type, $p$ is odd and $G$ is arbitrary. Then, if we let ``{\rm pd}'' denote ``projective dimension'', the following hold.
\begin{enumerate}
\item ${\rm pd}_{\zp[G]^-}T_p(\cm)^-={\rm pd}_{\zp[G]}T_p(\cm)^-=0.$
\item If, in addition, $\ck'$ is the $\zp$--cyclotomic
    extension of a number field $k$, such that $\ck/k$ is
    Galois of Galois group $\cg$ and $\cs$ and $\ct$ are
    $\cg$--invariant, then  ${\rm
    pd}_{\zp[[\cg]]^-}T_p(\cm)^-={\rm
    pd}_{\zp[[\cg]]}T_p(\cm)^-=1$, unless $T_p(\cm)^-=0$.
\end{enumerate}
\end{theorem}
\begin{proof}
(1) Since $T_p(\cm)^-$ is a free $\zp$--module, the statement in part (1) of the Theorem is equivalent to the $G$--cohomological triviality
of $T_p(\cm)^-$. (See \cite{Serre-Local}, Ch. IX, \S5.) In order to prove that, let $H$ denote a $p$--Sylow subgroup of $G$ and let $\ck^H$ be the maximal subfield
of $\ck$ fixed by $H$. Then, $\ck/\ck^H$ is Galois, of Galois group $H$. Also, $\ck^H$ is of CM--type. Indeed, if $\ck^H\subseteq \ck^+$, then
$2=[\ck:\ck^+]$ would divide $\mid H\mid=[\ck:\ck^H]$, which is false, because $p$ is odd. Consequently, we can apply Theorem \ref{free}(2) to the extension
$\ck/\ck^H$ to conclude that
$$\widehat H^i(H, T_p(\cm)^-)=0,\quad  \forall\, i\in\z\,.$$
Since this happens for any $p$--Sylow subgroup $H$ of $G$, the module $T_p(\cm)^-$ is $G$--cohomologically trivial and therefore
$\zp[G]$--projective, as desired. (See loc.cit.)

(2) Let $\Gamma_k:=G(\ck'/k)$. Since $\Gamma_k$ is a free abelian topological group,
the usual Galois--restriction exact sequence
\begin{equation}\label{groups-split}
\xymatrix{1\ar[r] &G\ar[r] &\cg\ar[r]^{\rm res_{\ck/\ck'}} &\Gamma_k\ar[r] &1}
\end{equation}
is split. Pick a splitting and let $\Gamma$ be its image in $\cg$. Then, we have
a group isomorphism $\cg\simeq G\rtimes \Gamma$, where $\rtimes$ denotes a semi-direct product. Let $K$ be the normal closure (over $k$) in $\ck$ of
$\ck^\Gamma$ (the maximal subfield of $\ck$ fixed by $\Gamma$.) Then $K/k$ is a Galois extension of number fields
and  $\ck=K\cdot\ck'$ (the compositum of $K$ and $\ck'$ inside $\ck$.) Consequently, if we let $G':=G(K/k)$,
Galois restriction induces an isomorphism between $G$ and the normal subgroup $G(K/\ck'\cap K)$ of $G'$. Via this isomorphism,
we identify $G$ with $G(K/\ck'\cap K)$ and $G'/G$ with $G(\ck'\cap K/k)$. We have an exact sequence
of groups induced by the obvious Galois restriction maps
\begin{equation}\label{groups}
\xymatrix{1\ar[r] &\cg\ar[r] &G'\times\Gamma\ar[r] & G'/G\ar[r] &1}.
\end{equation}
Consequently, $\zp[[G'\times\Gamma]]$ viewed as a left or right $\zp[[\cg]]$--module in the obvious way is free
of basis given by any complete set of representatives for the left cosets of $G$ in $G'$. We consider the
left $\zp[[G'\times\Gamma]]$--module
$$\widetilde{T_p(\cm)^-}:= \zp[[G'\times\Gamma]]\otimes_{\zp[[\cg]]}T_p(\cm)^-.$$
As $\zp[[G'\times\Gamma]]$ is a free right $\zp[[\cg]]$--module, (2) in the Theorem is equivalent
to
$${\rm pd}_{\zp[[G'\times\Gamma]]}\widetilde{T_p(\cm)^-}=1.$$
We claim that this follows directly
from part (1) of the Theorem combined with Proposition 2.2 and Lemma 2.3 in \cite{Popescu-CS}. In order to see this, let
us first note that we have an obvious ring isomorphism $\zp[[G'\times\Gamma]]\simeq\Lambda[G']$, where $\Lambda:=\zp[[\Gamma]]$.
Now, we have the following.
\begin{lemma}\label{projective-dimension} A finitely generated $\Lambda[G']$--module
$M$  satisfies ${pd}_{\Lambda[G']}M\leq 1$ if and only if the following conditions are satisfied.
\begin{enumerate}\item [(i)] ${\rm pd}_{\zp[G']}M\leq 1$.
\item[(ii)] $M$ has no non trivial finite $\Lambda$--submodules (i.e. ${\rm pd}_{\Lambda}M\leq 1$.)
\end{enumerate}
\end{lemma}
\begin{proof} Combine Proposition 2.2 and Lemma 2.3 in \cite{Popescu-CS}.\end{proof}
\noindent Since the $\Lambda$--module $\widetilde{T_p(\cm)^-}$ is $\zp$--free, it satisfies (ii) above automatically.
Also, as a consequence of the definitions and exact sequence (\ref{groups}), we have an isomorphism
of left $\zp[G']$--modules
$$\widetilde{T_p(\cm)^-}\simeq \zp[G']\otimes_{\zp[G]}{T_p(\cm)^-}.$$
Consequently, part (1) of the Theorem implies that ${\rm pd}_{\zp[G']}\widetilde{T_p(\cm)^-}=0$.
The above Lemma implies that ${\rm pd}_{\Lambda[G']}\widetilde{T_p(\cm)^-}\leq 1$. However, since
$\widetilde{T_p(\cm)^-}$ is of finite $\zp$--rank, we must have ${\rm pd}_{\Lambda[G']}\widetilde{T_p(\cm)^-}=1$.
This concludes the proof of part (2) in the Theorem.
\end{proof}
\medskip

\begin{corollary}\label{coinvariants} Assume that $\ck$ and $\ck'$ are of CM--type,
$p$ is odd and $G$ is arbitrary. Then, there is a canonical isomorphism of $\zp$--modules
$$T_p(\cm)^-/I_GT_p(\cm)^-\simeq T_p(\cm')^-,$$
where $I_G$ is the usual augmentation ideal in $\zp[G]$.
\end{corollary}
\begin{proof}
Since $T_p(\cm)^-$ is $\zp[G]$--projective, it is $G$--cohomologically trivial. In particular,
$\widehat H^0(G, T_p(\cm)^-)=\widehat H^{-1}(G, T_p(\cm)^-)=0$. As a consequence, if we denote by $N_G$ the usual
norm element in $\zp[G]$, we have
$$(T_p(\cm)^-)^G=N_G\cdot T_p(\cm)^-,\qquad T_p(\cm)^-/I_GT_p(\cm)^-\simeq N_G\cdot T_p(\cm)^-\,,$$
where the isomorphism of $\zp$--modules to the right is induced by multiplication with $N_G$. Now, the corollary is a direct consequence of
Proposition \ref{invariants}, part (2).
\end{proof}

\section{The Equivariant Main Conjecture}\label{EMC}

We begin by recalling two classical theorems on special values of global $L$--functions. Let $K/k$ be an abelian extension
of number fields of Galois group $G$. Throughout, if $\mathcal C$ is an algebraically closed field, we denote by $\widehat G(\mathcal C)$ the group
of irreducible $\mathcal C$--valued characters of $G$. Let $S$ be a finite set of primes in $k$, which contains
the set $S_\infty(k)\cup S_{\rm ram}(K/k)$ of primes which are either infinite or ramify in $K/k$. For every $\chi\in\widehat G(\C)$, we let $L_S(\chi, s)$ denote the $\C$--valued $S$--incomplete Artin $L$--function
associated to $\chi$, of complex variable $s$. These $L$--functions are holomorphic everywhere, except for a simple pole at $s=1$ when $\chi=\mathbf 1_G$. We let
$$\Theta_{S, K/k}:=\sum_{\chi\in\widehat G(\C)}L_S(\chi^{-1}, s)\cdot e_\chi, \qquad  \Theta_{S, K/k}: \C\to \C[G]$$
denote the so--called $S$--incomplete $G$--equivariant $L$--function associated to $K/k$, where $e_{\chi}:=1/|G|\sum_{\sigma\in G}\chi(\sigma)\cdot\sigma^{-1}$
is the idempotent associated to $\chi$ in $\C[G]$.
\begin{theorem}[Siegel \cite{Siegel}]\label{siegel-theorem} For all $n\in\z_{\geq 1}$, we have
$$\Theta_{S, K/k}(1-m)\in\q[G].$$
\end{theorem}
\noindent As usual, for every abelian group $M$ and every $m\in\z$, we let $M(m)$ denote the group $M$ endowed with the action by the absolute Galois group
$G_k$ of $k$ given by the $n$--th power of the cyclotomic character $c_k: G_k\to {\rm Aut}(\bmu_\infty)\simeq{\widehat\z}^\times$.
\begin{theorem}[Deligne-Ribet \cite{Deligne-Ribet}, Cassou-Nogu\`es \cite{Cassou-Nogues}]\label{deligne-ribet-theorem} For all $m\in\z_{\geq 1}$, we have
$${\rm Ann}_{\z[G]}(\q/\z(m)^{\, G_K})\cdot\Theta_{S, K/k}(1-m)\subseteq \z[G],$$
where $\q/\z(m)^{\, G_K}$ denotes the $\z[G]$--module of $G_K$--invariants of $\q/\z(m)$.
\end{theorem}
\noindent Now, let $T$ be a finite, non-empty set of finite primes in $k$, disjoint from $S$. We let
$$\delta_{T, K/k}:=\prod_{v\in T}(1-\sigma_v^{-1}\cdot({\bf N}v)^{1-s}), \qquad \delta_{T, K/k}:\C\to\C[G],$$
where $\sigma_v$ denotes the Frobenius morphism associated to $v$ in $G$.
\begin{corollary}\label{Deligne-Ribet-rewrite} Let $p$ be a prime number and $m\in\z_{\geq 1}$. Assume that either $T$ contains at least a prime which does not sit above $p$ or
 $p$ does not divide the cardinality of (the finite group) $\q/\z(m)^{\, G_K}$. Then, we have
$$\delta_{T, K/k}(1-m)\cdot\Theta_{S, K/k}(1-m)\in\z_{(p)}[G]\,.$$
\end{corollary}
\begin{proof} Remark that if $v\in T$ does not sit above $p$, then $\sigma_v$ acts on the group $\q/\z(m)^{\, G_K}\otimes\zp\simeq \q_p/\z_p(m)^{\, G_K}$ via
multiplication by $({\bf N}v)^{m}$. This implies that, under the hypotheses of the Corollary, we have
$$\delta_{T, K/k}(1-m)\in {\rm Ann}_{\z_{(p)}[G]}(\qp/\zp(m)^{\, G_K}).$$
Now, the Corollary follows from Theorem \ref{deligne-ribet-theorem}.
\end{proof}
\begin{definition} For $S$ and $T$ as above, the $S$--incomplete $T$--modified $G$--equivariant $L$--function
associated to $K/k$ is defined by
$$\Theta_{S, T, K/k}:=\delta_{T, K/k}(s)\cdot\Theta_{S, K/k}(s), \qquad \Theta_{S, T, K/k}: \C\to\C[G].$$
\end{definition}
\begin{remark}\label{theta-st-remark} Note that $\Theta_{S, T, K/k}$ is holomorphic everywhere on $\C$ and
$$\Theta_{S, T, K/k}(1-m)\in\z_{(p)}[G],$$
whenever the integer $m$, the prime $p$ and the set $T$ satisfy the hypotheses of the Corollary above. In particular, if $T$ contains a prime of
residual characteristics coprime to the cardinality of $(\q/\z(m)^{G_K})$, for some $m\in\z_{\geq 1}$, then
$$\Theta_{S, T, K/k}(1-m)\in\z[G].$$
Also, if $T$ contains two primes of distinct residual characteristics, then
$$\Theta_{S, T, K/k}(1-m)\in\z[G],\qquad \forall\, m\in\z_{\geq 1}.$$
\end{remark}
\medskip

Throughout the rest of this section, we fix an odd prime $p$ and an abelian extension $\ck/k$, where $\ck$ is a CM $\zp$--field and $k$ is a totally real number field.
Let $\cg:={\rm Gal}(\ck/k)$. We fix two finite, non--empty,
disjoint sets $S$ and $T$ of primes in $k$, such that $S$ contains  $S_{\rm ram}(\ck/k)\cup S_\infty$. Note that, in particular,
$S$ contains the set $S_p$ of $p$--adic primes of $k$.
We let $\cs$ and $\ct$ denote
the sets of finite primes in $\ck$ sitting above primes in $S$ and $T$, respectively.

The main goal of this section is the proof of an Equivariant Main Conjecture for the data $(\ck/k, S, T, p)$.  This statement expresses the first Fitting invariant
${\fit}_{\zp[[\cg]]^-}(T_p(\mk)^-)$ of the finitely generated module $T_p(\mk)^-$  over the (Noetherian, commutative) ring
$$\zp[[\cg]]^-:=\zp[[\cg]]/(1+j)$$
in terms of a certain equivariant $p$--adic $L$--function $\Theta_{S,T}^{(\infty)}$, which is an element of $\zp[[\cg]]^-$
canonically associated to the
data $(\ck/k, S, T, p)$. As usual, $j$ denotes the complex conjugation automorphism of $\ck$, viewed as an element
of $\cg$. The precise statement, to be proved in \S\ref{emc-section} below, is the following.
\begin{theorem}[The Equivariant Main Conjecture]\label{emc} Under the above hypotheses, we have the following equality of ideals in $\zp[[\cg]]^-$.
$${\rm Fit}_{\zp[[\cg]]^-}(T_p(\mk)^-)=(\Theta_{S, T}^{(\infty)}).$$
\end{theorem}

Let $\ck'$ be the cyclotomic $\zp$--extension of $k$ and $\Gamma_k:=G(\ck'/k)$.
As in the proof of Theorem \ref{projective}(2), the exact sequence
$$\xymatrix{1\ar[r] &G(\ck/\ck')\ar[r] &\cg\ar[r]^{\rm res_{\ck/\ck'}} &\Gamma_k\ar[r] &1
}$$
is split. We fix a splitting, i.e. we fix a subgroup $\Gamma$ of $\cg$ which is mapped isomorphically
to $\Gamma_k$ by ${\rm res_{\ck/\ck'}}$. We let $K:=\ck^\Gamma$ denote the maximal subfield of $\ck$ fixed by $\Gamma$.
Then, Galois theory combined with the fact that $\cg$ is abelian, implies that $K/k$ is Galois, $K\cdot\ck'=\ck$
and $\ck'\cap K=k$. If we let $G:=G(K/k)$, Galois restriction induces group isomorphisms $\cg\simeq G\times\Gamma_k$, $\Gamma\simeq\Gamma_k$, and $G(\ck/\ck')\simeq G$.
Below, we freely identify these groups via these isomorphisms.
Obviously, $K$ is a CM number field whose cyclotomic $\zp$--extension is $\ck$. We let $j$ denote the complex conjugation automorphism of $K$ as well.
We fix a topological generator $\gamma$ of $\Gamma$ and identify it via the Galois restriction
isomorphism with a generator of $\Gamma_k$. For any finite extension $\co$ of $\zp$, we have isomorphisms of compact $\co[G]$--algebras
\begin{equation}\label{power-series-iso}\co[[\cg]]\simeq \co[G][[\Gamma]]\simeq \co[G][[t]],\qquad \co[[\cg]]^\pm\simeq \co[G]^\pm[[\Gamma]]\simeq \co[G]^\pm[[t]].\end{equation}
Above, $\co[G]^\pm:=\co[G]/(1\mp j)$, $t$ is a variable, and the right-most isomorphisms send $\gamma$ to $(t+1)$, as usual.

\subsection{The work of Wiles on the Main Conjecture \cite{Wiles}} In order to simplify notations, in this section we assume that $\bmu_p\subseteq K$ (i.e. $\bmu_{p^\infty}\subseteq\ck$.)
Let $c:=\omega\cdot\kappa$ be the decomposition of the $p$--cyclotomic character $c:=c_p$ of $\cg$ into its Teichm\"uler component $\omega:=\omega_p$ and
its complement $\kappa:=\kappa_p$, as in \S\ref{appendix-twisting} in the Appendix. Note that, in this case, $\omega$ and $\kappa$ factor through $G$ and $\Gamma$, respectively. Let $u$ be the element in $(1+p\zp)$ given by
$u:=\kappa(\gamma)$.

We fix an embedding $\Bbb C\hookrightarrow\cp$. Via this embedding we identify $\widehat G(\Bbb C)$ and $\widehat G(\cp)$.   A character $\psi\in\widehat G(\cp)$ is called even if $\psi(j)=1$ and odd otherwise. Let $\co$ be a fixed finite extension of $\zp$ which contains the values of all $\psi\in\widehat G(\cp)$. We fix a uniformizer $\pi$ in $\co$ and denote by $Q(\co)$ the field of fractions of $\co$.
In \cite{Deligne-Ribet}, Deligne and Ribet proved that for every character $\psi\in\widehat G(\cp)$, there exist power series
 $G_{\psi, S}(t)$ and  $H_{\psi, S}(t)$ in $\co[[t]]$, uniquely determined by the following properties.
\begin{equation}\label{power-series}
 H_{\psi, S}=\left\{
               \begin{array}{ll}
                 t, & \hbox{if $\psi=\boldsymbol 1_G$;} \\
                 1, & \hbox{otherwise;}
               \end{array}
             \right.
\qquad\quad
\frac{G_{\psi, S}(u^m-1)}{H_{\psi, S}(u^m-1)}=L_S(\psi\omega^{-m}, 1-m),\quad \forall\, m\in\z_{\geq 1.}
 \end{equation}
The reader may consult \S1 in \cite{Wiles}
and \S4 in \cite{Popescu-CS} for the properties above and note that, since $K\cap\ck' =k$,
all the non trivial characters of $G$ are of ``type S'', in the terminology used in loc.cit.

By definition, the $S$--incomplete $p$--adic $L$--function $L_{p, S}(\psi, s)$ associated to an even character
$\psi\in\widehat G(\mathbb C)$ is given by
\begin{equation}\label{p-adicL} L_{p, S}(\psi, 1-s)=\frac{G_{\psi, S}(u^s-1)}{H_{\psi, S}(u^s-1)},\qquad s\in\zp.\end{equation}
The function $L_{p, S}(\psi, s)$ is $p$--adically analytic everywhere, except for a possible pole of order $1$
at $s=1$, if $\psi=\boldsymbol 1_G$.

\begin{remark}\label{odd-characters} As an immediate consequence of the functional equation for the
global $L$--functions $L_S(\chi, s)$ considered above, we have $L_S(\chi, 1-m)=0$, whenever $\chi$ and $1-m$ have the same parity, for all $m\in\z_{\geq 1}$ and all $\chi\in\widehat G(\C)$.
Assume that if $\psi$ is an odd character.
Then $\psi\omega^{-m}$ and $1-m$ have the same parity, for all $m\in\Bbb Z_{\geq 1}$. Consequently,
we have $L_S(\psi\omega^{-m}, 1-m)=0$, for all $m\in\z_{\geq 1}$. Therefore, $G_{\psi, S}(t)=0$ and $L_{p, S}(\psi, s)=0$, for all odd characters $\psi.$
\end{remark}

For simplicity, if $L$ is a subfield of $K$ containing $k$, we denote by $\cx_{L}$ the Galois
group of the maximal abelian pro--$p$ extension of $\cl:=L\cdot\ck'$ (the cyclotomic $\zp$--extension of
$L$), which is unramified away from the primes above those in $S$. We view $\cx_{L}$ as a module
over $\zp[[G(\cl/k)]]\simeq\zp[G(L/k)][[\Gamma]]$ (and consequently over $\zp[[\cg]]$) in the usual way.
Note that $\cx_K$ is the module we denoted by $\cxs$ in \S\ref{classical}. We will continue to use $\cxs$ instead of $\cx_K$ in what follows. The modules $\cx_{L}$ and $\cx_{L}^+$ are finitely generated, respectively
torsion and finitely generated modules over $\Lambda:=\zp[[\Gamma]]$ (a classical theorem of Iwasawa \cite{Iwasawa-zl}.)

Let $\psi\in\widehat G(\cp)$ be an even character. Let $K_\psi:= K^{\ker\psi}$ be the largest subfield of $K$ fixed by the kernel of $\psi$. We have $G(K_\psi/ k)\simeq F\times H$,
where $H$ is the $p$--Sylow subgroup of  $G(K_\psi/k)$ and $F$ is its maximal subgroup of order coprime to $p$. Consequently, $\psi$ (viewed as a
character of $G(K_\psi/ k)$) splits as $\psi=\varphi\cdot\rho$, with $\varphi$ a faithful character of $R$ (of order coprime to $p$) and $\rho$ a faithful character of  $H$ (of order
a power of $p$.) Let $h$ be a generator of $H$ and assume that ${\rm ord}(h)=p^n$, for some $n\in\Bbb Z_{\geq 0}$.  Let $e_{\varphi}:=1/|F|\sum_{f\in F}\varphi(f)\cdot f^{-1}$ be the
idempotent associated to $\varphi$. Note that $e_{\varphi}\in \co[F]$. The following definition of ``character $\mu$--invariants'' of $\cxs$ is due
to Greenberg (see \cite{Neukirch}, pp. 656--657.)

\begin{definition}[Greenberg] Let $\psi\in\widehat G(\cp)$ be an even character. Then $\mu_{\psi, \co}(\cxs)$ is defined to be
the $\mu$--invariant $\mu_\co(\cxs^\psi)$ of the $\co[[\Gamma]]$--module
$$\cxs^\psi:=\left\{
    \begin{array}{ll}
      (h^{p^{n-1}}-1)e_\varphi(\cx_{K_\psi}\otimes_{\zp}\co), & \hbox{if $n\geq 1$;} \\
      e_\varphi(\cx_{K_\psi}\otimes_{\zp}\co), & \hbox{otherwise.}
    \end{array}
  \right.
$$
\end{definition}
\noindent Note that if the character $\psi$ has order coprime to $p$, i.e. $\psi=\varphi$ and $n=0$ in the above notation,
then we have
$$e_\varphi(\cx_{K_\psi}\otimes_{\zp}\co)=\{x\in (\cx_{K_\psi}\otimes_{\zp}\co)\mid gx=\psi(g)x, \forall g\in G(K_\psi/k)\}.$$
 This reconciles the definition above with Wiles's definition (see \cite{Wiles}, p. 497) of ``character $\mu$--invariants''
for characters of order coprime to $p$. (See also Definition 11.6.14 in \cite{Neukirch}.)

The following is the statement of the classical (non--equivariant) Main Conjecture for the module $\cxs^+$, proved by Wiles
in \cite{Wiles}. Below, we denote by $\mathfrak m_\gamma$ the multiplication-by-$\gamma$ automorphism of all the relevant vector spaces.

\begin{theorem}[Wiles]\label{Wiles} For every even character $\psi\in\widehat G(\C)$, we have
$$G_{\psi, S}(t) \sim \pi^{\mu_{\psi, \co}(\cxs)}\cdot{\rm det}_{Q(\co)}((t+1)-\mathfrak m_\gamma\mid e_{\psi}(\cxs\otimes_{\zp}Q(\co))),$$
where $\sim$ denotes association in divisibility in the ring $\co[[t]]$ and $e_{\psi}$ is the idempotent associated to $\psi$ in $Q(\co)[G]$.
\end{theorem}
\begin{proof} (Sketch.) For characters $\psi$ of order coprime to $p$, the statement above is the combination of Theorems 1.3 and 1.4 in \cite{Wiles}.
For arbitrary characters, \cite{Wiles} contains a proof of the statement above only up to a power of $\pi$ (Theorem 1.3. loc.cit.)
However, once one has the right definition of $\mu_{\psi, \co}(\cxs)$ (see above), one can deduce the statement above for arbitrary
characters $\psi$ without much difficulty from Theorems 1.3 and 1.4 in \cite{Wiles} restricted to the case of the trivial character and base fields $K_\psi$ and $K_{\psi^p}$, respectively (see Theorem 11.6.16 in \cite{Neukirch} for a proof.)  \end{proof}

\begin{remark}\label{Wiles-Fitting}
%Observe that the determinant in the statement of Theorem \ref{Wiles} is the characteristic polynomial of $(\gamma-1)$ acting on the
%finite dimensional $Q(\co)$--vector space $e_{\psi}(\cxs\otimes_{\zp}Q(\co))$. This vector space is spanned by an obvious $\co[[\Gamma]]$--submodule.
%As a consequence, the characteristic polynomial in question is a monic polynomial in $\co[t]$.
As proved by Greenberg (see Proposition 1 in \cite{Greenberg}), we have an equality of monic polynomials in $\co[t]$
$${\rm det}_{Q(\co)}((t+1)-\frak m_\gamma\mid e_{\psi}(\cxs\otimes_{\zp}Q(\co)))={\rm det}_{Q(\co)}((t+1)-\mathfrak m_\gamma\mid e_{\psi}(\cx_{K_\psi}\otimes_{\zp}Q(\co))),$$
for all even $\psi$. Also, it is easy to see that
$$e_{\psi}(\cx_{K_\psi}\otimes_{\zp}Q(\co))=\cxs^\psi\otimes_\co Q(\co)\,,$$
for all even $\psi$. Consequently, the right hand-side of the equivalence $\sim$ in Theorem \ref{Wiles} is precisely the characteristic polynomial
${\rm char}_{\co[[\Gamma]]}(\cxs^\psi)$ of the $\co[[\Gamma]]$--module $\cxs^\psi$. On the other hand, by definition, $\cxs^\psi$ is an
$\co[[\Gamma]]$--submodule of $(\cx_{K_\psi}\otimes_{\zp} \co)$. As such, it is a torsion $\co[[\Gamma]]$--module with no finite non trivial $\co[[\Gamma]]$--submodules (a classical theorem of Iwasawa.)
Consequently, ${\rm pd}_{\co[[\Gamma]]} \cxs^\psi=1$ (see Lemma \ref{projective-dimension}) and
its first Fitting ideal ${\rm Fit}_{\co[[\Gamma]]}(\cxs^\psi)$ is principal, generated by ${\rm char}_{\co[[\Gamma]]}(\cxs^\psi)$ (see Lemma 2.4 in \cite{Popescu-CS} and also Proposition \ref{fitting-calculation}(1) in the Appendix.)
This leads to the following equivalent, and perhaps more elegant formulation of Theorem \ref{Wiles}:
$${\rm Fit}_{\co[[\Gamma]]}(\cxs^\psi)=(G_{\psi, S}(t)),\qquad \text{for all even $\psi$},$$
where $(G_{\psi, S}(t))$ denotes the principal ideal of $\co[[\Gamma]]$ generated by $G_{\psi, S}(t)$.

In light of this reformulation,
it would be natural to expect that a $G$--equivariant refinement of Theorem \ref{Wiles} would give the Fitting ideal
${\rm Fit}_{\zp[[\cg]]^+}(\cxs^+)$ in terms of an equivariant $p$--adic $L$--function. Unfortunately, in general, the $\zp[[\cg]]^+$--module
$\cxs^+$ has infinite projective dimension, which makes its Fitting ideal non--principal and difficult to compute. This is the main reason why
the module $\cxs^+$ is replaced with $T_p(\mk)^-$, a $\zp[[\cg]]^-$--module of projective dimension $1$ (see Theorem \ref{projective} above), whose Fitting
ideal is principal (see Proposition \ref{fitting-calculation}(1) in the Appendix). Moreover, $T_p(\mk)^-$ has a quotient isomorphic to $(\cxs^+)^\ast(1)$ (see Remark \ref{remark-link-classical}.)

\end{remark}

\begin{corollary}\label{Wiles+zero-mu} If the $\mu$--invariant $\mu_{\zp}(\cxs^+)$ of the $\zp[[\Gamma]]$--module $\cxs^+$ is $0$, then the following hold
for all even characters $\psi\in\widehat G(\cp)$.
\begin{enumerate}
  \item $\mu_{\psi, \co}(\cxs)=0.$
  \item $G_{\psi, S}(t) \sim {\rm det}_{Q(\co)}((t+1)-\mathfrak m_\gamma\mid e_{\psi}(\cxs\otimes_{\zp}Q(\co)))$ in $\co[[t]]$.
\end{enumerate}
\end{corollary}
\begin{proof}
(2) is a consequence of (1) and Theorem \ref{Wiles}.
Part (1) is proved as follows. It is easily seen that Galois restriction gives an isomorphism of
$\zp[[\cg]]$--modules $\cxs^+\simeq \cx_{K^+}$, where $K^+$ is the maximal real subfield
of $K$. Also, since $\psi$ is even, we have $K_\psi\subseteq K^+$ and Galois restriction
leads to an exact sequence of $\zp[[\cg]]$--modules
$$
\xymatrix{\cx_{K^+}\ar[r] &\cx_{K_\psi}\ar[r] & G(K_\psi\cap\ck'/k)\ar[r] &1
}.$$
Since $G(K_\psi\cap\ck'/k)$ is finite and $\mu_{\zp}(\cxs^+)=0$, we have $\mu_{\zp}(\cx_{K_\psi})=0$.
Consequently, we have $\mu_\co(\cx_{K_\psi}\otimes_{\zp} \co)=0$. Therefore $\mu_\co(e_{\varphi}(\cx_{K_\psi}\otimes_{\zp} \co))=0$
and $\mu_{\psi, \co}(\cxs)=0$.  We have used the fact that  the $\mu$--invariant of a quotient or submodule
of an Iwasawa module is at most equal to that of the module itself.
\end{proof}

\subsection{The relevant equivariant $p$--adic $L$--functions}\label{equivariant-L-functions} For simplicity, we will continue to assume that $\boldsymbol\mu_p\subseteq K$,
unless otherwise stated. If $R$ is a commutative ring with $1$, we denote by $Q(R)$ the total ring of fractions of $R$. By abusing notation,
we denote by  $\iota, t_m: Q(\co[[\cg]])\simeq Q(\co[[\cg]])$ the unique $Q(\co)$--algebra automorphisms obtained by extending to $Q(\co)[[\cg]]$ the $\zp$--algebra
automorphisms $\iota, t_m: \zp[[\cg]]\simeq\zp[[\cg]]$ defined in \S\ref{appendix-twisting} of the Appendix, for all $m\in\z$.

As usual, we freely identify $\co[[\cg]]$
and $Q(\co[[\cg]])$ with $\co[G][[t]]$ and $Q(\co[G][[t]])$, respectively, via the first isomorphism in (\ref{power-series-iso}).
By abusing notation once again, we let $\psi:\co[G][[t]]\to \co[[t]]$ denote
the unique $\co[[t]]$--algebra morphism extending $\psi: G\to \co$, for all $\psi\in\widehat G(\cp)$.
\begin{remark}
Note that the non zero--divisors of $\co[G][[t]]$ (respectively $\co[G]$) are precisely those elements $f\in \co[G][[t]]$ (respectively $f\in \co[G]$)
with the property that $\psi(f)\ne 0$ in $\co[[t]]$ (respectively in $\co$),
for all $\psi\in\widehat G(\cp)$.
\end{remark}
\noindent We consider the following power series in $\frac{1}{|G|}\co[G][[t]]$:
$$G_S(t):=\sum_{\psi\in\widehat G(\cp)}G_{\psi, S}(t)\cdot e_\psi, \quad H_S(t):=\sum_{\psi\in\widehat G(\cp)}H_{\psi, S}(t)\cdot e_\psi=t\cdot e_{\boldsymbol 1_G}+(1-e_{\boldsymbol 1_G}).$$
Observe that, for all $m\in\z$,  we have equalities
$$\iota(e_\psi)=e_{\psi^{-1}}, \quad t_m(e_\psi)=e_{\psi\omega^{-m}},\quad \iota(t)=(1+t)^{-1}-1,\quad t_m(t)=u^m(1+t)-1,$$
for all characters $\psi\in\widehat G(\cp)$. Consequently, we have the following, for all $m\in\z$.
\begin{eqnarray}
\nonumber  (\iota\circ t_m)(G_S(t)) &=& \sum_{\psi\in\widehat G(\cp)}G_{\psi^{-1}\omega^m, S}(u^m(1+t)^{-1}-1)\cdot e_\psi, \\
\nonumber  (\iota\circ t_m)(H_S(t)) &=& (u^m(1+t)^{-1}-1)\cdot e_{\omega^m} + (1-e_{\omega^m}).
\end{eqnarray}
For every $n\in\z_{\geq 0}$, we let $K_n$ denote the unique field, with $K\subseteq K_n\subseteq \ck$ and $[K_n:K]=p^n$. We let $G_n:=G(K_n/k)$ and note that
$G_n\simeq G\times \Gamma/\Gamma^{p^n}$. We denote by $\pi_n: \co[[\cg]]\twoheadrightarrow \co[G_n]$ the usual (surjective) $\co$--algebra morphism induced by Galois restriction.
By applying the remark above to the group rings $\co[[\cg]]$ and $\co[G_n]$, it is easily seen that $|G|(\iota\circ t_m)(H_S)$ is not a zero--divisor in $\co[[\cg]]$, for all $m\ne 0$.
Moreover,
$\pi_n(|G|(\iota\circ t_m)(H_S))$ is not a zero--divisor in $\co[G_n]$, for all $n$. Let
$$\vu:=\{f\in \co[[\cg]]\mid \pi_n(f)\text{ not a zero--divisor in }\co[G_n], \forall\, n\}, \qquad \vu_n:=\pi_n(\vu)\,.$$
It is easily proved that $\vu_n$ is in fact the set of all non zero--divisors of $\co[G_n]$. As a consequence, we have an equality $\vu_n^{-1}\co[G_n]=Q(\co[G_n])=Q(\co)[G_n]$.
The morphisms  $\pi_n$ above can be uniquely extended to surjective $Q(\co)$--algebra morphisms (for simplicity, also denoted) $\pi_n: \vu^{-1}\co[[\cg]]\twoheadrightarrow Q(\co[G_n])$, whose projective
limit gives an injective $Q(\co)$--algebra morphism
$$\vu^{-1}\co[[\cg]]\hookrightarrow\underset{n}{\underset{\longleftarrow}\lim}\, Q(\co[G_n]).$$
Next, we consider the element
$$g_S:=\frac{(\iota\circ t_1)(G_S(t))}{(\iota\circ t_1)(H_S(t))}=\frac{|G|(\iota\circ t_1)(G_S(t))}{|G|(\iota\circ t_1)(H_S(t))}\in \vu^{-1}\co[[\cg]],$$
and describe its images $\pi_n(g_S)$ in $Q(\co[G_n])$, for all $n$. In order to do that, let
$$\Theta_S^{(n)}(s):=\Theta_{S, K_n/k}(s),\qquad \forall\, n\in\z_{\geq 0}$$
be the $S$--incomplete $G_n$--equivariant $L$--function associated to $K_n/k$. According to Siegel's Theorem \ref{siegel-theorem}, we have
$$\Theta_S^{(n)}(1-m)\in\q[G_n]\subseteq Q(\co[G_n]), \qquad \forall\, n\in\z_{\geq 0}, \quad \forall\, m\in\z_{\geq 1}.$$
Moreover, the inflation property of Artin $L$--functions implies that
$$(\Theta_S^{(n)}(1-m))_n\in\underset{n}{\underset{\longleftarrow}\lim}\, Q(\co[G_n]),\qquad \forall\, m\in\z_{\geq 1}.$$
We let $\Theta_S^{(\infty)}(1-m):=(\Theta_S^{(n)}(1-m))_n$ and view it as an element of $\underset{n}{\underset{\longleftarrow}\lim}\, Q(\co[G_n])$.
\begin{lemma}\label{twisting-gs} For all $m\in\z_{\geq 1}$, we have.
\begin{enumerate}\item $t_{1-m}(g_S)\in\vu^{-1}\co[[\cg]].$
\item $\Theta_S^{(\infty)}(1-m)=t_{1-m}(g_S)$ in $\vu^{-1}\co[[\cg]]$.
\end{enumerate}\end{lemma}

\begin{proof}
Part (1) follows from the observation that the element
$$t_{1-m}(|G|(\iota\circ t_1)(H_S))=|G|(\iota\circ t_m)(H_S)$$
belongs to $\vu$, for all $m\in\z_{\geq 1}$. Part (2) is equivalent to
$$\pi_n(t_{1-m}(g_S))=\Theta^{(n)}_S(1-m),$$
for all $m\in\z_{\geq 1}$ and all $n\in\z_{\geq 0}$.
This is Proposition 4.1 in \cite{Popescu-CS}. Note that in loc.cit. our power series
$(\iota\circ t_1)(G_S(t))$ and $(\iota\circ t_1)(H_S(t))$ are denoted $G_S(t)$ and $H_S(t)$, respectively.
However, $g_S$ denotes the same element in $\vu^{-1}\co[[\cg]]$.
\end{proof}

Next, we use the additional set of primes $T$ to eliminate the denominators of the elements $t_{1-m}(g_S)$ in
$\vu^{-1}\co[[\cg]]$, for all $m\in\z_{\geq 1}$. For that purpose, we consider
$$\delta_T^{(n)}:=\delta_{T, K^{(n)}/k},\quad  \Theta_{S, T}^{(n)}:=\Theta_{S, T, K^{(n)}/k}=\delta_{T, K^{(n)}/k}\cdot\Theta_{S, K^{(n)}/k} ,$$
viewed as holomorphic functions $\C\to\C[G_n]$, for all $n\in\z_{\geq 0}$. Note that we have
$$\delta_T^{(\infty)}(1-m):=(\delta_T^{(n)}(1-m))_n\in \underset{n}{\underset{\longleftarrow}\lim}\, \zp[G_n]=\zp[[\cg]],$$
for all $m\in\z_{\geq 1}$, as an immediate consequence of the definition of the $\delta_T^{(n)}$'s. Consequently, since $T\cap S_p=\emptyset$,
Remark \ref{theta-st-remark} implies that
$$(\Theta_{S, T}^{(n)}(1-m))_n\in \underset{n}{\underset{\longleftarrow}\lim}\, \zp[G_n]=\zp[[\cg]], \qquad \forall\, m\in\z_{\geq 1}.$$
 We let $\Theta_{S, T}^{(\infty)}(1-m):=(\Theta_{S, T}^{(n)}(1-m))_n$ and view these as elements in $\zp[[\cg]]$.
\begin{lemma}\label{twisting-theta-st} For all $m\in\z_{\geq 1}$, we have the following equalities in $\zp[[\cg]]$.
\begin{enumerate}
\item $\delta_T^{(\infty)}(1-m)=t_{1-m}(\delta_T^{(\infty)}(0))$.
\item $\Theta_{S, T}^{(\infty)}(1-m)=t_{1-m}(\delta_T^{(\infty)}(0))\cdot t_{1-m}(g_S)=t_{1-m}(\Theta_{S,T}^{(\infty)}(0)).$
\end{enumerate}
\end{lemma}
\begin{proof}  Let $m$ be as above. From the definitions, we have
$$\delta_T^{(\infty)}(1-m)=\prod_{v\in T}(1-(\sigma_v^{(\infty)})^{-1}\cdot{{\bf N}v}^{m}),$$
where $\sigma_v^{(\infty)}$ is the Frobenius morphism associated to $v$ in $\cg$. Now, since $T\cap S_p=\emptyset$, we have
$c(\sigma_v^{(\infty)})={\bf N}v$, for all $v\in T$. Combined with the last displayed equality, this implies part (1) of the Lemma.
Part (2) is a direct consequence of part (1) combined with Lemma \ref{twisting-gs}(2).
\end{proof}

\begin{corollary}
We have the following.
\begin{enumerate}
\item $g_S\in \vu_{\zp}^{-1}\zp[[\cg]]$, where $\vu_{\zp}:=\vu\cap\zp[[\cg]]$.
\item $G_S\in 1/|G|\zp[[\cg]]$.
\end{enumerate}
\end{corollary}
\begin{proof} Part (1) follows from Lemma \ref{twisting-theta-st}(2) with $m=1$ and the obvious observation that $\delta_T^{(\infty)}(0)\in\vu_{\zp}$.
Part (2) is an immediate consequence of part (1) and the fact that $H_S\in 1/|G|\zp[[\cg]]$.\end{proof}

\begin{definition} The $S$--incomplete, respectively $S$--incomplete $T$--modified, $G$--equivariant $p$--adic $L$--functions $\Theta_S^{(\infty)}\in\vu_{\zp}^{-1}\zp[[\cg]]$ and $\Theta_{S,T}^{(\infty)}\in\zp[[\cg]]$
associated to $(\ck/k, S, T, p)$ are defined by
$$\Theta_S^{(\infty)}:=\Theta_S^{(\infty)}(0)=g_S, \qquad \Theta_{S,T}^{(\infty)}:=\Theta_{S,T}^{(\infty)}(0)\,.$$
For consistency, we let $\delta_T^{(\infty)}:=(\delta_T^{(n)}(0))_n\in\zp[[\cg]]$.
\end{definition}

\begin{remark} Note that, by Lemmas \ref{twisting-gs} and \ref{twisting-theta-st}, we have
$$\Theta_{S,T}^{(\infty)}=\delta_T^{(\infty)}\cdot\Theta_S^{(\infty)}$$
for all $m\in\z_{\geq 1}$. Obviously, the link between $\Theta_{S, T}^{(\infty)}$ and the classical $p$--adic $L$--functions $L_p(\psi, s)$ defined in \eqref{p-adicL} is given by
\begin{equation}\label{link-to-classical-Lp}\chi(\Theta_{S,T}^{(\infty)})=\chi(\delta_T^{(\infty)})\cdot\frac{G_{\chi^{-1}\omega, S}(u(1+t)^{-1}-1)}{H_{\chi^{-1}\omega, S}(u(1+t)^{-1}-1)}\,,
\end{equation}
for all characters $\chi\in\widehat G(\cp)$.
Also, note that since for all odd characters $\psi\in\widehat G_n(\cp)$ we have $G_{\psi, S}=0$ (see Remark \ref{odd-characters}) and $p$ is odd,
we have
$$\Theta_S^{(\infty)}\in \frac{1}{2}(1-j)\cdot\vu_{\zp}^{-1}\zp[[\cg]], \qquad \Theta_{S,T}^{(\infty)}\in\frac{1}{2}(1-j)\cdot\zp[[\cg]].$$
In what follows, we view $\Theta_{S,T}^{(\infty)}$ and $\Theta_S^{(\infty)}$ as elements of both $1/2(1-j)\zp[[\cg]]$ and $\zp[[\cg]]^-=\zp[[\cg]]/(1+j)$
via the ring isomorphism $1/2(1-j)\zp[[\cg]]\simeq \zp[[\cg]]^-$ given by reduction modulo $(1+j)$.
\end{remark}

\begin{remark}\label{no-mup}
In this section we assumed that $\bmu_p\subseteq  K$ (i.e. $\bmu_{p^\infty}\subseteq \ck$). This hypothesis was used whenever the twisting
automorphisms $t_m$ of $\co[[\cg]]$ with respect to the various powers $c^m$ of the $p$--cyclotomic character $c$ of $\cg$ were needed.

\noindent However, even if $\bmu_p\not\subseteq K$, we can still construct any object $\ast$ in the list $G_S(t)$, $H_S(t)$, $\Theta_S^{(n)}(s)$, $\delta_T^{(n)}(s)$, $\Theta_{S,T}^{(n)}(s)$ for the data $(\ck/k, S, T, p)$, just as above. Let $\widetilde K:=K(\mu_p)$, $\widetilde\ck=\ck(\mu_p)$, $\widetilde G:=G(\widetilde K/k)$, $\widetilde G_n:=G(\widetilde K_n/k)$, $\widetilde\cg:=G(\widetilde\ck/k)$. Further,
let $\widetilde\ast$ denote the analogue of $\ast$ for $(\widetilde\ck/k, S, T, p)$. By abuse of notation, let
$\pi:\C[\widetilde G_n]\twoheadrightarrow \C[G_n]$, $\pi: \co[[\widetilde\cg]]\twoheadrightarrow \co[[\cg]]$, etc., denote the $\co$--algebra morphisms induced by Galois restriction. The inflation property of Artin $L$--functions
immediately implies that
$$\pi(\widetilde\ast)=\ast\,,$$
for all $\ast$ as above. Consequently, we can construct equivariant $p$--adic $L$--functions $\Theta_S^{(\infty)}:=(\Theta_S^{(n)}(0))_n\in\vu_{\zp}^{-1}\zp[[\cg]]^-$ and $\Theta_{S,T}^{(\infty)}:=(\Theta_{S,T}^{(n)}(0))_n\in\zp[[\cg]]^-$, for $(\ck/k, S, T, p)$. Obviously, these will satisfy
$$\pi(\widetilde\Theta_{S, T}^{(\infty)})=\Theta_{S, T}^{(\infty)}.$$
\end{remark}

\subsection{\bf The Proof of the Equivariant Main Conjecture}\label{emc-section} In this section, we prove Theorem \ref{emc}. Consequently,
we work under the hypothesis $\mu_{\ck}=0$. In particular, we have $\mu_{\zp}(\cxs^+)=0$. (Recall from classical Iwasawa theory that $\mu_{\zp}(\cxs^+)=\mu_{K,p}^-$, for all
sets $S$ as above.)

\begin{lemma}
It suffices to prove Theorem \ref{emc} under the assumption that $\bmu_p\subseteq\ck$.
\end{lemma}
\begin{proof} Assume that $\bmu_p\not\subseteq\ck$. With notations as in remark \label{no-mup} above, let $\Delta:=G(\widetilde \ck/\ck)$.
Also let $\cm:=\mk$ and $\widetilde{\cm}:=\cm^{\widetilde\ck}_{\widetilde\cs, \widetilde\ct}$, where $\widetilde\cs$ and $\widetilde\ct$ are the sets of
primes in $\widetilde\ck$ sitting above primes in $\cs$ and $\ct$, respectively. Since
$$\ker(\pi: \zp[[\widetilde\cg]]^-\twoheadrightarrow \zp[[\cg]]^-)=I_{\Delta}\zp[[\widetilde\cg]]^-,$$
where $I_{\Delta}$ is the augmentation ideal in $\zp[\Delta]$, Corollary \ref{coinvariants} implies that we have isomorphisms
of $\zp[[\cg]]^-$--modules
$$T_p(\cm)^-\simeq T_p(\widetilde\cm)^-/I_{\Delta} T_p(\widetilde\cm)^-\simeq T_p(\widetilde\cm)^-\otimes_{\zp[[\widetilde\cg]]^-}\zp[[\cg]]^-.$$
Consequently, the second equality in \eqref{fitt-base-change} (Appendix) implies that we have
$${\rm Fit}_{\zp[[\cg]]^-}T_p(\cm)^-=\pi({\rm Fit}_{\zp[[\widetilde\cg]]^-}T_p(\widetilde\cm)^-).$$
Now the last equality in Remark \ref{no-mup} shows that Theorem \ref{emc} for the
data $(\widetilde\ck/k, S, T, p)$ implies the same result for the data $(\ck/k, S, T, p)$.
\end{proof}
\noindent As a consequence of the above Lemma, we may and will assume that $\bmu_p\subseteq\ck$ throughout the rest of this section. Now, note that Theorem \ref{projective}(1) shows that the $\zp[G]^-$--module
$T_p(\mk)^-$ is projective. Consequently, if we apply Proposition \ref{fitting-calculation}(1) (see Appendix) for the semi-local ring $R:=\zp[G]^-$, $g:=\gamma$ and the $R[[\Gamma]]$--module $M:=T_p(\mk)^-$, we conclude that Theorem \ref{emc} is equivalent to the following association in divisibility in $\zp[[\cg]]^-\simeq \zp[G]^-[[t]]$:
\begin{equation}\label{equivalent-association} \Theta_{S, T}^{(\infty)}\sim {\rm det}_{\zp[G]^-}((t+1)-\mathfrak m_\gamma\mid T_p(\mk)^-).\end{equation}

\begin{lemma}\label{G}For all odd characters $\chi\in\widehat G(\cp)$, we have the following in $\co[[t]]$.\begin{enumerate}
\item $G_{\chi^{-1}\omega, S}(u(1+t)^{-1}-1)\sim \chi\left({\rm det}_{Q(\co)[G]}((t+1)-\mathfrak m_\gamma\mid T_p(\cm^{\ck}_{\cs, \emptyset})^-\otimes_{\zp}Q(\co))\right).$
\item $\mu(G_{\chi^{-1}\omega, S}(u(1+t)^{-1}-1))=0$, with notations as in Corollary \ref{corollary-association}.
\end{enumerate}
\end{lemma}
\begin{proof} (1) Combine Corollary \ref{Wiles+zero-mu}(2)
and Lemma \ref{twist-poly-lemma}(3), applied to the lattice
$\cl:=\cxs^+\otimes_{\zp}\co$ and $n=1$ to conclude that
$$G_{\chi^{-1}\omega, S}(u(1+t)^{-1}-1)\sim \chi\left({\rm det}_{Q(\co)[G]}((t+1)-\mathfrak m_\gamma\mid (\cxs^+)^\ast(1)\otimes_{\zp}Q(\co))\right ).$$
Now use the $\zp[[\cg]]$--module isomorphism $T_p(\cm^{\ck}_{\cs, \emptyset})^-\simeq (\cxs^+)^\ast(1)$ given by Lemma \ref{link-classical} to conclude the proof of (1).

(2) This follows from (1) and the fact that the right-hand-side in (1) is a monic polynomial in $\co[t]$ (see Lemma \ref{twist-poly-lemma} in the Appendix.)
\end{proof}

\begin{lemma}\label{H} If $Q(\co)$ is endowed with the trivial $\cg$--action, then:
\begin{enumerate} \item $H_S(t)={\rm det}_{Q(\co)[G]}((t+1)-\mathfrak m_{\gamma}\mid Q(\co))$.
\item For all odd $\chi\in\widehat G(\cp)$, we have an association in divisibility in $\co[[t]]$:
$$H_{\chi^{-1}\omega, S}(u(1+t)^{-1}-1)\sim \chi({\rm det}_{Q(\co)[G]}((t+1)-\mathfrak m_{\gamma}\mid Q(\co)(1)))\,.$$
\item For all odd $\chi\in\widehat G(\cp)$, we have $\mu(H_{\chi^{-1}\omega, S}(u(1+t)^{-1}-1))=0$, with notations as in Corollary \ref{corollary-association}.
\end{enumerate}\end{lemma}
\begin{proof} Part (1) is clear from the equality $H_S(t):=t\cdot e_{\mathbf 1_G}+ (1-e_{{\mathbf 1}_G})$. Part (2) is a direct consequence of part (1),
and Lemma \ref{twist-poly-lemma}(3) applied to $\cl:=\co$ (with trivial $\cg$--action), $V=Q(\co)$, and $n=1$. Observe that $Q(\co)^\ast(1)\simeq Q(\co)(1)$.
Part (3) follows from (2) and the fact that the right-hand-side in (2) is a monic polynomial in $\co[t]$ (see Lemma \ref{twist-poly-lemma} in the Appendix.)
\end{proof}
\begin{lemma}\label{delta}
The $\zp[[\cg]]$--module $T_p(\Delta_{\ck, \ct})$ satisfies the following.
\begin{enumerate}
\item $T_p(\Delta_{\ck, \ct})\simeq\bigoplus_{v\in T}\zp[[\cg]]/(\delta_v^{(\infty})$, where $\delta_v^{(\infty)}:=(1-(\sigma_v^{(\infty})^{-1}\cdot{\bf N}v)$.
\item ${\rm Fit}_{\zp[[\cg]]} T_p(\Delta_{\ck, \ct})=(\delta_T^{(\infty})$.
\item ${\rm pd}_{\zp[[\cg]]}T_p(\Delta_{\ck, \ct})=1$ and ${\rm pd}_{\zp[G]}T_p(\Delta_{\ck, \ct})=0$.
\item $\delta_T^{(\infty)}\sim {\rm det}_{\zp[G]}((t+1)-\mathfrak m_\gamma\mid T_p(\Delta_{\ck, \ct}))$ in $\zp[[\cg]]$
\item $\mu(\chi(\delta_T^{(\infty)}))=0$, for all $\chi\in\widehat G(\cp)$.
\end{enumerate}
\end{lemma}
\begin{proof}
 (1) For every $v\in T$, let us fix a prime $w(v)\in \ct$ sitting above $v$. Then, we have an obvious isomorphism of $\z[[\cg]]$--modules
$$\Delta_{\ck, \ct}\simeq \bigoplus_{v\in T}\left(\kappa(w(v))^\times\otimes_{\z[[\cg_v]]}\z[[\cg]]\right),$$
where $\cg_v$ is the decomposition group associated to $v$ in $\cg$ (topologically generated by
$\sigma_v^{(\infty})$) and $\kappa(w(v))$ denotes, as usual, the residue field associated to $w(v)$. This induces an isomorphism of $\zp[[\cg]]$--modules
$$T_p(\Delta_{\ck, \ct})\simeq \bigoplus_{v\in T}\left(T_p(\kappa(w(v))^\times)\otimes_{\zp[[\cg_v]]}\zp[[\cg]]\right)\,.$$
The isomorphism above combined with Lemma \ref{twisting-Fitting}(4) (see Appendix) applied to the extension $\kappa(w(v))/\kappa(v)$ of Galois group $\cg_v$,
for $q:={\rm card}(\kappa(v))$ and $\sigma_q=\sigma_v^{(\infty)}$ concludes the proof of (1).

(2) is a direct consequence of (1) and the definition of the Fitting ideal.

(3)
Note that if $\alpha\in\zp$, such that $\sigma_v^{(\infty)}=\gamma^\alpha$ and $\sigma_v$ is the Frobenius associated to $v$ in $G$,
then we have
$$\delta_v^{(\infty)}\sim (\sigma_v\cdot(t+1)^\alpha-{\bf N}v)\quad  \text{in $\zp[[\cg]]\simeq\zp[G][[t]]$}.$$
Consequently,
$\chi(\delta_v^{(\infty)})=(\chi(\sigma_v)(1+t)^\alpha-{\bf N}v)$ is not a zero divisor in $\zp(\chi)[[t]]$, for all $\chi\in\widehat G(\cp)$ and therefore
$\delta_v^{(\infty)}$ is not a zero divisor in $\zp[[\cg]]$, for all $v\in T$. This observation combined with the isomorphism in part (1)
leads to the equality ${\rm pd}_{\zp[[\cg]]}T_p(\Delta_{\ck, \ct})=1$. Now, in order to obtain the equality ${\rm pd}_{\zp[G]}T_p(\Delta_{\ck, \ct})=0$, one applies Lemma \ref{projective-dimension}
for $M:=T_p(\Delta_{\ck, \ct})$ and $H:=G$, keeping in mind that $T_p(\Delta_{\ck, \ct})$ is $\zp$--free.

(4) This is a direct consequence of (2) and (3) combined with Proposition \ref{fitting-calculation}(1) applied for $R:=\zp[G]$, $g:=\gamma$ and the $R[[\Gamma]]=\zp[[\cg]]$--module
$M:=T_p(\Delta_{\ck, \ct})$.

(5) This is a direct consequence of (4) and the fact that the right-hand-side in (4) is a monic polynomial in $\zp[G][[t]]$ (see Remark \ref{monicity-remark} in the Appendix.)
\end{proof}
\begin{remark}\label{remark-delta} Note that the proofs of (1), (2) and the first equality in (3) of the above Lemma are independent on our
assumption that $\cg\simeq G\times \Gamma$.
\end{remark}

Now, we are ready to prove \eqref{equivalent-association} above. Let
$$\Theta:=\Theta_{S, T}^{(\infty)},\qquad F:= {\rm det}_{\zp[G]^-}((t+1)-\mathfrak m_\gamma\mid T_p(\mk)^-).$$
We view $\Theta$ and $F$ as power series in $\zp[G]^-[[t]]$ and plan on using Corollary \ref{corollary-association} (see Appendix) to show that $\Theta\sim F$. Note that, in fact,
$F$ is a monic polynomial in $\zp[G]^-[t]$ (see Remark \ref{monicity-remark} in the Appendix). Consequently, with notations as in Corollary \ref{corollary-association}, we have
\begin{equation}\label{mu-F}\mu(\chi(F))=0,\qquad  \forall\, \chi\in\widehat G(\cp), \text{ $\chi$ odd.}\end{equation}
Now, Lemma \ref{G}(2), Lemma \ref{H}(3) and Lemma \ref{delta}(5) combined with equality \eqref{link-to-classical-Lp} above show that
\begin{equation}\label{mu-Theta}\mu(\chi(\Theta))=0,\qquad  \forall\, \chi\in\widehat G(\cp), \text{ $\chi$ odd.}\end{equation}
Now, exact sequence \eqref{sequence-empty-T} combined with equality \eqref{link-to-classical-Lp} and Lemma \ref{G}(1), Lemma \ref{H}(2) and Lemma \ref{delta}(4) show that the following holds in $\co[[t]]$.
\begin{equation}\label{Theta-F}\chi(\Theta)\sim \chi\left( {\rm det}_{Q(\co)[G]^-}((t+1)-\mathfrak m_\gamma\mid T_p(\mk)^-\otimes_{\zp}Q(\co)\right)=\chi(F),\end{equation}
for all odd $\chi\in\widehat G(\cp)$. Now, as a consequence of \eqref{mu-F}, \eqref{mu-Theta} and \eqref{Theta-F}, Corollary \ref{corollary-association} in the Appendix leads to
$$\Theta\sim F \text{ in $\zp[G]^-[[t]]$,}$$
which concludes the proof of \eqref{equivalent-association} and that of Theorem \ref{emc}. \hfill $\Box$
\begin{corollary}\label{full-group-ring}
Under the hypotheses of Theorem \ref{emc}, we have
$${\rm Fit}_{\zp[[\cg]]}((T_p(\mk)^-)^\ast)={\rm Fit}_{\zp[[\cg]]}(T_p(\mk)^-)=\left(\Theta_{S,T}^{(\infty)},\quad \frac{1}{2}(1+j)\right)$$
where $(T_p(\mk)^-)^\ast:={\rm Hom}_{\zp}(T_p(\mk)^-, \zp)$ with the coinvariant
$\cg$--action.
\end{corollary}
\begin{proof}
Theorem \ref{projective}(1) permits us to apply Proposition \ref{fitting-calculation}(3) (see Appendix) to $M:=T_p(\mk)^-$. The latter implies the first equality above.
The second equality is a direct consequence of Theorem \ref{emc}. Note that in the statement of the Corollary $\Theta_{S, T}^{(\infty)}$ is viewed as an element of
$\zp[[\cg]]$.
\end{proof}
\section{Applications of the Equivariant Main Conjecture}\label{applications-section}

In this section, we give two application of Theorem \ref{emc} proved above: the first is a proof of a refinement of an imprimitive version of the Brumer--Stark Conjecture,
away from its $2$--primary part (see Theorem \ref{refined-BS} below); the second is a proof of a refinement of the Coates-Sinnott Conjecture, away from its $2$--primary part
(see Theorem \ref{cs-theorem} and Corollary \ref{cs-corollary} below.) Of course, these theorems are proved under the assumption that the relevant classical Iwasawa $\mu$--invariants vanish.

In what follows, if $M$ is a $\zp[[\mathcal H]]$--module, for some prime $p$ and some abelian profinite group $\mathcal H$,
then $M^\vee:={\rm Hom}_{\zp}(M, \qp/\zp)$ and $M^\ast:={\rm Hom}_{\zp}(M, \zp)$ denote its Pontrjagin and $\zp$--module dual, respectively. depending on the context, $M^\vee$ and $M^\ast$ will
be endowed either with the covariant $\mathcal H$--action $({}^\sigma f)(m):=f({}^\sigma m)$ or the contravariant one $({}^\sigma f)(m):=f({}^{\sigma^{-1}} m)$, for all $m\in M$, $\sigma\in\mathcal H$, and all $f\in M^\vee$ and $f\in M^\ast$, respectively.

Throughout this section, we let $K/k$ be an abelian extension of number fields of Galois group $G$ and let $S$ be a fixed finite set of primes in $k$, containing
the set $S_{\rm ram}(K/k)\cup S_\infty$. We will use the notations introduced at the beginning of \S\ref{EMC}.

\subsection{\bf The Brumer-Stark Conjecture} We fix a non-empty, finite set $T$ of primes in $k$, such that  $S\cap T=\emptyset$. We assume that
$T$ contains either at least two primes of distinct residual characteristics or a prime of residual characteristic coprime to $w_K:=|\mu_K|$, where $\mu_K$ denotes the group of roots of unity in $K$.
(It is easily seen that this last condition is equivalent to the non-existence of non trivial roots of unity in $K$ which are congruent to $1$ modulo all primes in $T$.)

For simplicity, we let $\Theta_{S,T}:=\Theta_{S, T, K/k}$ denote
the $S$--incomplete, $T$--modified $G$--equivariant $L$--function associated to $(K/k, S, T)$ at the beginning of \S\ref{EMC}. Note that, under
our assumptions, Remark \ref{theta-st-remark} for $m=1$ gives
$$\Theta_{S, T}(0)\in\z[G].$$
In order to simplify notations, we let $S$ and $T$ also denote the sets of primes in $K$ sitting above primes in the original sets $S$ and $T$, respectively. Similarly, depending on the context, $S_\infty$ will denote the set of infinite primes in either $K$ or $k$. As in \S4.3 of \cite{Popescu-PCMI}, to the set of data $(K,T)$, one can associate a $T$--modified Arakelov class-group (Chow group) ${\rm CH}^1(K)^0_T$, endowed with a natural $\z[G]$--module structure, as follows. First, one defines the divisor group
$${\rm Div}(K)_T:=(\bigoplus_{w\not\in S_\infty\cup T}\z\cdot w)\bigoplus(\bigoplus_{w\in S_{\infty}}\Bbb R \cdot w)\,,$$
where the direct sums are taken over all the primes $w$ in $K$. Then, one defines
$${\rm deg}_{K, T}: {\rm Div}(K)_T\to \Bbb R$$
to be the unique degree map which is $\z$--linear on the left direct summand and $\Bbb R$--linear on the right and satisfies
$${\rm deg}_{K, T}(w):=\left\{
                      \begin{array}{ll}
                        \log|{\bf N}w|, & \hbox{if $w$ is a finite prime;} \\
                        1 & \hbox{if $w$ is an infinite prime.}
                      \end{array}
                    \right.$$
We let ${\rm Div}(K)^0_T:=\ker({\rm deg}_{K, T})$. Next, one defines a divisor map
$${\rm div}_K: K_T^\times\to {\rm Div}(K)^0_T, \quad {\rm div}_K(x):=\sum_{w\not\in S_\infty\cup T}{\rm ord_w}(x)\cdot w +\sum_{w\in S_{\infty}}(-\log|x|_w)\cdot w,$$
where ${\rm ord}_w$ and $|\cdot|_w$ denote the canonically normalized valuation and metric associated to $w$, respectively and $K_T^\times$ denotes the subgroup of $K^\times$ consisting of elements congruent to $1$ modulo all primes in $T$. Finally, one defines
$${\rm CH}^1(K)_T^0:=\frac{{\rm Div}(K)^0_T}{{\rm div}_K(K_T^\times)}\,.$$
If endowed with the natural quotient topology, this is a compact group whose volume is the absolute value of the leading term $\zeta_{S, T, k}^\ast(0)$
at $s=0$ of the modified zeta function $\zeta_{S,T,k}:=\Theta_{S, T, k/k}: \C\to \C$ associated to $k$ (see loc.cit.)

\begin{remark} Note the difference, both in notation and definition, between the finite divisor group ${\mathcal Div}_{K,T}$ and divisor map ${div}_{K}$ defined
in \S\ref{ideal-class-groups} and, respectively, the Arakelov divisor group ${\rm Div}(K)_T$ and divisor map ${\rm div}_{K}$ defined above. The link between the finite $T$--modified class-group
$C_{K,T}$ defined in \S\ref{ideal-class-groups}  and the compact $T$--modified Arakelov class--group ${\rm CH}^1(K)^0_T$ defined above is captured by the following obvious exact sequence of $\z[G]$--modules.
$$\xymatrix{ 0\ar[r] &\dfrac{{\rm Div}(K)^0(S_\infty)}{{\rm div}_{K}(U_{K,T})}\ar[r] & {\rm CH}^1(K)^0_T\ar[r] & C_{K, T}\ar[r] &0.
}$$
Above, ${\rm Div}(K)^0(S_\infty):=\bigoplus_{v\in S_\infty}\Bbb R\cdot v$ denotes the $\Bbb R$--vector space of Arakelov divisors of degree $0$ supported on $S_\infty$ (denoted by $\Bbb RX_{S_\infty}$ in \cite{Popescu-PCMI})
and $U_{K,T}$ denotes the group of units in $K$ which are congruent to $1$ modulo all the primes in $T$.
(See \S4.3 of \cite{Popescu-PCMI} for the exact sequence above.)
\end{remark}

In \S4.3 of \cite{Popescu-PCMI}, we proved that the Brumer--Stark Conjecture ${\bf BrSt}(K/k, S)$ for the data $(K/k, S)$ (as stated, for example, in Chpt. IV, \S6 of \cite{Tate-Stark})
is equivalent to the following statement.

\begin{conjecture}[Brumer-Stark]\label{BS} Let $(K/k, S)$ be as above. Then, for all sets $T$ satisfying the above conditions, we have
\begin{equation*}{\bf BrSt}(K/k, S, T):\qquad \Theta_{S, T}(0)\in{\rm Ann}_{\z[G]} {\rm CH}^1(K)^0_T\,.\end{equation*}
\end{conjecture}
\medskip

\begin{remark}\label{BS-reduction-remark} Let us fix $(K/k, S, T)$ as above. It is easily seen that if $S$ contains at least two primes which split completely in $K/k$, then $\Theta_{S,T}(0)=0$.
 If $S$ contains exactly one prime which splits completely in $K/k$, then $\Theta_{S,T}(0)=0$ unless $k$ is
an imaginary quadratic field, $K/k$ is unramified everywhere and $S=S_\infty$. In the latter case, the proof of Conjecture ${\bf BrSt}(K/k, S, T)$ is an easy, albeit instructive exercise: Indeed, according
to Corollary 4.3.3 (top change) in \cite{Popescu-PCMI}, one can assume that $K$ is the Hilbert class--field of $k$. In that case, it is easily seen that
$$\Theta_{S,T}(0)=\dfrac{\prod_{v\in T}(1-{\bf N}v)}{w_k}\cdot {{\bf N}_G},$$
where $v$ runs through primes in $k$, ${\bf N}_G:=\sum_{\sigma\in G}\sigma$ is the usual norm element in $\z[G]$, and $w_k$ is the cardinality of the group of roots of unity $\bmu_k$ in $k$. Now, the desired
annihilation result follows from the fact that ideals in $k$ become principal in $K$ (Hilbert's capitulation theorem) and the exact sequence of abelian groups
$$\xymatrix{1\ar[r] &\bmu_k\ar[r] &\Delta_{k, T}\ar[r]  &C_{k, T}\ar[r] &C_k\ar[r] &1.
}$$
(See exact sequence \eqref{t-sequence} and note that $U_k=\bmu_k$ in this case.)
We leave the remaining details to the interested reader.
As a consequence, it suffices to study the Brumer-Stark Conjecture in the case where {\bf $k$ is totally real and $K$ is totally imaginary.} (Otherwise,
$S$ will contain at least one infinite prime which splits completely in $K/k$ and the above considerations settle the conjecture in that case.)
\end{remark}

Now, Proposition 4.3.7 in \cite{Popescu-PCMI} provides us with the following additional reduction.

\begin{proposition}\label{BS-reduction-CM} Let $(K/k, S, T)$ be as above. Assume that $k$ is totally real and $K$ is totally imaginary. Let $K^{CM}$ be the maximal
CM subfield of $K$, let $G_{CM}:=G(K^{CM}/k)$ and $\widetilde \Theta_{S, T}:=\Theta_{S, T, \widetilde K/k}.$ Then, the following
are equivalent, for all primes $p>2$.
\begin{enumerate}
\item $\Theta_{S, T}(0)\in{\rm Ann}_{\zp[G]} ({\rm CH}^1(K)^0_T\otimes\zp)$.
\item $\widetilde\Theta_{S, T}(0)\in{\rm Ann}_{\zp[G_{CM}]} (C_{K^{CM}, T}\otimes\zp)^-$.
\end{enumerate}
\end{proposition}
\begin{proof} See Proposition 4.3.7 in \cite{Popescu-PCMI}. The upper script ``${}^-$'' in (2) refers to the action of the unique complex conjugation
morphism in $G^{CM}$ on $(C_{K^{CM}, T}\otimes\zp)$.
\end{proof}

As a consequence of the Equivariant Main Conjecture (Theorem \ref{emc}), we can prove the following refinement of statement (2) in the
Proposition above, under certain hypotheses (see below.)
In what follows, $S_p$ denotes the set of $p$--adic primes in $k$ and $\mu_{K,p}$
denotes the Iwasawa $\mu$--invariant associated to the number field $K$ and the prime $p$.

\begin{theorem}\label{refined-BS} Let $(K/k, S, T)$ be as above and let $p$ be an odd prime. Assume that $K$ is CM, $k$ is totally real, $S_p\subseteq S$ and
$\mu_{K,p}=0$. Then, we have
$$\label{theta-fitting} \overline{\bf BrSt}(K/k, S, T, p):\qquad \Theta_{S, T}(0)\in {\rm Fit}_{\zp[G]}\, ((A_{K,T}^-)^\vee),$$
where $A_{K,T}:=(C_{K,T}\otimes\zp)$ and
the dual is endowed with the covariant $G$--action.
\end{theorem}

\begin{proof} In what follows, all occurring $\zp$--module or Pontrjagin duals are endowed with the {\bf covariant} action by the appropriate groups.

We let $\ck$ be the cyclotomic $\zp$--extension of $K$ and $\cg:={\rm Gal}(\ck/k)$. As usual, we let $\cs$ and $\ct$ denote the sets of finite primes
in $\ck$ sitting above primes in $S$ and $T$, respectively. Note that the hypotheses of Theorem \ref{emc} are satisfied by the data $(\ck/k, S, T, p)$.
Consequently, Corollary \ref{full-group-ring} implies that
\begin{equation}\label{theta-fitting}\Theta_{S, T}^{(\infty)}\in {\rm Fit}_{\zp[[\cg]]}((T_p(\mk)^-)^\ast)\,.\end{equation}
Now, recall that, with notations as in \S\ref{ideal-class-groups}  and under the current hypotheses,
the $\zp[[\cg]]$--module
$$\ca_{\ck, \ct}^-\simeq\, \underset{n}{\underrightarrow\lim}\, A_{K_n, T_n}^-$$
is a torsion, divisible $\zp$--module of finite co-rank and that the transition maps in the injective limit above
are injective (see Lemma \ref{no-capitulation}.) Consequently, since $A_{K,T}^-$ is finite, we can fix $n\in\Bbb N$, such that we have an inclusion  of $\zp[[\cg]]$--modules
$$A_{K,T}^- \subseteq \ca_{\ck, \ct}^-[p^n].$$
The inclusion above induces a natural surjection of $\zp[[\cg]]$--modules
 \begin{equation}\label{surjection}\ca_{\ck, \ct}^-[p^n]^{\,\vee}\twoheadrightarrow (A_{K,T}^-)^\vee. \end{equation}
Next, we need the following elementary result.
\begin{lemma}\label{Pontrjakin-zp-dual} Let $p$ be a prime, $\mathcal H$ a profinite abelian group, and $M$ a $\zp[[\mathcal H]]$--module. Assume that $M$ is $\zp$--torsion, divisible, of finite
corank. Then there exist canonical isomorphisms of $\zp[[\mathcal H]]$--modules
$$M[p^m]^{\,\vee} \simeq T_p(M)^\ast\otimes_{\zp}\z/p^m\z\,, \qquad\text{ for all }m\in\z_{\geq 1}.$$
\end{lemma}
\begin{proof} Fix $m\in\z_{\geq 1}$. There is a canonical isomorphism of $\zp[[\mathcal H]]$-modules
$$T_p(M)^\ast\simeq M^\vee.$$ This induces a
canonical isomorphism $T_p(M)^\ast/p^m\simeq M^\vee/p^m$. However, the exact functor $\ast\to {\rm Hom}_{\zp}(\ast, \qp/\zp)$
applied to the exact sequence
$$\xymatrix {0\ar[r] &M[p^m]\ar[r] &M\ar[r]^{\times p^m} &M\ar[r] &0}$$
induces an isomorphism $M^\vee/p^m\simeq M[p^m]^\vee$. This concludes the proof.
%If we apply the functor $\ast\to {\rm Hom}_{\zp}(\ast, \zp)$ to
%the exact sequence
%$$\xymatrix{0\ar[r] &T_p(M)\ar[r]^{p^m} &T_p(M)\ar[r]
%&M[p^m]\ar[r] &0 },$$ we obtain a $\zp[[\mathcal]]$--module isomorphism
%$$T_p(M)^\ast\otimes_{\zp}\z/p^m\z\simeq {\rm Ext}^1_{\zp}(M[p^m],
%\zp).$$ Similarly, if we apply the functor $\ast\to {\rm
%Hom}_{\zp}(M[p^m], \ast)$ to the exact sequence
%$$\xymatrix{0\ar[r] &\zp\ar[r] &\qp\ar[r]
%&\qp/\zp\ar[r] &0 },$$ we obtain a $\zp[[\mathcal H]]$--module isomorphism
%$$M[p^m]^\vee\simeq {\rm Ext}^1_{\zp}(M[p^m],
%\zp).$$ Now, we combine the two isomorphisms above to conclude the proof of the Lemma.
\end{proof}

\noindent
Consequently, we obtain the following surjective morphisms of $\zp[[\cg]]$--modules.
\begin{equation}\label{surjections}(T_p(\mk)^-)^\ast\twoheadrightarrow T_p(\ca_{\ck, \ct}^-)^\ast\twoheadrightarrow \ca_{\ck, \ct}^-[p^n]^{\,\vee}\twoheadrightarrow (A_{K,T}^-)^\vee.\end{equation}
The first surjection above is obtained by taking $\zp$--duals in exact sequence \eqref{motive-exact-sequence}.
Note that $T_p(\ca_{\ck, \ct}^-)=T_p(\ca_{\ck, \ct})^-$. Also, note that the exact sequence \eqref{motive-exact-sequence} is split in the category of $\zp$--modules (because ${\mathcal Div}_{\ck}(\cs\setminus\cs_p)^-\otimes\zp$ is $\zp$--free) and therefore it stays exact after taking $\zp$--duals. The second surjection is given by Lemma \ref{Pontrjakin-zp-dual} applied to $M:=\ca_{\ck, \ct}^-$, $m:=n$, $\mathcal H:=\cg$. Finally, the third surjection is
\eqref{surjection} above.

Now, apply the first part of \eqref{fitt-base-change} in the Appendix for $R:=\zp[[\cg]]$, $M:=(T_p(\mk)^-)^\ast$ and $M':=(A_{K,T}^-)^\vee$, in combination with the surjections \eqref{surjections}
as well as \eqref{theta-fitting} above to conclude that we have
$$\Theta_{S,T}^{(\infty)}\in {\rm Fit}_{\zp[[\cg]]}T_p(\mk)^- = {\rm Fit}_{\zp[[\cg]]}(T_p(\mk)^-)^\ast \subseteq {\rm Fit}_{\zp[[\cg]]}\, ((A_{K,T}^-)^\vee).$$
Now, consider the projection $\pi:\zp[[\cg]]\twoheadrightarrow\zp[G]$ given by Galois restriction. By the definition of $\Theta_{S, T}^{(\infty)}$, we have
$$\pi(\Theta_{S,T}^{(\infty)})=\Theta_{S,T}(0),\qquad A_{K,T}^-\simeq A_{K,T}^-\otimes_{\zp[[\cg]]}\zp[G],$$
where the isomorphism is viewed in the category of $\zp[G]$--modules. Consequently, \eqref{fitt-base-change} (see Appendix)
applied for $M:=A_{K,T}^-$ and $\rho:=\pi$ implies that
$$\Theta_{S,T}(0)\in {\rm Fit}_{\zp[G]}\, ((A_{K,T}^-)^\vee),$$
which concludes the proof of the Theorem.
\end{proof}
\begin{corollary}
Assume that $(K/k, S, T)$ satisfy the hypotheses in Conjecture \ref{BS}. Let $p$ be a prime. If $k$ is totally real and $K$ is totally imaginary,  assume that
$p$ is odd, $S_p\subseteq S$ and $\mu_{K^{CM},p}=0$. Then, we have:
$${\bf BrSt}(K/k, S, T, p):\qquad \Theta_{S,T}(0)\in{\rm Ann}_{\zp[G]}({\rm CH}^1(K)_T^0\otimes\zp).$$
\end{corollary}
\begin{proof} Combine Theorem \ref{refined-BS} with Proposition \ref{BS-reduction-CM} and Remark \ref{BS-reduction-remark} and note that
$${\rm Fit}_{\zp[G]}(M^\vee)\subseteq {\rm Ann}_{\zp[G]}(M),$$
for any $\zp[G]$--module $M$, if $M^\vee$ is endowed with the covariant $G$--action.
\end{proof}

\begin{remark} Some cases of the Brumer-Stark Conjecture over an arbitrary totally real number field were also settled by the first
author in \cite{Greither-Fitting} with different methods and working under somewhat more restrictive hypotheses. See Theorem 10 in loc.cit.
\end{remark}

\subsection{The Coates-Sinnott Conjecture.} As usual, $K/k$ is an abelian extension of number fields of
Galois group $G$, and $S$ is a finite set of primes in $k$, such that $S_{\rm ram}(K/k)\cup S_\infty\subseteq S$.

If $p$ is a prime number, we let ${\rm H}^i_{et}(\co_{K, S}[1/p], \zp(n))$ denote the $i$--th \'etale cohomology
group of the affine scheme ${\rm Spec}(\co_{K,S}[1/p])={\rm Spec}(\co_{K,S\cup S_p})$ with coefficients in the \'etale $p$--adic sheaf $\zp(n)$, for all $i\in\z_{\geq 0}$ and
all $n\in\z_{\geq 2}$. Also, for every $m\in\z_{\geq 0}$, we let ${\rm K}_m(\co_{K,S})$ denote the $m$--th Quillen ${\rm K}$--group of $\co_{K,S}$.
For the definitions and properties of these \'etale cohomology groups and ${\rm K}$--groups, the reader may consult Kolster's excellent survey article \cite{Kolster}.
For a discussion closer in spirit to the current section, the reader may consult \cite{Popescu-CS} as well. In the present context, all these \'etale cohomology and ${\rm K}$--theory groups come endowed with natural
$\zp[G]$--module, respectively $\z[G]$--module structures.

\noindent Quillen showed in \cite{Quillen} that ${\rm K}_m(\co_{K,S})$ is a finitely generated abelian group, for all $m$. Borel showed in \cite{Borel} that we have the following remarkable equalities
\begin{equation}\label{Borel}{\rm rank}_{\z}\, {\rm K}_{2n-i}(\co_{K,S})=\left\{
                                            \begin{array}{ll}
                                              0, & \hbox{if $i=2$;} \\
                                              {\rm ord}_{s=(1-n)}\, \zeta_{K, S}(s), & \hbox{if $i=1$,}
                                            \end{array}
                                          \right.\end{equation}
for all $n\geq 2$, where $\zeta_{K,S}$ is the $S$--incomplete zeta--function associated to $K$. For all primes $p>2$, Soul\'e \cite{Soule} and later Dwyer-Friedlander \cite{Dwyer-Friedlander}
constructed surjective
$\zp[G]$--linear $p$--adic Chern character morphisms
\begin{equation}\label{chern}
{\rm ch}^i_{p,n}: {\rm K}_{2n-i}(\co_{K,S})\otimes\zp\twoheadrightarrow {\rm H}^i_{et}(\co_{K, S}[1/p], \zp(n)), \quad \forall\, i=1,2,\quad \forall\, n\in\z_{\geq 2}.
\end{equation}
Soul\'e proved in \cite{Soule-regulators} that these morphisms have finite kernels.
Similar morphisms have been defined for $p=2$ (see \cite{Dwyer-Friedlander}.) Their kernels and cokernels are finite but non trivial, in general.

\noindent Consequently, the group ${\rm H}^2_{et}(\co_{K, S}[1/p], \zp(n))$ is finite while ${\rm H}^1_{et}(\co_{K, S}[1/p], \zp(n))$ is finitely generated over $\zp$ of $\zp$--rank equal
to ${\rm ord}_{s=(1-n)}\, \zeta_{K, S}(s)$, for all
$p$ and all $n\geq 2$. If $i\geq 3$, the group ${\rm H}^i_{et}(\co_{K, S}[1/p], \zp(n))$
vanishes for $p> 2$ and is a finite, $2$-primary group for  $p= 2$ (see \cite{Kolster}, \S2.)

With notations as in \S\ref{EMC}, we have the following.

\begin{lemma}\label{CS-lemma}For all $n\geq 2$ and all primes $p$, the following hold.
\begin{enumerate}
\item We have a $\zp[G]$--module isomorphism
$${\rm H}^1_{et}(\co_{K, S}[1/p], \zp(n))_{\rm tors}\simeq \left(\qp/\zp(n)\right)^{G_K}.$$
\item We have an equality of $\zp[G]$--ideals
$${\rm Ann}_{\zp[G]}({\rm H}^1_{et}(\co_{K, S}[1/p], \zp(n))_{\rm tors})=\langle\delta_{T, K/k}(1-n)\mid T \rangle,$$
where $T$ runs through all the finite, non-empty sets of primes in $k$ which are disjoint from $S\cap S_p$.
\item ${\rm Ann}_{\zp[G]}({\rm H}^1_{et}(\co_{K, S}[1/p], \zp(n))_{\rm tors})\cdot\Theta_{S, K/k}(1-n)\subseteq \zp[G].$
\end{enumerate}
\end{lemma}
\begin{proof} For (1), see \S2 in \cite{Kolster}. Note that in loc.cit., the author deals with the \'etale cohomology groups
of ${\rm Spec}(\co_K[1/p])$ rather than those of ${\rm Spec}(\co_{K,S}[1/p])$. However, for all $n$ and $p$ as above, Soul\'e's localization sequence in \'etale
cohomology (see \cite{Soule}) establishes an isomorphism of $\zp[G]$--modules
\begin{equation}\label{independent-h1}{\rm H}^1_{et}(\co_{K}[1/p], \zp(n))\simeq {\rm H}^1_{et}(\co_{K, S}[1/p], \zp(n)).\end{equation}

Part (2) follows from part (1) and a well--known lemma of Coates (see Lemma 2.3 in \cite{Coates} and its proof.) Also, see the proof
of Corollary \ref{Deligne-Ribet-rewrite} above.

Part (3) is a direct consequence of part (2) and Corollary \ref{Deligne-Ribet-rewrite} above.
\end{proof}

\begin{conjecture}[Coates-Sinnott, cohomological version]\label{CS-coh}
For all $(K/k, S, p\,, n)$ as above, the following holds.
\begin{eqnarray}{\bf CS}(K/k, S, p\,, n):\quad
\nonumber {\rm Ann}_{\zp[G]}({\rm H}^1_{et}(\co_{K, S}[1/p], \zp(n))_{\rm tors})\cdot\Theta_{S, K/k}(1-n)\subseteq  \\
\nonumber \subseteq {\rm Ann}_{\zp[G]}({\rm H}^2_{et}(\co_{K,
S}[1/p], \zp(n))).
\end{eqnarray}
\end{conjecture}

\noindent With notations as above, let $e_n(K/k)$ be the idempotent in $\zp[G]$ given by
$$e_n(K/k):=\left\{
              \begin{array}{ll}
                \prod_{v\in S_\infty(k)}\frac{1}{2}(1+(-1)^n\sigma_v), & \hbox{if $k$ is totally real;} \\
                0, & \hbox{otherwise,}
              \end{array}
            \right.$$
where $\sigma_v$ is the generator of the decomposition group $G_v$, for all $v\in S_\infty(k)$. The main goal of this
section is a proof of the following refinement of the conjecture above, under the expected hypotheses.

\begin{theorem}\label{cs-theorem} Assume that $(K/k, S, p\,, n)$ are as above. If $k$ is totally real, assume that $p>2$ and
$\mu_{K(\bmu_p)^{CM},p}=0$, where $K(\bmu_p)^{CM}$ is the maximal $CM$--subfield of $K(\bmu_p)$. Then, the following holds.
\begin{eqnarray}{\bf \overline{CS}}(K/k, S, p\,, n):\quad
\nonumber {\rm Ann}_{\zp[G]}({\rm H}^1_{et}(\co_{K, S}[1/p], \zp(n))_{\rm tors})\cdot\Theta_{S, K/k}(1-n)= \\
\nonumber =e_n(K/k)\cdot {\rm Fit}_{\zp[G]}({\rm H}^2_{et}(\co_{K,
S}[1/p], \zp(n))).
\end{eqnarray}
\end{theorem}

\begin{proof} We begin by making a few useful reduction steps.
\begin{lemma}\label{Sp} It suffices to prove statement $\overline{\bf CS}(K/k, S, p\,, n)$ under the assumption that $S_p\subseteq S$, where
$S_p$ denotes the set of $p$--adic primes in $k$.
\end{lemma}
\begin{proof} This is an immediate consequence of \eqref{independent-h1} and the obvious equality in $\zp[G]$
$$\Theta_{S\cup S_p, K/k}(1-n) = \Theta_{S,K/k}(1-n)\cdot \prod_{v\in S_p\setminus S}(1-\sigma_v^{-1}\cdot({\bf N}v)^{n}),$$
for all $n$ and $p$ as above. Indeed, for all $v\in S_p\setminus S$, the element $(1-\sigma_v^{-1}\cdot({\bf N}v)^{n})$ is a divisor of
$(1-p^m)$ in $\zp[G]$, for a large $m$ (say, $m$ such that $p^m=|\kappa(w)|^n$, for $w$ prime in $K$ dividing $v$) and it is therefore a unit in $\zp[G]$. Therefore,
the elements $\Theta_{S\cup S_p, K/k}(1-n)$ and $\Theta_{S,K/k}(1-n)$ differ by a unit in $\zp[G]$.
\end{proof}

\begin{lemma}\label{k-totally-real}
It suffices to prove statement $\overline{\bf CS}(K/k, S, p\,, n)$ under the assumption that $k$ is a totally real number field.
\end{lemma}
\begin{proof} Indeed, the functional equation satisfied by the $S$-incomplete $L$--function associated to a character $\chi\in\widehat G(\C)$ implies the following formula for the order of vanishing at $s=(1-n)$,
for all $n\in\z_{\geq 2}$.
\begin{equation}\label{orders-of-vanishing}
{\rm ord}_{s=(1-n)}L_S(\chi, s)=\left\{
                      \begin{array}{ll}
                        r_2(k)+a(\chi)^+, & \hbox{if $n$ is odd;} \\
                        r_2(k)+a(\chi)^-, & \hbox{if $n$ is even,}
                      \end{array}
                    \right.
\end{equation}
where $r_2(k)$ denotes the number of complex infinite places in $k$,
$$a(\chi)^+={\rm card}\{v\in S_\infty(k)\mid \chi\mid_{G_v}=\mathbf 1_{G_v}\}, \quad a(\chi)^-={\rm card}\{v\in S_\infty(k)\mid \chi\mid_{G_v}\ne\mathbf 1_{G_v}\}.$$
Note that $a(\chi)^++a(\chi)^-=r_1(k)$, the number of real
infinite places in $k$. The above formula shows that if $r_2(k)>0$ (i.e. if $k$ is not totally real), then we have $\Theta_{S, K/k}(1-n)=0$, for all $n\in\z_{\geq 2}$. This concludes the proof of the Lemma.
\end{proof}
\begin{lemma}\label{top-change} Assume that $\widetilde K/k$ is an abelian extension of number fields, of Galois group $\widetilde G$, with
$K\subseteq \widetilde K$ and $S_{\rm ram}(\widetilde K/k)\subseteq S$. Then, for all $p>2$ and $n$ as above,
$$\overline{\bf CS}(\widetilde K/k, S, p\,, n)\Rightarrow \overline{\bf CS}(K/k, S, p\,, n).$$
\end{lemma}
\begin{proof} Let us fix $p>2$ and $n$ as above. In order to simplify notations, we let ${\rm H}^i_K:={\rm H}^i_{et}(\co_{K, S}[1/p], \zp(n))$ and similarly
 for ${\rm H}^i_{\widetilde K}$, for all $i=1,2$. Let $H:=G(\widetilde K/K)$. Galois restriction induces the usual $\zp$--algebra morphism
$\pi: \zp[\widetilde G]\twoheadrightarrow \zp[G]$, whose kernel is the relative augmentation ideal $I_H$ associated to $H$ in $\zp[\widetilde G]$. Proposition 2.10 in \cite{Kolster}
gives a canonical $\zp[G]$--module isomorphism
$${\rm H}^2_K\simeq ({\rm H}^2_{\widetilde K})_H,\qquad {\rm H}^1_K\simeq ({\rm H}^1_{\widetilde K})^H$$
where $M^H:=M[I_H]$ and $M_H:=M/I_HM\simeq M\otimes_{\zp[\widetilde G]}\zp[G]$ denote the modules of $H$--invariants and $H$--coinvariants, respectively,
associated to any $\zp[\widetilde G]$--module
$M$. The isomorphism above combined with the second equality in \eqref{fitt-base-change} of the
Appendix (applied for $\rho:=\pi$, $R:=\zp[\widetilde G]$, $R':=\zp[G]$, and $M:={\rm H}^2_K$  gives
\begin{equation}\label{h2-coinvariants}\pi({\rm Fit}_{\zp[\widetilde G]}({\rm H}^2_{\widetilde K}))={\rm Fit}_{\zp[G]}({\rm H}^2_{K}).
\end{equation}
The above equality combined with the obvious equalities $\pi(e_n(\widetilde K/k))=e_n(K/k)$, $\pi(\Theta_{S,\widetilde K/k}(1-n))=\Theta_{S,K/k}(1-n)$ and
\begin{equation}\label{h1-invariants} \pi({\rm Ann}_{\zp[\widetilde G]}( {\rm H}^1_{\widetilde K})_{\rm tors})={\rm Ann}_{\zp[G]}( {\rm H}^1_{K})_{\rm tors}
\end{equation}
concludes the proof of the Lemma. (Apply Lemma \ref{CS-lemma}(2) for the last equality.)
\end{proof}
\begin{lemma}\label{K-CM} With notations as in the previous Lemma, assume that $k$ is totally real, $\widetilde K$ is totally imaginary and
$K:={\widetilde K}^{\,CM}$ is the maximal CM subfield of $\widetilde K$. Then, for all $p>2$ and $n$ as above, we have
$$\overline{\bf CS}(\widetilde K/k, S, p\,, n) \Leftrightarrow \overline{\bf CS}(K/k, S, p\,, n).$$
\end{lemma}
\begin{proof} Let us fix $p>2$ and $n$ as above. We use the notations in the proof of the previous Lemma.
Note that $H$ is the $2$--group generated by the set
$$\{\sigma_{v}\cdot \sigma_{v'}\mid v, v'\in S_\infty(k), v\ne v'\}.$$
This shows that, for all $\chi\in\widehat G(\C)$, we have
$$\chi\mid_H\ne\mathbf 1_H\quad  \Rightarrow \quad L_{S}(\chi, 1-n)=0, \quad  \chi(e_n(\widetilde K/k))=0.$$
Indeed, if $\chi\mid_H\ne\mathbf 1_H$, then there exist $v, v'\in S_\infty(k)$, with $v\ne v'$, such that
$\chi(\sigma_v)=1$ and $\chi(\sigma_{v'})\ne 1$. On one hand, this shows that $\chi(e_n(\widetilde K/k))=0$. On the other hand, it shows that $a(\chi)^+\geq 1$ and $a(\chi)^-\geq 1$. Now, \eqref{orders-of-vanishing}
implies that $L_S(\chi, 1-n)=0$, for all $n\geq 2$. Consequently,  if $e_H:=|H|^{-1}\cdot \sum_{h\in H}h$ is the idempotent element
associated to $H$ in $\zp[\widetilde G]$ (note that $p\nmid |H|$), we have
$$\Theta_{S,\widetilde K/k}(1-n)\in e_H\qp[\widetilde G], \qquad e_n(\widetilde K/k)\in e_H\zp[\widetilde G].$$
Now, note that since $\zp[\widetilde G]=e_H\zp[\widetilde G]\oplus (1-e_H)\zp[\widetilde G]$  and $I_H=(1-e_H)\zp[\widetilde G]$, the map $\pi$ establishes a ring isomorphism $\pi: e_H\zp[\widetilde G]\simeq \zp[G]$.
Consequently, \eqref{h2-coinvariants} and \eqref{h1-invariants} imply that $\pi$ establishes an isomorphism at the level of $e_H\zp[\widetilde G]$--ideals
$$\pi: e_H{\rm Fit}_{\zp[\widetilde G]}({\rm H}^2_{\widetilde K})\simeq {\rm Fit}_{\zp[G]}({\rm H}^2_{K}),\quad
\pi: e_H{\rm Ann}_{\zp[\widetilde G]}({\rm H}^1_{\widetilde K})_{\rm tors}\simeq {\rm Ann}_{\zp[G]}({\rm H}^1_{K})_{\rm tors} .$$
Now, the equalities $\pi(e_n(\widetilde K/k))=e_n(K/k)$ and  $\pi(\Theta_{S,\widetilde K/k}(1-n))=\Theta_{S,K/k}(1-n)$ conclude the proof of the Lemma.
\end{proof}

Now, we are ready to return to the proof of Theorem \ref{cs-theorem}. We fix $(K/k,S, p)$ as in the theorem. According to the previous four Lemmas, we may assume that $k$ is totally real
(otherwise, the statement is trivially true, according to Lemma \ref{k-totally-real}), $S_p\subseteq S$ (see Lemma \ref{Sp}), $K$ is CM and
$\bmu_p\subseteq K$ (otherwise, we replace $K$ by $K(\bmu_p)^{CM}$ and apply Lemmas \ref{K-CM} and \ref{top-change}, respectively.) Under these hypotheses,
we assume in addition that $p>2$ and $\mu_{K,p}=0$.

As usual, we let $\ck$ denote the cyclotomic $\zp$--extension of $K$, $\cg:=G(\ck/k)$ and $\Gamma:=G(\ck/K)$. We fix a finite, nonempty set $T$ set of primes in $K$, such that $S\cap T=\emptyset.$  Also,
we let $\cs$ and $\ct$
denote the sets of finite primes in $\ck$ sitting above primes in $S$ and $T$, respectively. For simplicity, we let $\Theta_S:=\Theta_{S, K/k}$, $\Theta_{S,T}:=\Theta_{S,T,K/k}$ and $\delta_T:=\delta_{T,K/k}$
and resume using all the notations introduced in \S\ref{EMC}. In particular, $j\in\cg$ denotes the unique complex conjugation automorphism of $\ck$ (identified as usual with the complex
conjugation automorphism $j\in G$ of $K$.) Note that we have
$$e_n:=e_n(K/k)=\frac{1}{2}(1+(-1)^n j), \qquad \forall\, n\in\z.$$
We view $e_n$ either as an element of $\zp[[\cg]]$ or $\zp[G]$, for all $n\in\z$.
Observe that if $M$ is a $\zp[[\cg]]$--module and $n\in\z$,
then $e_nM=M^+$, if $n$ is even and $e_nM=M^-$, if $n$ is odd. Also, note that as a consequence of \eqref{orders-of-vanishing} we have
\begin{equation}\label{theta-component}\Theta_{S}(1-n)=e_n\cdot \Theta_{S}(1-n), \quad
\Theta_{S,T}(1-n)=e_n\cdot \Theta_{S,T}(1-n),\qquad \forall\, n\in\z_{\geq 2}.
\end{equation}

In what follows, all the occurring $\zp$--module duals or Pontrjagin duals are endowed with the {\bf contravariant}
actions by the appropriate groups, unless stated otherwise.
We will need the following elementary result.
\begin{lemma}\label{elementary-lemma} Let $M$ be a $\zp[[\cg]]$--module. Assume that $M$ is $\zp$--free of finite rank and that $M_{\Gamma}$ is finite. Then,
\begin{enumerate}\item $M^\Gamma=0$.
\item $(M_{\Gamma})^\vee\simeq (M^\ast)_{\Gamma},$ as $\zp[G]$--modules.
\item If ${\rm pd}_{\zp[[\cg]]}M\leq 1$, then ${\rm pd}_{\zp[G]}M_{\Gamma}\leq 1$.
\end{enumerate}
\end{lemma}
\begin{proof} Part (1) is immediate. For part (2), see Lemma 5.18 in \cite{GP}. Part (3) is an immediate consequence of
$\zp[[\cg]]^{\,\Gamma}=0$ and $\zp[[\cg]]_{\,\Gamma}\simeq\zp[G]$. \end{proof}
\begin{proposition}\label{dual-coinvariants}
Let $n\in\z_{\geq 2}$, $p>2$, and ${\rm H}^i_{K}:={\rm H}_{et}^i(\co_{K,S}, \zp(n))$, for all $i=1,2$. Then, we have $\zp[G]$--module isomorphisms
\begin{enumerate}
\item $e_n\cdot {\rm H}^2_K\simeq (\cxs^+(-n)_{\Gamma})^\vee\simeq (\cxs^+(-n)^\ast)_{\Gamma};$\medskip

\item $e_n\cdot {\rm H}^1_K=e_n\cdot ({\rm H}^1_K)_{\rm tors}\simeq({\rm H}^1_K)_{\rm tors}\simeq\zp(n)_\Gamma,$
\end{enumerate}
where $M^\Gamma$ and $M_{\Gamma}$ denote the $\zp[G]$--modules of $\Gamma$--invariants and $\Gamma$--coinvariants,
respectively, for all $\zp[[\cg]]$--modules $M$.
 \end{proposition}
\begin{proof} Let  ${\rm H}^i_{K^+}:={\rm H}_{et}^i(\co_{K^+,S}, \zp(n))$, for all $i=1,2$, where $K^+:=K^{j=1}$, as usual. Then, Proposition 2.9 in \cite{Kolster} shows that there is a canonical $\zp[G]$--module
isomorphism ${\rm H}^1_{K^+}\simeq ({\rm H}^1_K)^+$. Consequently, \eqref{chern} and \eqref{Borel} together with \eqref{orders-of-vanishing} give
$${\rm rank}_{\zp}\, (e_n\cdot {\rm H}^1_K)=\left\{
                         \begin{array}{ll}
                          r_2(K)-r_1(K^+) , & \hbox{if $n$ is odd;}\\
                            r_2(K^+) , & \hbox{if $n$ is even}
                         \end{array}
                       \right\} = 0.
$$ Consequently, $e_n{\rm H}^1_K$ is a finite group. On one hand, this implies the equality in part (2) of the Proposition. On the other hand, via the spectral sequence argument on pp. 237--238 of \cite{Kolster}
this leads to isomorphisms of $\zp[G]$--modules
$$e_n\cdot {\rm H}^2_{K}\simeq e_n\cdot{\rm H}^1_{et}(\co_{K,S}, \qp/\zp(n))\simeq e_n\cdot {\rm H}^1(G_{\ck, \cs}^{(p)}, \qp/\zp(n))^{\Gamma},$$
where $G_{\ck, \cs}^{(p)}$ is the Galois group of the maximal pro--$p$ extension of $\ck$ which is unramified away from $\cs$ and the right--most cohomology group
is a Galois cohomology group. Now, since $\bmu_{p^\infty}\subseteq\ck$ and therefore $G_{\ck, \cs}^{(p)}$ acts trivially on $\qp/\zp(n)$ and since
$\cxs$ is the maximal abelian quotient of $G_{\ck, \cs}^{(p)}$, we have the following isomorphisms of $\zp[G]$--modules.
\begin{eqnarray}
  \nonumber   {\rm H}^1(G_{\ck, \cs}^{(p)}, \qp/\zp(n))^{\Gamma}\simeq & {\rm Hom}_{\zp}(G_{\ck, \cs}^{(p)}, \qp/\zp(n))^{\Gamma} & \\
   \nonumber                                                    \simeq & {\rm Hom}_{\zp}(\cxs, \qp/\zp(n))^{\Gamma} &\simeq\quad (\cxs(-n)_{\Gamma})^\vee.
  \end{eqnarray}
Now,  the first isomorphism in part (1) of the Proposition follows from the last two displayed isomorphisms and the obvious equality $e_n\cdot\cxs(-n)=\cxs^+(-n)$.

The second isomorphism in part (1) follows from Lemma \ref{elementary-lemma}(2) applied to $M:=\cxs^+(-n).$
Note that the finiteness of ${\rm H}^2_K$ and the first isomorphism in part (1) of the Proposition imply that $\cxs^+(-n)_{\Gamma}$ is finite. Also, $\cxs^+(-n)$ is
finitely generated over $\zp$ (a theorem of Iwasawa) and therefore $\zp$--free of finite rank (a consequence of $\mu_{K,p}=0$.)
This concludes the proof of part (1) of the Proposition.

The isomorphism in part (2) of the Proposition is an immediate consequence of the isomorphism $({\rm H}^1_K)_{\rm tors}\simeq(\qp/\zp(n))^\Gamma$ (see Lemma \ref{CS-lemma}(1) and recall that $\bmu_{p^\infty}\subseteq\ck$)
and the obvious equality $\qp(n)^\Gamma=\qp(n)_{\Gamma}=0$.
\end{proof}

 Note that the hypotheses of Lemma \ref{link-classical} and Theorem \ref{emc} are satisfied by the data $(\ck/k, \cs, \ct, p)$.
Lemma \ref{link-classical} and exact sequence \eqref{sequence-empty-T} tensored with $\zp(n-1)$ lead to the following four term exact sequence of $\zp[[\cg]]$--modules.
\begin{equation}\label{four-term-sequence} 0\to \zp(n)\to T_p(\Delta_{\ck, \ct})^-(n-1)
\to T_p(\mk)^-(n-1)\to \cxs^+(-n)^\ast
\to 0.\end{equation}
We intend to apply Proposition \ref{four-term-sequence-fitting} in the Appendix to the sequence of $\Gamma$--coinvariants associated to the exact sequence above.

First, let us note that each of the four modules in the exact sequence above is $\zp$--free. Most importantly, the module of $\Gamma$--coinvariants associated to each of these is finite.
Indeed, $\zp(n)_{\Gamma}$ and $(\cxs^+(-n)^\ast)_{\Gamma}$ are finite, due to Proposition \ref{dual-coinvariants}. On the other hand,
Lemma \ref{delta}(1) and Remark \ref{remark-delta} combined with Lemma \ref{twisting-Fitting}(1) in the Appendix imply that we have a $\zp[G]$--module isomorphism
$$T_p(\Delta_{\ck, \ct})(n-1)_{\Gamma}\simeq\bigoplus_{v\in T}\zp[G]/(1-\sigma_v^{-1}\cdot N_v^{n}).$$
Since $n\geq 2$, the element $(1-\sigma_v^{-1}\cdot N_v^{n})$ is not a zero--divisor in $\zp[G]$, for all $v\in T$. Consequently, the above isomorphism
implies that  $T_p(\Delta_{\ck, \ct})(n-1)_{\Gamma}$ is finite. Therefore, its direct summand  $T_p(\Delta_{\ck, \ct})^-(n-1)_{\Gamma}$ is indeed finite. Now, the finiteness
of $T_p(\mk)^-(n-1)_{\Gamma}$ follows from the exact sequence
$$T_p(\Delta_{\ck, \ct})^-(n-1)_{\Gamma}
\to T_p(\mk)^-(n-1)_{\Gamma}\to (\cxs^+(-n)^\ast)_{\Gamma}
\to 0.$$
Now, Lemma \ref{elementary-lemma} implies that the $\Gamma$--invariants of the four modules in \eqref{four-term-sequence} are trivial. Consequently, we obtain the following exact sequence of finite $\zp[G]$--modules
at the level of $\Gamma$--coinvariants
\begin{equation}\label{four-term-sequence-coninvariants} 0\to e_n\cdot{\rm H}^1_K\to T_p(\Delta_{\ck, \ct})^-(n-1)_{\Gamma}
\to T_p(\mk)^-(n-1)_{\Gamma}\to e_n\cdot {\rm H}^2_K
\to 0.\end{equation}
Above, we have used Proposition \ref{dual-coinvariants} to identify the $\Gamma$--coinvariants of the end terms of \eqref{four-term-sequence} with $e_n\cdot{\rm H}^1_K$ and
$e_n\cdot{\rm H}^2_K$, respectively.

Since ${\rm pd}_{\zp[[\cg]]}T_p(\Delta_{\ck, T})^-\leq 1$ and ${\rm pd}_{\zp[[\cg]]} T_p(\mk)^-\leq 1$ (see Lemma \ref{delta}(3) and Remark \ref{remark-delta} for the first and Theorem \ref{projective}(2) for the second),
we have
$${\rm pd}_{\zp[G]}\,(T_p(\Delta_{\ck, \ct})^-(n-1)_{\Gamma})\leq 1, \qquad {\rm pd}_{\zp[G]}\, (T_p(\mk)^-(n-1)_{\Gamma})\leq 1,$$
as a consequence of \ref{twisting-Fitting}(3) and Lemma \ref{elementary-lemma}(3).
Consequently, we may apply Proposition \ref{four-term-sequence-fitting} in the Appendix to the
exact sequence  \eqref{four-term-sequence-coninvariants}.
This way, we obtain
\begin{eqnarray}\label{product-fitting}
 & {\rm Fit}_{\zp[G]}({e_n{\rm H}^1_K}^\vee)\cdot {\rm Fit}_{\zp[G]}(T_p(\mk)^-(n-1)_{\Gamma})=\\
  \nonumber & ={\rm Fit}_{\zp[G]}(e_n{\rm H}^2_K)\cdot {\rm Fit}_{\zp[G]}(T_p(\Delta_{\ck, \ct})^-(n-1)_{\Gamma}).
\end{eqnarray}
where the dual is endowed with the covariant $G$--action. Now, let us note that for any $\zp[[\cg]]$--module $M$ and $\zp[G]$--module $N$,  we have
$$e_n\cdot M(n-1)=M^-(n-1), \qquad e_n\cdot{\rm Fit}_{\zp[G]}(N)={\rm Fit}_{e_n\zp[G]}(e_n\cdot N).$$
Consequently, if we combine Lemma \ref{twisting-Fitting}(2) in the Appendix with Corollary \ref{full-group-ring}
and Lemma \ref{delta}(2), respectively, we obtain
\begin{eqnarray}
\nonumber & e_n\cdot{\rm Fit}_{\zp[G]}(T_p(\mk)^-(n-1)_{\Gamma}) = e_n\cdot (\pi\circ t_{1-n}(\Theta_{S,T}^{(\infty)}))=(\Theta_{S,T}(1-n)),\\
\nonumber & e_n\cdot{\rm Fit}_{\zp[G]}(T_p(\Delta_{\ck, \ct})^-(n-1)_{\Gamma})= e_n\cdot (\pi\circ t_{1-n}(\delta_{T}^{(\infty)}))=e_n\cdot(\delta_T(1-n)),
\end{eqnarray}
where $\pi:\zp[[\cg]]\to\zp[G]$ is the usual projection. Note that above we have used the second equality in \eqref{theta-component}. Consequently, \eqref{product-fitting} implies that
$${\rm Fit}_{\zp[G]}({e_n{\rm H}^1_K}^\vee)\cdot\Theta_{S,T}(1-n)=e_n{\rm Fit}_{\zp[G]}({\rm H}^2_K)\cdot(\delta_T(1-n)).$$
However, since $\Theta_{S,T}(1-n)=\delta_T(1-n)\cdot\Theta_S(1-n)$ and $\delta_T(1-n)$ is not a zero--divisor in $\zp[G]$ (an easy exercise !), the last equality implies
$${\rm Fit}_{\zp[G]}(({\rm H}^1_K)_{\rm tors}^{\,\,\vee})\cdot\Theta_S(1-n)=e_n{\rm Fit}_{\zp[G]}({\rm H}^2_K).$$
Now, since $({\rm H}^1_K)_{\rm tors}$ is a cyclic module (see Lemma \ref{elementary-lemma}(2)) and the dual is endowed with the covariant
$G$--action, we have equalities
$${\rm Fit}_{\zp[G]}(({\rm H}^1_K)_{\rm tors}^{\,\,\vee})={\rm Ann}_{\zp[G]}(({\rm H}^1_K)_{\rm tors}^{\,\,\vee})={\rm Ann}_{\zp[G]}(({\rm H}^1_K)_{\rm tors}).$$
When combined with the last displayed equality, this concludes the proof of Theorem \ref{cs-theorem} (our refinement of the cohomological Coates-Sinnott Conjecture.)
\end{proof}

\begin{remark} Results somewhat weaker than our Theorem \ref{cs-theorem} were obtained with different methods in \cite{Burns-Greither}, Corollary 2 (which imposes restrictions upon $K/k$ and $p$) and
\cite{NQD}, Th\'eor\`eme 4.3 (which imposes restrictions upon $K/k$ and $n$.)
 In both cases, the vanishing of the appropriate Iwasawa $\mu$--invariant is assumed.
\end{remark}

Finally, we would like to mention that the well known Quillen-Lichtenbaum Conjecture states that the Chern character maps ${\rm ch}^i_{p,n}$ (see \eqref{chern} above) are isomorphisms,
for all $p>2$ and $n\geq 2$. On the other hand, it is known that the Quillen-Lichtenbaum Conjecture is a consequence of the Bloch-Kato Conjecture for finitely generated fields (e.g., see Theorem 2.7 in \cite{Kolster}.)
To our knowledge, recent work of Rost and Voevodsky has lead to a proof of the Bloch-Kato Conjecture.

\noindent The following is an immediate consequence of Theorem \ref{cs-theorem}.

\begin{corollary}[a refined ${\rm K}$--theoretic Coates--Sinnott Conjecture] \label{cs-corollary} Let $K/k$ be an abelian extension of number fields of
Galois group $G$. Let $S$ be a finite set of primes in $k$, such that $S_\infty(k)\cup S_{\rm ram}(K/k)\subseteq S$. If $k$ is totally real, assume that
$\mu_{K(\bmu_p)^{CM},p}=0$, for all primes $p>2$. Also, assume that the Quillen-Lichtenbaum Conjecture holds. Then, for all $n\in\z_{\geq 2}$, we have the following equality
of $\z[1/2][G]$--ideals.
\begin{eqnarray}
\nonumber \z[1/2]\,{\rm Ann}_{\z[G]}({\rm K}_{2n-1}(\co_{K,S})_{\rm tors})\cdot\Theta_S(1-n)=\\
\nonumber =e_n(K/k)\cdot \z[1/2]\,{\rm Fit}_{\z[G]}({\rm K}_{2n-2}(\co_{K,S})).
\end{eqnarray}
\end{corollary}
\noindent As the reader will notice right away, the ${\rm K}$--theoretic statement above is closer in spirit to the conjecture originally formulated by Coates and Sinnott in \cite{Coates-Sinnott}.
The difference is the presence of a Fitting ideal rather than an annihilator on the right-hand side and an equality rather than an inclusion of ideals. These differences justify the use of the term ``refined'' above.
\section{Appendix: Algebraic Ingredients}
\subsection{Determinants and Fitting Ideals}\label{appendix-fitting}
Let $R$ be a commutative ring with $1$, $P$ a finitely generated, projective $R$--module and $f\in{\rm End}_R(P)$. Then,
the determinant $\det_R(f\mid P)$ of $f$ acting on $P$ is defined as follows. We take a finitely generated $R$--module $Q$,
such that $P\oplus Q$ is a (finitely generated) free $R$--module, then we let $f\oplus\mathbf 1_Q\in{\rm End}_R(P\oplus Q)$, where
$\mathbf 1_Q$ is the identity of $Q$, and define
$${\rm det}_R(f\mid P):={\rm det}_R(f\oplus \mathbf 1_Q\mid P\oplus Q)\,.$$
It is easy to check (use Schanuel's Lemma !) that the definition above does not depend on $Q$. Now, one can use the same strategy to define
the characteristic polynomial ${\rm det}_R(X-f\mid P)\in R[X]$ of variable $X$. Indeed, $P\otimes_R R[X]$ is a finitely
generated, projective $R[X]$--module. One defines
$${\rm det}_R(X-f\mid P):={\rm det}_{R[X]}({\rm id}_P\otimes X-f\otimes 1\mid P\otimes_R R[X])\,.$$
For any $P$, $R$ and $f$ as above and any $R$--algebra $R'$, we have base-change equalities
\begin{equation}\label{det-base-change}
\begin{array}{c}
  {\rm det}_R(f\mid P)={\rm det}_{R'}(f\otimes\mathbf 1_{R'}\mid P\otimes_R R')\,, \\
{}\\
  {\rm det}_R(X-f\mid P)={\rm det}_{R'}(X- (f\otimes\mathbf 1_{R'})\mid P\otimes_R R')\,.
\end{array}
\end{equation}
\begin{remark}\label{monicity-remark} It is easily checked that the polynomial $F(X):={\rm det}_R(X-f\mid P)$ in $R[X]$ defined above is always a monic polynomial.
%Indeed, this is clear if $P$ is a free
%$R$--module. In general, one uses the second equality in \eqref{det-base-change} above for $R':=R_{\mathfrak p}$, the localization of $R$
%at $\mathfrak p$, for all $\mathfrak p\in{\rm Spec}(R)$ to prove that the leading coefficient of $F(X)$ maps to $1$ in $R_{\mathfrak p}$ via the localization map
%$R\to R_{\mathfrak p}$, for all $\mathfrak p$. this forces the leading coefficient in question to equal $1$ in $R$.
\end{remark}

Now, for any $R$ as above and any finitely presented $R$--module $M$, the first Fitting invariant (ideal) ${\rm Fit}_R(M)$ of $M$ over $R$ is defined as follows.
First, one considers a finite presentation of $M$
$$\xymatrix{
R^n\ar[r]^\phi &R^m\ar[r] &M\ar[r] &0}.$$ By definition, the
Fitting ideal ${\rm Fit}_R(M)$ is the ideal in $R$ generated by the
determinants of all the $m\times m$ minors of the matrix $A_\phi$
associated to $\phi$ with respect to $R$--bases of $R^n$ and
$R^m$. It is well-known that the definition does not depend on the
chosen presentation or bases, and
\begin{equation}\label{ann-fitt}
{\rm Ann}_R(M)^m\subseteq{\rm Fit}_R(M)\subseteq {\rm Ann}_R(M)\,.
\end{equation}
Two elementary properties of Fitting ideals which are used throughout the paper state that
if $M\twoheadrightarrow M'$ is a surjective morphism of finitely presented $R$--modules and $\rho:R\to R'$ is a morphism of commutative rings with $1$,
then one has
\begin{equation}\label{fitt-base-change}{\rm Fit}_R(M)\subseteq {\rm Fit}_R(M'),\qquad {\rm Fit}_{R'}(M\otimes_R R')=\rho({\rm Fit}_R(M))R'.\end{equation}
For more details on general properties of Fitting ideals, the reader can consult the Appendix of \cite{Mazur-Wiles}.

In what follows, if $R$ is a commutative topological ring and $\Gamma$ is
a profinite group, then the profinite group algebra
$$R[[\Gamma]]:=\underset{\mathfrak H}{\underset\longleftarrow\lim} R[\Gamma/\mathfrak H],$$
where $\Gamma/\mathfrak  H$ are all the finite quotients of
$\Gamma$ by (open and) closed subgroups $\mathfrak H$, is viewed as a
topological $R$--algebra endowed with the usual projective limit
topology.

Below, by a semi-local ring $R$ we mean a direct sum of finitely many local rings. Examples of such rings include $\co[G]$ and $\co[G]^\pm:=\co[G]/(1\mp j)$, where $G$ is a finite, abelian group,
$\co$ is a finite integral extension of $\zp$, for some odd prime $p$ and $j$ is an element of order $2$ in $G$.
\begin{proposition}\label{fitting-calculation}
Let $R$ be a commutative, semi-local, compact topological ring and $\Gamma$ a pro-cyclic group of topological generator $g$. Let $M$ be a topological $R[[\Gamma]]$--module,
which is projective and finitely generated as an $R$--module. Let $$F(X):={\rm det}_{R}(X-\mathfrak m_g\mid M),$$ where $\mathfrak m_g$ is the $R[[\Gamma]]$--module automorphism of $M$ given
by multiplication by $g$. Then, the following hold.
\begin{enumerate}\item $M$ is finitely presented as an $R[[\Gamma]]$--module. Also, if we let $F(g)$ be the image of $F(X)$ via the $R$--algebra morphism $R[X]\to R[[\Gamma]]$ sending $X$ to $g$, we have an equality of $R[[\Gamma]]$-ideals
$${\rm Fit}_{R[[\Gamma]]}(M)=(F(g))\,.$$
\item
Let $M^\ast_R:={\rm Hom}_R(M, R)$, viewed as a
topological $R[[\Gamma]]$--module with the covariant
$\Gamma$--action, given by $\sigma\cdot
f(x):=f(\sigma\cdot x)$, for all $f\in M^\ast_R$,
$\sigma\in\Gamma$ and $x\in M$. Then, we have
$${\rm Fit}_{R[[\Gamma]]}(M)={\rm Fit}_{R[[\Gamma]]}(M^\ast_R)\,.$$
\item Assume that $R=\zp[G]$, where $G$ is
    a finite, abelian group and $p$ is a prime number. Let $M^\ast:={\rm Hom}_{\zp}(M, \zp)$, viewed as an
    $R[[\Gamma]]\simeq\zp[[G\times\Gamma]]$--module with the covariant $G\times\Gamma$--action. Then, we have
    $${\rm Fit}_{R[[\Gamma]]}(M^\ast)={\rm Fit}_{R[[\Gamma]]}(M).$$
\end{enumerate}
\end{proposition}
\begin{proof} See Proposition 4.1 and Corollary 4.2 in \cite{GP}.\end{proof}

\begin{proposition}\label{four-term-sequence-fitting}
Let $R:=\zp[G]$, for some finite abelian group $G$ and prime number $p$. Assume that we have an exact sequence
of finite $R$--modules
$$\xymatrix{0\ar[r] &A\ar[r] &P\ar[r] &P'\ar[r] &A'\ar[r] &0.
}$$ Further, assume that ${\rm pd}_{\zp[G]}P\leq 1$ and ${\rm pd}_{\zp[G]}P'\leq 1$. Then, we have
$${\rm Fit}_R(A^\vee)\cdot{\rm Fit}_R(P')={\rm Fit}_R(A)\cdot{\rm Fit}_R(P),$$
where the dual $A^\vee:={\rm Hom}(A, \qp/\zp)$ is endowed with the covariant $G$--action.
\end{proposition}
\begin{proof}See Lemma 5, p.179 of \cite{Burns-Greither}. In loc.cit. this result is proved for
general finitely generated $\zp$--algebras $R$ which are $\zp$--free and relatively Gorenstein over $\zp$. Also, note
that ${\rm pd}_{\zp[G]}P=0$ if and only if $P=0$ (and similarly for $P'$), in which case the equality above is immediate.\end{proof}
\subsection{\bf Twisting}\label{appendix-twisting} In what follows, we fix an odd prime $p$, a field $k$ of characteristic different from $p$ and a Galois extension
$\ck/k$. We let $\cg:={\rm Gal}(\ck/k)$ and assume that the group of $p$--power roots of unity $\bmu_{p^\infty}$ is contained
in $\ck$. As usual, we denote by $c_p:\cg\to\zp^\times={\rm Aut}(\bmu_{p^\infty})$ the $p$--cyclotomic character
of $\cg$ and decompose $c_p:=\omega_p\cdot\kappa_p$ in its tame (Teichm\"uller) and wild components, $\omega_p:\cg\to\bmu_{p-1}$ and $\kappa_p:\cg\to(1+p\zp)$, respectively.
For simplicity, we let $c:=c_p$, $\omega:=\omega_p$ and $\kappa:=\kappa_p$.

Let $\co$ be a finite, integral extension of $\zp$. We let $Q(\co)$ denote its field of fractions. We consider the unique continuous isomorphisms
of  $\co$--algebras
$${t_n}: \co[[\cg]]\overset\sim\longrightarrow \co[[\cg]],\qquad \iota: \co[[\cg]]\overset\sim\longrightarrow \co[[\cg]]^{\rm op}$$
satisfying $t_n(g)=c_p(g)^n\cdot g$ and $\iota(g)=g^{-1}$, for all $g\in \cg$ and all $n\in\z$.
For any $\co[[\cg]]$--module $M$ and any $n\in\z$, we let $M(n)$ denote the usual Tate twist of $M$ by $c_p^n$. More precisely, $M(n):=M$ with a new $\co[[\cg]]$--action given by
$\lambda\ast x:=t_n(\lambda)\cdot x$, for all $\lambda\in \co[[\cg]]$ and $x\in M$. Also, we let $M^\ast:={\rm Hom}_\co(M, \co)$ and view it as an $\co[[\cg]]$--module with the contra-variant $\cg$--action
given by $g\cdot f(x)=f(\iota(g)\cdot x)$, for all $f\in M^\ast$, $x\in M$ and $g\in\cg$. Throughout, $\co$ and $Q(\co)$ are viewed as a $\co[[\cg]]$--modules
with the trivial $\cg$--action. Also, if $M$ and $N$ are $\co[[\cg]]$--modules, the $M\otimes_\co N$ is viewed as an $\co[[\cg]]$--module with the diagonal $\cg$--action.

\begin{lemma}\label{twisting-Fitting}
Assume that $\cg$ is abelian and $M$ is a finitely presented $\co[[\cg]]$--module. Then, for all $n\in\z$, the following hold.
\begin{enumerate}
\item ${\rm Ann}_{\co[[\cg]]}(M(n))=t_{-n}\left({\rm Ann}_{\co[[\cg]]}(M) \right)$;
\item ${\rm Fit}_{\co[[\cg]]}(M(n))=t_{-n}\left({\rm Fit}_{\co[[\cg]]}(M) \right)$;
\item  ${\rm pd}_{\co[[\cg]]}M={\rm pd}_{\co[[\cg]]} M(n)$;
\item If $k:=\Bbb F_q$ is the finite field of $q$ elements ($p\nmid q$), $\sigma_q\in\cg$ is the $q$--power Frobenius
automorphism of $\ck$, and $M:=T_p(\ck^\times)$, then
$${\rm Fit}_{\zp[[\cg]]}(M(n))={\rm Ann}_{\zp[[\cg]]}(M(n))=(1-q^n\cdot\sigma_q^{-1}).$$
\end{enumerate}
\end{lemma}
\begin{proof} For the easy proof of (1) and (2), see Lemma 3.1 in \cite{Popescu-CS}. Part (3) is an immediate consequence of the existence of a unique
$\zp[[\cg]]$--module isomorphism $\zp[[\cg]]\simeq\zp[[\cg]](n)$ which sends $g\to c(g)^ng$, for all $g\in\cg$. Part (4) is derived from part (2) as follows.
Under the hypotheses of part (4), $\cg$ is a pro--cyclic group of (topological) generator $\sigma_q$. Since $\zp$ is a cyclic
$\zp[[\cg]]$--module, (\ref{ann-fitt}) gives
$$\quad {\rm Fit}_{\zp[[\cg]]}(\zp)={\rm Ann}_{\zp[[\cg]]}(\zp)=(1-\sigma_q^{-1})\,.$$
Now, note that $T_p(\ck^\times)=T_p(\bmu_{p^\infty})=\zp(1)$ and $c(\sigma_q)=q$, and apply (1) with $M:=\zp$.
\end{proof}

Throughout the rest of this subsection, we will assume that $\ck$ is a finite, abelian extension of Galois group $G$ of the cyclotomic
$\zp$--extension $\ck'$ of $k$ and that $\cg$ is abelian. Consequently, there is a non-canonical group isomorphism $\cg\simeq G\times\Gamma$,
where $\Gamma\simeq\zp$. We fix a topological generator $\gamma$ of $\Gamma$. We identify $\co[[\cg]]$, $\co[G][[\Gamma]]$ and $\co[G][[t]]$ via the obvious $\co[G]$--algebra isomorphisms
$$\co[[\cg]]\simeq \co[G][[\Gamma]]\simeq \co[G][[t]], \qquad \gamma\to (t+1).$$
We let $K:=\ck^\Gamma$ and identify $G$ and $\Gamma$ with $G(K/k)$ and $G(\ck'/k)$ via the usual Galois--restriction isomorphisms.
Note that since $\bmu_{p^\infty}\subseteq\ck$, we have $\bmu_p\subseteq K$. Consequently, $\omega$ and $\kappa$ factor through $G$ and $\Gamma$, respectively.
For simplicity, we assume that $\co$ contains the values of all the irreducible $\cp$--valued characters of $G$. We denote by $\widehat G(\cp)$ the set
of all $\cp$--valued irreducible characters of $G$ and, for all $\chi\in\widehat G(\cp)$, we let $e_\chi:=1/|G|\sum_{\sigma\in G}\chi(\sigma)\cdot\sigma^{-1}$ denote the idempotent associated to $\chi$ in $Q(\co)[G]$.
Each $\chi$ as above will be extended to the unique $Q(\co)[X]$-algebra and $Q(\co)\otimes_\co\co[[\Gamma]]$--algebra morphisms
$$\chi: Q(\co)[G][X]\to Q(\co)[X], \qquad \chi: Q(\co)\otimes_\co\co[[\cg]]\to Q(\co)\otimes_\co\co[[\Gamma]]$$
which send $g\to \chi(g)$, for all $g\in G$, respectively. Also, for a polynomial $P\in Q(\co)[G][X]$, we denote by $P(\gamma)$ (respectively $P(t+1)$) its image
via the unique $Q(\co)[G]$--algebra morphism $Q(\co)[G][X]\to Q(\co)\otimes_\co\co[[\cg]]$ which sends $X\to\gamma$ (respectively $X\to (t+1)$.)

Next, we let $\cl$ denote an $\co[[\cg]]$--module, which is free of finite rank as an $\co$--module. We consider the following $Q(\co)$--vector spaces
$$V:=Q(\co)\otimes_\co\cl, \qquad V^\ast:={\rm Hom}_{Q(\co)}(V, Q(\co))\simeq Q(\co)\otimes_\co\cl^\ast,$$
endowed with the usual $Q(\co)\otimes_\co\co[[\cg]]$--module structures.
We consider the following polynomials in $Q(\co)[G][X]$ and $Q(\co)[X]$, respectively:
$$P_V(X):={\rm det}_{Q(\co)[G]}(X-\mathfrak m_\gamma\mid V), \quad P_{V, \chi}(X):=\chi(P_V(X))={\rm det}_{Q(\co)}(X-\mathfrak m_\gamma\mid e_\chi V),$$
where $\mathfrak m_\gamma$ denotes the automorphism of $V$ and $e_\chi V$ given by multiplication with $\gamma$, for all $\chi\in\widehat G(\cp)$.
 Observe that, for all $\chi\in\widehat G(\cp)$ and all $n\in\z$, we have
 \begin{equation}\label{twist-poly}
 \chi(t_n(P_{V}(\gamma)))=P_{V, \chi\omega^n}(\kappa(\gamma)^n\gamma), \quad \chi((\iota\circ t_n)(P_{V}(\gamma)))=P_{V, \chi^{-1}\omega^n}(\kappa(\gamma)^n\gamma^{-1}).\end{equation}
\begin{lemma}\label{twist-poly-lemma} The following hold for all $n\in\z$ and all $\chi\in\widehat G(\cp)$.
\begin{enumerate}
\item If $\cl$ is a projective $\co[G]$--module, then $P_{V(n)}(X),\, P_{V^\ast(n)}(X)$ are monic polynomials in $\co[G][X]$.
\item $P_{V(n), \chi}(X),\, P_{V*(n), \chi}$ are monic polynomials in $\co[X]$.
\item $ P_{V(n), \chi}(\gamma)\sim \chi(t_{-n}(P_{V}(\gamma)))$ and $P_{V^\ast(n), \chi}(\gamma)\sim \chi((\iota\circ t_n)(P_{V}(\gamma)))$, where ``$\sim$'' denotes association in divisibility in the ring $\co[[\Gamma]]$.
\end{enumerate}
\end{lemma}
\begin{proof} Fix an $n\in\z$ and note that we have $Q(\co)\otimes_\co\co[[\cg]]$--module isomorphisms
\begin{equation}\label{isos}
V(n)\simeq Q(\co)[G]\otimes_{\co[G]}\cl(n), \qquad V^\ast(n)\simeq Q(\co)[G]\otimes_{\co[G]}\cl^\ast(n).
\end{equation}
In order to prove (1), first note that if $\cl$ is $\co[G]$--projective, then the $\co[G]$--modules $\cl(n)\simeq\cl\otimes_{\co}\co(n)$ and
$\cl^\ast(n)\simeq\cl(-n)^\ast$ are projective as well. Indeed, since the modules in question are $\co$--free and $\co$ is a PID, their $\co[G]$--projectivity is equivalent to their $G$--cohomological triviality
(See \cite{Serre-Local}, Ch. IX, \S5, Theorem 7 for $\z[G]$--modules. The same proof works for $\co[G]$--modules, for a general PID $\co$.)
Now, apply the Corollary to Proposition 1 in \cite{Serre-Local}, Ch. IX, \S3 to arrive at the desired result.
(Loc.cit. deals with the case of $\z[G]$--modules. The same argument works for $R[G]$--modules, where $R$ is a PID, in particular $R=\co$.) Now, part (1) follows by applying \eqref{det-base-change} to conclude that
$$P_{V(n)}(X)={\rm det}_{\co[G]}(X-\mathfrak m_\gamma\mid \cl(n)), \qquad P_{V^\ast(n)}(X)={\rm det}_{\co[G]}(X-\mathfrak m_\gamma\mid \cl^\ast(n))\,.$$
The monicity follows from Remark \ref{monicity-remark}.

Part (2) follows similarly: First, one uses \eqref{isos} to conclude that there are isomorphisms of $Q(\co)\otimes_\co[[\Gamma]]$--modules
$e_\chi \cdot V(n)\simeq Q(\co)\otimes_\co e_\chi\cdot\cl(n)$ and $e_\chi\cdot V^\ast(n)\simeq Q(\co)\otimes_\co e_\chi\cdot\cl^\ast(n)$, where $e_\chi\cdot\cl(n)$ and $e_\chi\cdot\cl^\ast(n)$
are viewed as (necessarily free) $\co$--submodules of maximal rank in $e_\chi\cdot V(n)$ and $e_\chi\cdot V^\ast(n)$, respectively. Then, \eqref{det-base-change} gives the equalities
$$P_{V(n), \chi}(X)={\rm det}_{\co}(X-\mathfrak m_\gamma\mid e_\chi\cdot\cl(n)), \qquad P_{V^\ast(n), \chi}(X)={\rm det}_{\co}(X-\mathfrak m_\gamma\mid e_\chi\cdot\cl^\ast(n)),$$
which conclude the proof of part (2).

Next, we prove the second ``$\sim$'' in part (3). The first ``$\sim$'' is proved similarly. For every $\chi\in\widehat G(\cp)$, we fix an $\co$--basis $\mathbf x_\chi$ of $e_\chi\cl$. Note that $\mathbf x_\chi$ is also a $Q(\co)$--basis of $e_\chi V$. Let $A_{\gamma, \chi}$ be the matrix of
$\mathfrak m_\gamma$ restricted to $e_\chi V$ with respect to this basis. Then,  $A_{\gamma, \chi}\in {\rm GL}_{m_\chi}(\co)$, where $m_\chi={\rm rk}_\co e_\chi\cl={\rm dim}_{Q(\co)}e_\chi V$. It is easily proved that for all $n$ and $\chi$ as above, the matrix of $\mathfrak m_\gamma$ restricted to $e_\chi\cdot V^\ast(n)=(e_{\chi^{-1}\omega^n} V)^\ast$ with respect to the basis $\mathbf x_{\chi^{-1}\omega^n}^\ast$ (dual basis of $\mathbf x_{\chi^{-1}\omega^n}$) is $\kappa(\gamma)^n\cdot(A_{\gamma, \chi^{-1}\omega^n}^{-1})^t$,
where ${}^t$ stands for transposition. Consequently, \eqref{twist-poly} combined with the fact that $\det(A_{\gamma, \chi^{-1}\omega^n})\in \co^\times$ and $\gamma\in \co[[\Gamma]]^\times$ imply that the following hold in $\co[[\Gamma]]$:
\begin{eqnarray}
  \nonumber P_{V^\ast(n), \chi}(\gamma) &=&\det(\gamma\cdot I_{m_{\chi^{-1}\omega^n}}-\kappa(\gamma)^n\cdot(A_{\gamma, \chi^{-1}\omega^n}^{-1})^t) \\
  \nonumber &\sim & \det(\kappa(\gamma)^n\gamma^{-1}\cdot I_{m_{\chi^{-1}\omega^n}}-A_{\gamma, \chi^{-1}\omega^n})\\
  \nonumber &=& \chi((\iota\circ t_n)(P_{V}(\gamma)))
\end{eqnarray} This concludes the proof of the Lemma.
\end{proof}
\subsection{\bf Equivariant Power Series.}\label{appendix-power-series} (Compare with \S2 in \cite{Burns-Greither}.)  Let $G$ be an arbitrary finite abelian group, $p$ a prime, and $\co$ a finite, integral extension of $\zp$ which contains the values of all characters $\chi\in\widehat G(\cp)$. We let $\pi$ be a uniformizer of $\co$. We identify the set of $\zp$--algebra morphisms ${\rm Hom}(\zp[G], \co)$ with the set
$\chi\in \widehat G(\cp)$ of $\cp$--valued characters of $G$ in the obvious manner. We let $I$ denote a radical ideal of $\zp[G]$ of pure codimension $1$.
It is easily seen that this means that $I=\cap_{\chi\in\mathcal F}\ker(\chi)$, for
some set $\mathcal F\subseteq {\rm Hom}(\zp[G], \co)$. In what follows, we let $\fa:=\zp[G]/I$.
Obviously, $\mathcal F$ coincides with the set of $\zp$--algebra morphisms ${\rm Hom}(\fa, \co)$.
Also, we have an injective $\zp$--algebra morphism
$$\fa\longrightarrow \oplus_{\chi}\, \co,\qquad x\longrightarrow (\chi(x))_{\chi\in {\rm Hom}(\fa, \co)}.$$
For every $\chi\in{\rm Hom}(\fa, \co)$, we abuse notation once again and
let $\chi$ also denote the unique $\fa[[t]]$--algebra morphism
$$\chi: \fa[[t]]\longrightarrow \co[[t]],$$ which sends $x\to\chi(x)$, for all $x\in \fa.$
\begin{definition}\label{mu-power-series}
\begin{enumerate}
\item The $\mu$--invariant $\mu(f)$ of a power series $f\in \co[[t]]$
is the largest exponent $r\in\Bbb Z_{\geq 0}$, such that $f\in\pi^r\co[[t]]$.
\item A power series $F\in \fa[[t]]$ is said to have $\mu$--invariant equal to $0$ (and we write $\mu(F)=0$) if
$$\mu(\chi(F))=0, \qquad\text{ for all }\chi\in{\rm Hom}(\fa, \co).$$
\item A polynomial $F\in \fa[t]$ is said to be Weierstrass if $\chi(F)$ is a Weierstrass polynomial in $\co[t]$
(i.e. $\chi(F)$ is monic and all its non--leading coefficients are divisible by $\pi$), for all $\chi$ as above.
\end{enumerate}
\end{definition}
\smallskip

\begin{remark} The rings $\fa$ considered above are the most general reduced quotients of $\zp[G]$ of pure Krull dimension $1$.
It is easy to see that any ring $\fa$ as above is {\rm admissible}, in the sense of \cite{Burns-Greither}, \S2.
Also, it is easy to prove that our definition of power series $F\in\fa[[t]]$ of $\mu$--invariant equal to $0$ is equivalent with
the definition in loc.cit., for all rings $\fa$ as above.
\end{remark}

%Note that if we let $G=\Delta\times P$, where $P$ is the $p$--Sylow subgroup of $G$ and $\Delta$ its complement in $G$, we can
%write the injective $\zp[P][[t]]$--algebra morphism
%$$\zp[G][[t]]\hookrightarrow \bigoplus_{\psi\in\widehat\Delta(\cp)}\co[P][[t]], \qquad F\to \sum_{\psi\in\widehat\Delta(\cp)}\psi(F),$$
%which send $\delta\to  (\psi(\delta))_{\psi}$, for all $\delta\in\Delta$.
%It is easy to see that we have an equivalence
%$$\mu(F)=0 \quad\Leftrightarrow\quad \psi(F)\not\in\mathfrak m_{\co[P]}[[t]], \quad\forall\,\psi\in\widehat\Delta(\cp),$$
%where $\mathfrak m_{\co[P]}$ is the maximal ideal in the local ring $\co[P]$, for all $\psi$ as above.
%Also, note that a polynomial $F\in \zp[G][t]$ is Weierstrass if and only if $\psi(F)$ is a monic polynomial
%in $\co[P][t]$ with all its non-leading coefficients in $\mathfrak m_{\co[P]}$, for all $\psi$.  These observations
%make Definitions \ref{mu-power-series} above equivalent to those in \S2 of \cite{Burns-Greither}.

%It is also useful to note that if $\mu(F)=0$, for some $F\in\zp[G][[t]]$, then $F$ is not a zero--divisor in $\zp[G][[t]]$. Indeed, this is a consequence of
%$\chi(F)\ne 0$  and therefore $\chi(F)$ is not a zero--divisor in the integral domain $\co[[t]]$, for all $\chi\in\widehat G(\cp)$.
%\end{remark}
In the following, if $R$ is a commutative ring with $1$, and $f, g\in R$, we write ``$f\sim g$ in $R$'' to mean that $f$ and $g$ are associated in divisibility in $R$,
i.e. there exists a unit $u\in R^\times$, such that $f=u\cdot g$.
\begin{lemma}\label{associated-in-divisibility} Assume that $F, \Theta\in \fa[[t]]$, such that $\mu(F)=\mu(\Theta)=0$ and
$$\chi(F)\sim\chi(\Theta) \text{ in } \co[[t]], \qquad \text{ for all }\,\chi\in{\rm Hom}(\fa, \co).$$
Then, we have $F\sim \Theta$ in $\fa[[t]]$.
\end{lemma}
\begin{proof} Since $\mu(F)=\mu(\Theta)=0$, Proposition 2.1 in \cite{Burns-Greither} (the equivariant Weierstrass preparation theorem) shows that there exist unique Weierstrass polynomials $f, g\in \fa[t]$ and units $u, v\in \fa[[t]]^\times$, such that
$$F=u\cdot f, \quad \Theta= v\cdot \theta.$$
This implies that we have Weierstrass decompositions in $\co[[t]]$
$$\chi(F)=\chi(u)\cdot\chi(f), \quad \chi(\Theta)= \chi(v)\cdot\chi(\theta),$$
with $\chi(u), \chi(v)\in \co[[t]]^\times$ and $\chi(f), \chi(\theta)$ Weierstrass polynomials in $\co[t]$, for all $\chi\in{\rm Hom}(\fa, \co)$. However, since the Weierstrass decomposition is unique
in $\co[[t]]$ (according to the classical Weierstrass preparation theorem), our hypotheses combined with the above equalities imply  that
$$\chi(f)=\chi(\theta), \qquad \text{ for all }\, \chi\in{\rm Hom}(\fa, \co).$$
Consequently, we have $f=\theta$ and $F=uv^{-1}\cdot \Theta$. Therefore, $F\sim\Theta$ in $\fa[[t]]$.
% Now, although a priori we only know that $uv^{-1}\in \co[G][[t]]^\times$,
% since $\Theta$ is not a zero--divisor in $\zp[G][[t]]$ (see the Remark above), it is easy to see that the equality $F=uv^{-1}\cdot \Theta$ implies in fact that $uv^{-1}\in\zp[G][[t]]^\times.$
% Consequently, we have $F\sim \Theta$ in $\zp[G][[t]].$
\end{proof}

Now, let us assume that $p$ is odd and that $G$ has an element $j$ of order $2$. Let $\zp[G]^-:=\zp[G]/(1+j)$ and call
a character $\chi\in\widehat G(\cp)$ odd if $\chi(j)=-1$. Clearly, $I:=(1+j)=\cap_{\chi}\ker(\chi)$, where
$\chi$ runs through the odd characters of $G$. Therefore, $\fa:=\zp[G]^-$ is a reduced quotient of $\zp[G]$ of pure Krull dimension $1$. Therefore, the following is a direct consequence of the above Lemma.

%Note that for all odd $\chi$, the morphism $\chi:\zp[G][[t]]\to \co[[t]]$ considered above factors
%through $\chi: \zp[G]^-[[t]]\to \co[[t]]$. (Note the abuse of notation !)

\begin{corollary}\label{corollary-association} Let $F, \Theta\in\zp[G]^-[[t]]$, such that
$$\mu(\chi(F))=\mu(\chi(\Theta))=0, \qquad\chi(F)\sim\chi(\Theta) \text{ in $\co[[t]]$,}$$
for all odd $\chi\in\widehat G(\cp)$.
Then, we have $F\sim\Theta$ in $\zp[G]^-[[t]]$.
\end{corollary}
%\begin{proof}
%View $F$ and $\Theta$ as elements in $1/2(1-j)\zp[G][[t]]\subseteq\zp[G][[t]]$ via the reduction
%modulo $(1+j)\zp[G]$ ring isomorphism $1/2(1-j)\zp[G][[t]]\simeq\zp[G]^-[[t]]$. Now, apply the above Lemma to the power series
%$\widetilde F:= F+1/2(1+j)$ and $\widetilde\Theta=\Theta +1/2(1+j)$ in $\zp[G][[t]]$ to derive the desired result.
%\end{proof}

\bibliographystyle{plain}
\bibliography{emc-bibliography}
\end{document}